\documentclass[dvipdfmx,12pt]{amsart}

\usepackage{amsmath,amsthm,amssymb,mathdots,shuffle}
\usepackage{setspace}
\allowdisplaybreaks
\newtheorem{theorem}{Theorem}[section]
\newtheorem{proposition}[theorem]{Proposition}
\newtheorem{definition}[theorem]{Definition}
\newtheorem{lemma}[theorem]{Lemma}
\newtheorem{corollary}[theorem]{Corollary}
\newtheorem{conjecture}[theorem]{Conjecture}

\theoremstyle{remark}
\newtheorem{remark}[theorem]{Remark}

\numberwithin{equation}{section}
\newcommand{\R}{\mathbb R}
\newcommand{\Z}{{\mathbb Z}}

\newcommand{\C}{{\mathbb C}}
\newcommand{\Q}{{\mathbb Q}}

\DeclareMathOperator{\wt}{wt}
\DeclareMathOperator{\dep}{dep}
\newcommand{\kk}{\boldsymbol{k}}
\newcommand{\pp}{\boldsymbol{p}}

\newcommand{\stirlingtwo}[2]{\genfrac{\{}{\}}{0pt}{}{#1}{#2}}
\newcommand{\bs}{\boldsymbol{s}}
\newcommand{\bl}{\boldsymbol{l}}
\newcommand{\qbinom}[2]{\genfrac{[}{]}{0pt}{}{#1}{#2}}

\title[Supercongruences of multiple harmonic $q$-sums]
{Supercongruences of multiple harmonic $q$-sums and generalized finite/symmetric multiple zeta values}
\author[Y.~Takeyama]{Yoshihiro Takeyama}
\address[Y.~Takeyama]{Department of Mathematics\\
	Faculty of Pure and Applied Sciences\\
	University of Tsukuba\\
	Tsukuba, Ibaraki 305-8571\\
	Japan}
\email{takeyama@math.tsukuba.ac.jp}

\author[K.~Tasaka]{Koji Tasaka}
\address[K.~Tasaka]{Department of Information Science and Technology\\
	Aichi Prefectural University\\
	Nagakute, Aichi 480-1198\\
	Japan
}
\email{tasaka@ist.aichi-pu.ac.jp}



\begin{document}

\begin{abstract}
The Kaneko--Zagier conjecture describes a correspondence between finite multiple zeta values and symmetric multiple zeta values.
Its refined version has been established by Jarossay, Rosen and Ono--Seki--Yamamoto.
In this paper, we explicate these conjectures through studies of multiple harmonic $q$-sums.
We show that the (generalized) finite/symmetric multiple zeta value are obtained
%%%%
%from certain realizations of a sequence of
by taking an algebraic/analytic limit of
%%%%
multiple harmonic $q$-sums.
As applications, new proofs of reversal, duality and cyclic sum formulas for
the generalized finite/symmetric multiple zeta values are given.
\end{abstract}

\maketitle

%\tableofcontents

%%%%%%%%%%%%%%%%%%%%%%%%%%%%%%%%%%
\section{Introduction}

In 2013, Kaneko and Zagier \cite{KanekoZagier} posed a conjecture  (called the Kaneko--Zagier conjecture) about a correspondence between the finite multiple zeta value $\zeta_{\mathcal{A}}(\kk)\in \mathcal{A}=\big(\prod_{p}\Z/p\Z\big) \big/ \big(\bigoplus_p \Z/p\Z\big)$, where $p$ runs over all primes, and the symmetric multiple zeta value $\zeta_{\mathcal{S}}(\kk)\in \mathcal{Z}/\pi^2\mathcal{Z}$, where $\mathcal{Z}$ denotes the $\Q$-vector space spanned by all multiple zeta values.
Recently, its refined version has been established by Jarossay  \cite[Conjecture 5.3.2]{Jarossay19}, Rosen \cite[Conjecture 2.3]{Rosen19} and Ono--Seki--Yamamoto \cite[Conjecture 4.3]{OnoSekiYamamoto}.
In this paper, we explicate these conjectures through studies of multiple harmonic $q$-sums.
%We also deal with a refined ($\pp$-adic vs. $t$-adic) version of the Kaneko--Zagier conjecture from $q$-analogue view points.
%The multiple harmonic $q$-sum is introduced by Bradley \cite[Definition 4]{Bradley05_2} as a $q$-analogue of the multiple harmonic sum $H_m(\kk)$ defined in \eqref{eq:def_mhs};
%for all index $\kk$, it holds that $H_{m}(\kk;1) =H_m(\kk) $.

In \cite{BTT18},
the authors developed a new approach to simultaneously give relations among
both $\zeta_{\mathcal{A}}(\kk)$'s and $\zeta_{\mathcal{S}}(\kk)$'s over $\Q$,
and partially support the Kaneko--Zagier conjecture.
An epoch-making invention is that both $\zeta_{\mathcal{A}}(\kk)$ and $\zeta_{\mathcal{S}}(\kk)$ are obtained from certain limiting values of a multiple harmonic $q$-sum $H_{p-1}(\kk;q)$ (defined in \eqref{eq:MHqS}) at $q$ being primitive $p$-th roots of unity.
In the current paper, this result is recast by introducing a {$\mathcal{Q}$-multiple zeta value} $\zeta_{\mathcal{Q}}(\kk)$ made from the sequence $(H_{p-1}(\kk;q)\mod [p])_p$, where $[p]=(1-q^p)/(1-q)$ and $p$ runs over all primes.
Constructing two algebra maps $\phi_{\mathcal{A}}$ and $\phi_{\mathcal{S}}$, we prove in Theorem \ref{thm:AS_from_q} that
\begin{equation}\label{eq:recast_BTT}
\phi_{\mathcal{A}}\big( \zeta_{\mathcal{Q}}(\kk)\big)=\zeta_{\mathcal{A}}(\kk) \quad \mbox{and} \quad \phi_{\mathcal{S}}\big( \zeta_{\mathcal{Q}}(\kk)\big)\equiv\zeta_{\mathcal{S}}(\kk)  \mod \pi i\mathcal{Z}[\pi i].
\end{equation}
As an application, from a family of relations among $\zeta_{\mathcal{Q}}(\kk)$'s, one obtains the corresponding family of relations among both $\zeta_{\mathcal{A}}(\kk)$'s and $\zeta_{\mathcal{S}}(\kk)$'s in the same form.
These results, together with the statement of the Kaneko--Zagier conjecture, are summarized in \S2.

In \S3, we explicate a refined version of the Kaneko--Zagier conjecture, which is initiated by Hirose, Rosen and Jarossay, independently.
It describes a conjectural relationship between
the $\widehat{\mathcal{A}}$-multiple zeta value $\zeta_{\widehat{\mathcal{A}}}(\kk)$
and the $\widehat{\mathcal{S}}$-multiple zeta value $\zeta_{\widehat{\mathcal{S}}}(\kk)$, lying in the $\Q$-algebras $\widehat{\mathcal{A}}$ and $\big(\mathcal{Z}[\pi i]/(\pi i)\big)[[t]]$, respectively.
Although there are natural surjections $\widehat{\mathcal{A}}\rightarrow \mathcal{A}$ and $\big(\mathcal{Z}[\pi i]/(\pi i)\big)[[t]]\rightarrow \mathcal{Z}[\pi i]/(\pi i)$, which send $\zeta_{\widehat{\mathcal{A}}}(\kk)\mapsto \zeta_{\mathcal{A}}(\kk)$ and $\zeta_{\widehat{\mathcal{S}}}(\kk)\mapsto \zeta_{\mathcal{S}}(\kk)$, respectively,
the relationship between
%%%%%%%
%the refined version and the Kaneko--Zagier conjecture
the Kaneko--Zagier conjecture and its refined version
%%%%%%%
is not clearly written in the literature.
For future reference, we prove in Proposition \ref{prop:implication} that the refined version implies the Kaneko--Zagier conjecture
%%%%%%%
%(under the assumption that Rosen's lifting conjecture holds).
under the assumption that Rosen's lifting conjecture holds.
%%%%%%%

Our new story starts from \S4, where we introduce a ``$q$-analogue" $\widehat{\mathcal{Q}}$ of the $\Q$-algebra $\widehat{\mathcal{A}}$.
In the $\Q$-algebra $\widehat{\mathcal{Q}}$, for each index $\kk$ we define a unified object $\zeta_{\widehat{\mathcal{Q}}}(\kk)$, called the {$\widehat{\mathcal{Q}}$-multiple zeta value} (Definition \ref{def:hatQ-MZV}).
As a generalization of \eqref{eq:recast_BTT}, we show in Theorems \ref{thm:hatA_from_q} and \ref{thm:hatS_from_q} that
\begin{equation}\label{eq:generalization_BTT}
\phi_{\widehat{\mathcal{A}}}\big( \zeta_{\widehat{\mathcal{Q}}}(\kk)\big)=\zeta_{\widehat{\mathcal{A}}}(\kk)\quad \mbox{and} \quad \phi_{\widehat{\mathcal{S}}}\big( \zeta_{\widehat{\mathcal{Q}}}(\kk)\big)=\zeta_{\widehat{\mathcal{S}}}(\kk)  \mod \pi i,
\end{equation}
which are our main results of this paper.
We will also compute images of other variants
%%%%%%%%%%
%(a star version of and a `conjugate model') of multiple harmonic $q$-sums under the maps $\phi_{\widehat{\mathcal{A}}}$ and $\phi_{\widehat{\mathcal{S}}}$.
(a star version and a `conjugate model') of multiple harmonic $q$-sums under the maps $\phi_{\widehat{\mathcal{A}}}$ and $\phi_{\widehat{\mathcal{S}}}$.
%%%%%%%%%
%Our proof of the second equation in \eqref{eq:generalization_BTT} can be found in Appendix \ref{app:proof_of_main_result}.
Applying these results to relations among $\widehat{\mathcal{Q}}$-multiple zeta values, we obtain relations among both $\zeta_{\mathcal{A}}(\kk)$'s and $\zeta_{\mathcal{S}}(\kk)$'s in the same form, which also support the refined version of the Kaneko--Zagier conjecture.

As examples of relations, we extend the reversal, duality and cyclic sum formulas for both the $\widehat{\mathcal{A}}$-multiple zeta value and
the $\widehat{\mathcal{S}}$-multiple zeta value obtained in \cite{HiroseMuraharaOno,Jarossay19,Kawasaki,OnoSekiYamamoto,Rosen15,Seki19} to corresponding formulas for $\widehat{\mathcal{Q}}$-multiple zeta values.
This will be the subject in \S\ref{sec:rel}.

\S\ref{sec:observations} is devoted to computing dimensions of the $\Q$-vectors space spanned by $\mathcal{Q}$-multiple zeta values, based on experimental works.
We indicate that all $\Q$-linear relations among $\mathcal{F}$-multiple zeta values ($\mathcal{F}\in\{\mathcal{A},\mathcal{S}\}$) of weight up to 5 may be obtained from relations among variants of $\mathcal{Q}$-multiple zeta values.
We also study relations of $\zeta_{\mathcal{Q}_2}(\kk)=(H_{p-1}(\kk;q)\mod [p]^2)_p$.

%%%%%%%%%%%%%%%%%%%%%%%%%%%%%%%%%%
\subsection*{Notation}
In this paper, we often use the following notation.
We call a finite ordered list $\kk=(k_1,\ldots,k_d)$ of positive integers an \emph{index}
and write $\wt(\kk)=k_1+\cdots+k_d$ (weight) and $\dep(\kk)=d$ (depth).
An index $\kk=(k_{1}, \ldots , k_{d})$ is called \textit{admissible} if $k_{1} \ge 2$.
We allow the empty index $\varnothing$ to be the unique index $\kk$ such that $\wt(\kk)=\dep(\kk)=0$.
For any function $F$ on indices, set $F(\varnothing)=1$.
For tuples $\kk =(k_1,\ldots,k_d)$ and $\bl=(l_1,\ldots,l_d)$, we write
\begin{equation}\label{eq:ind_notation}
\begin{aligned}
&\kk+\bl:=(k_1+l_1,\ldots,k_d+l_d),\quad b\binom{\kk}{\bl} := \prod_{j=1}^d \binom{k_j+l_j-1}{l_j},\\
&\overline{\kk}:=(k_d,\ldots,k_1),\quad \kk_a:=(k_1,\ldots,k_a),\quad \kk^{a} :=(k_{a+1},\ldots,k_d)\quad (0\le a\le d),
\end{aligned}
\end{equation}
where $\kk_{0}=\kk^{d}=\varnothing$.

%%%%%%%%%%%%%%%%%%%%%%%%%%%%%%%%%%%%%%%%%%
\subsection*{Acknowledgments}
The authors would like to thank Masataka Ono, Shin-ichiro Seki for valuable comments.
They are also very grateful to Julian Rosen for comments on relations among $\mathcal{Q}$-multiple zeta values that motivated them to add \S\ref{subsec:variants}.
This work is partially supported by
JSPS KAKENHI Grant Number 18K03233, 18H01110 and 20K14294.

%%%%%%%%%%%%%%%%%%%%%%%%%%%%%%%%%%
%%%%%%%%%%%%%%%%%%%%%%%%%%%%%%%%%%
%%%%%%%%%%%%%%%%%%%%%%%%%%%%%%%%%%
%%%%%%%%%%%%%%%%%%%%%%%%%%%%%%%%%%
\section{Kaneko--Zagier conjecture and its $q$-analogue view points}

In this section, we first recall the Kaneko--Zagier conjecture on a correspondence between finite multiple zeta values ($\mathcal{A}$-MZVs) and symmetric multiple zeta values ($\mathcal{S}$-MZVs).
Then, we recast the work of \cite{BTT18}.% by introducing $\mathcal{Q}$-MZVs. %, where the authors  relationships between $\mathcal{A}$-MZVs and $\mathcal{S}$-MZVs through studies of a $q$-analogue of multiple harmonic sums.

%%%%%%%%%%%%%%%%%%%%%%%%%%%%%%%%%%
\subsection{MZV}
The \emph{multiple zeta value} (abbreviated by MZV) is defined for an admissible index $\kk=(k_1,\ldots,k_d)$ by
\begin{align*}
\zeta(\kk)=\sum_{m_1>\cdots>m_d>0} \frac{1}{m_1^{k_1}\cdots m_d^{k_d}} \in \R.
\label{eq:def-MZV}
\end{align*}
Let $\mathcal{Z}_k$ denote the $\Q$-vector space spanned by all MZVs of weight $k$.
The sum
\[\mathcal{Z}=\sum_{k\ge0} \mathcal{Z}_k\]
forms a $\Q$-algebra.
Zagier \cite{Zagier94} observed that the equality $\dim_{\Q} \mathcal{Z}_k \stackrel{?}{=} d_k$ holds for all $k\ge0$, where $d_k$ is given by $\sum_{k\ge0}d_kx^k=1/(1-x^2-x^3)$.
Goncharov \cite{Goncharov01} and Terasoma \cite{Terasoma02} showed independently the inequality $\dim_{\Q} \mathcal{Z}_k \le d_k$ for all $k$.
This obeservation/result tells us that MZVs satisfy numerous linear relations over $\Q$.
The first example of relations is $\zeta(3)=\zeta(2,1)$, which was discovered by Euler \cite{Euler76}.
Since 1990's, explicit families of relations, such as the sum formula \cite{Granville97}, the cyclic sum formula \cite{HoffmanOhno03}, the duality formula \cite{Zagier94}, %Ohno's relation \cite{Ohno99}, the associator relation \cite{Drinfeld90},
the regularized double shuffle relation \cite{IharaKanekoZagier06,Racinet02} and so on,  have been intensively studied by various approaches and also applied in several branches of mathematics and physics.

%%%%%%%%%%%%%%%%%%%%%%%%%%%%%%%%%%
\subsection*{Regularizations}
As a basic of MZVs, we quickly recall the shuffle and the stuffle regularization of MZVs, mainly following \cite[\S2]{IharaKanekoZagier06}.

%Let us recall the algebraic setup \cite{Hoffman97}.
Let $\mathfrak{h}=\Q\langle x_0,x_1\rangle$ be the non-commutative polynomial algebra over $\Q$ and
set $\mathfrak{h}^1=\Q+\mathfrak{h}x_1$.
Define the \emph{shuffle product} $\shuffle : \mathfrak{h}^{\otimes 2} \rightarrow \mathfrak{h}$ inductively by
\begin{equation*}
uw\shuffle vw'=u(w \shuffle vw')+v(uw \shuffle w')
\end{equation*}
for $w,w' \in \mathfrak{h}$ and $u, v \in \{x_0, x_1\}$,
with the initial condition $w \shuffle 1=w=1\shuffle w$.
Equipped with the shuffle product, the vector space $\mathfrak{h}$ forms a commutative $\Q$-algebra
and $\mathfrak{h}^1$ is its $\Q$-subalgebra.
We write $\mathfrak{h}_\shuffle$ and $\mathfrak{h}^1_\shuffle$ for the commutative $\Q$-algebras
with the shuffle product.
It is known that $\mathfrak{h}_\shuffle\cong \mathfrak{h}^1_\shuffle[x_0]$.
For an index $\kk=(k_1,\ldots,k_d)$ we write
\[ x_{\kk}=x_0^{k_1-1}x_1\cdots x_0^{k_d-1}x_1,\]
which forms a $\Q$-linear basis of $\mathfrak{h}^1$.

Let $\mathfrak{H}^1=\Q\langle y_k\mid k\ge1\rangle$ be the non-commutative polynomial algebra over $\Q$.
Define the \emph{stuffle product}
$\ast : \mathfrak{H}^1\otimes \mathfrak{H}^1 \rightarrow \mathfrak{H}^1$
inductively by
\begin{equation*}
y_kw\ast y_lw'=y_k(w \ast y_lw')+y_l(y_kw \ast w')+y_{k+l}(w\ast w')
\end{equation*}
for $w,w' \in \mathfrak{H}^1$ and $k,l\ge1$, with the initial condition $w \ast 1=w=1\ast w$.
Equipped with the stuffle product,
the vector space $\mathfrak{H}^1$ forms a commutative $\Q$-algebra and
we denote it by $\mathfrak{H}^1_\ast$.
For an index $\kk=(k_1,\ldots,k_d)$ we also write
\[ y_{\kk}=y_{k_1}\cdots y_{k_d},\]
which forms a $\Q$-linear basis of $\mathfrak{H}^1$.
Since we use Racinet's formulation of the regularization theorem later (see \eqref{eq:reg_theorem}), we distinguish $\mathfrak{H}^1$ ($y$-words) from $\mathfrak{h}^1$ ($x$-words), while the $\Q$-linear map $\mathfrak{H}^1\rightarrow \mathfrak{h}^1,\ y_{\kk}\mapsto x_{\kk}$ is
an isomorphism of $\Q$-vector spaces.

There are algebra homomorphisms
\begin{equation*}
Z^\shuffle: \mathfrak{h}_\shuffle^1 \longrightarrow \R[T] \quad \mbox{and} \quad Z^\ast : \mathfrak{H}_\ast^1 \longrightarrow \R[T]
\end{equation*}
such that $Z^\shuffle (x_1)=T$, $Z^{\ast} (y_1)=T$ and
\[Z^\shuffle(x_{\kk}) =Z^\ast(y_{\kk})= \zeta(\kk)\]
for all admissible index $\kk$ (see \cite[Proposition 1]{IharaKanekoZagier06}).
For an index $\kk$ we write
\begin{equation}\label{eq:def_reg_MZV}
\zeta^\shuffle(\kk;T)=Z^\shuffle(x_{\kk})
\quad \mbox{and} \quad
\zeta^\ast(\kk;T)=Z^\ast(y_{\kk}),
\end{equation}
which are called the shuffle and the stuffle \emph{regularized MZV}, respectively.
These are elements in the polynomial ring $\mathcal{Z}[T]$ over the $\Q$-algebra $\mathcal{Z}$.

%%%%%%%%%%%%%%%%%%%%%%%%%%%%%%%%%%
\subsection{Multiple harmonic sum modulo $p$}

For a positive integer $m$ and an index $\kk=(k_1,\ldots,k_d)$, we define the \emph{multiple harmonic sum} $H_m(\kk)\in\Q$ by
\begin{equation*}\label{eq:def_mhs}
H_m(\kk)=H_m(k_1,\ldots,k_d)=\sum_{m\ge m_1>\cdots>m_d>0} \frac{1}{m_1^{k_1}\cdots m_d^{k_d}}.
\end{equation*}
We understand $H_m(\kk)=0$ if $\dep(\kk)>m$.
Zhao \cite{Zhao08} and Hoffman \cite{Hoffman15} independently discovered mod $p$ congruence relations among $H_{p-1}(\kk)$'s, which holds for all large primes $p$.
%A prototype of relations is Wolstenholme's theorem \cite{Wolstenholme62}; the congruence
%\begin{equation}\label{eq:wolstenholme}
%H_{p-1}(1)\equiv 0 \mod{p^2}
%\end{equation}
%holds for all primes $p\ge5$.
A prototypical example is Hoffman's duality (\cite[Theorem 4.6]{Hoffman15}); for each index $\kk$ and all primes $p$, it holds that
\begin{equation}\label{eq:hoffman's_dual}
H_{p-1}^\star(\kk)+H_{p-1}^\star(\kk^\vee) \equiv 0 \mod{p},
\end{equation}
where we set
\[H_m^\star(k_1,\ldots,k_d)=\sum_{m\ge m_1\ge\cdots\ge m_d>0} \frac{1}{m_1^{k_1}\cdots m_d^{k_d}}\]
and the index $\kk^\vee$ is called Hoffman's dual index
%%%%%%%
%writing each component
obtained by writing each component
%%%%%%%
$k_i$ as a sum of 1 and then interchanging commas `,' and plus signs `+'.
For example, $(3,1)^\vee=(1+1+1,1)^\vee=(1,1,1+1)=(1,1,2)$.

%%%%%%%%%%%%%%%%%%%%%%%%%%%%%%%%%%
\subsection{$\mathcal{A}$-MZV}
A study of finite multiple zeta values has been initiated by Kaneko--Zagier \cite{KanekoZagier}.
A crucial feature is that mod $p$ congruence relations among $H_{p-1}(\kk)$'s being independent from choices of large primes $p$, such as \eqref{eq:hoffman's_dual}, turn out to be $\Q$-linear relations among finite multiple zeta values, which we now define in the ring
\begin{align*}
\mathcal{A} = \left( \prod_{p:{\rm prime}} \Z/p\Z \right) \big/ \left( \bigoplus_{p:{\rm prime}} \Z/p\Z\right).
\end{align*}
%We first review a few basics on $\mathcal{A}$, and then, define the finite multiple zeta value.

The above ring $\mathcal{A}$ is due to Kontsevich \cite[\S2.2]{Kontsevich09} and its ring structure is given by the component-wise addition and multiplication.
We denote by $(a_p)_p$ an element in $\mathcal{A}$, where $p$ runs over all primes and $a_p\in \Z/p\Z$.
An element $(a_p)_p$ in $\mathcal{A}$ equals the zero element $0\in \mathcal{A}$, if and only if $a_p=0$ for all large enough $p$.
For simplicity of notation, we will view $(c_p\mod p)_p $ as an element in $\mathcal{A}$ for each sequence $\{c_n \}_{n\ge0} \subset\Q$ such that the denominator of $c_p$ may be divisible by $p$ for finitely many primes $p$ (to be precise, for such $p$ we should replace $c_p$ with any elements in $\Z/p\Z$).
Under this notation, the map $\Q\rightarrow\mathcal{A}, c\mapsto(c\mod p)_p$ is well-defined and it induces scalar multiplication of $\Q$ on $\mathcal{A}$.
Thus, $\mathcal{A}$ forms a $\Q$-algebra.

\begin{definition}
The \emph{finite multiple zeta value} (call it $\mathcal{A}$-MZV for short) is defined for each index $\kk$ by
\[ \zeta_{\mathcal{A}}(\kk)= \left( H_{p-1}(\kk) \mod{p} \right)_p \in \mathcal{A}.\]
Its star version (called $\mathcal{A}$-MZSV) is denoted by $\zeta_{\mathcal{A}}^\star(\kk)$, replacing $H_{p-1}(\kk)$ with $H_{p-1}^\star(\kk)$.
\end{definition}

Let $\mathcal{Z}_k^{\mathcal{A}}$ be the $\Q$-vector subspace of $\mathcal{A}$ spanned by all $\mathcal{A}$-MZVs of weight $k$.
Set
\[ \mathcal{Z}^{\mathcal{A}} = \sum_{k\ge0} \mathcal{Z}_k^{\mathcal{A}}.\]
It follows that the $\Q$-linear map $\mathfrak{H}^1_\ast\rightarrow \mathcal{Z}^{\mathcal{A}}, \ y_{\kk}\mapsto \zeta_{\mathcal{A}}(\kk)$ forms a $\Q$-algebra map.
Zagier numerically observed that $\dim_\Q \mathcal{Z}^{\mathcal{A}}_k\stackrel{?}{=} d_k-d_{k-2}$ for $k\ge2$.
Hence, $\mathcal{A}$-MZVs also satisfy many relations over $\Q$.
As an example of relations, since \eqref{eq:hoffman's_dual} holds for all sufficiently large primes, we obtain
\begin{equation}\label{eq:hoffman's_dual_finite}
\zeta_{\mathcal{A}}^\star(\kk)+\zeta_{\mathcal{A}}^\star(\kk^\vee)=0.
\end{equation}
The above relation \eqref{eq:hoffman's_dual_finite} is called \emph{Hoffman's duality} for the $\mathcal{A}$-MZSV.

%%%%%%%%%%%%%%%%%%%%%%%%%%%%%%%%%%
\subsection{$\mathcal{S}$-MZV}
With a hint by Kontsevich (see \cite[\S9]{Kaneko19}), Kaneko--Zagier introduced a real counter part of $\mathcal{A}$-MZVs.
Recall the notation \eqref{eq:ind_notation} and the regularized MZV \eqref{eq:def_reg_MZV}.
For an index $\kk$ and $\bullet\in\{\ast,\shuffle\}$, we define
\begin{align*} \zeta_{\mathcal{S}}^\bullet(\kk)&=
\sum_{a=0}^{\dep(\kk)} (-1)^{\wt(\kk_a)}
\zeta^\bullet(\overline{\kk_a};T)\zeta^\bullet(\kk^a;T) .
\end{align*}
It is shown in \cite{KanekoZagier} (see also \cite{Kaneko19}) that the right side does not depend on $T$ and that
\[ \zeta^{\ast}_{\mathcal{S}}(\boldsymbol{k})\equiv \zeta^{\shuffle}_{\mathcal{S}}(\boldsymbol{k})\mod \pi^2\mathcal{Z} \]
holds for any index $\kk$.
%where $\mathcal{Z}=\bigoplus_{k\ge0}\mathcal{Z}_k$ is the MZV algebra.
The following definition is therefore independent of the choice of regularizations. % $\bullet\in\{\ast,\shuffle\}$.

\begin{definition}
For each index $\kk$, we define the \emph{symmetric multiple zeta value} $\zeta_{\mathcal{S}}(\boldsymbol{k}) $ (call it $\mathcal{S}$-MZV for short) by
\[\zeta_{\mathcal{S}}(\boldsymbol{k}) = \zeta^\ast_{\mathcal{S}}(\boldsymbol{k}) \mod \pi^2\mathcal{Z}.\]
Its star version ($\mathcal{S}$-MZSV) is defined by
\[ \zeta^\star_{\mathcal{S}}(k_1,\ldots,k_d) =
\sum_{\substack{\footnotesize \square \ \mbox{is either a comma `,'}\\ \footnotesize \mbox{or a plus `$+$'}}}
\zeta_{\mathcal{S}}(k_1 \square \cdots \square k_d).\]
\end{definition}
Note that the above expression of $\mathcal{S}$-MZSV naturally arises from the standard decomposition of $H_m^\star(\kk)$ in terms of $H_m(\kk)$.
For example, we have
\begin{align*}
H_m^\star(k_1,k_2)&=\sum_{m\ge m_1\ge m_2\ge 1}\frac{1}{m_1^{k_1}m_2^{k_2}}=\bigg( \sum_{m\ge m_1> m_2\ge 1}+ \sum_{m\ge m_1= m_2\ge 1}\bigg)\frac{1}{m_1^{k_1}m_2^{k_2}} \\
&=H_m(k_1,k_2)+H_m(k_1+k_2),
\end{align*}
which corresponds to the definition $\zeta^\star_{\mathcal{S}}(k_1,k_2)=\zeta_{\mathcal{S}}(k_1,k_2)+\zeta_{\mathcal{S}}(k_1+k_2)$.
%Hence,

By definition, $\mathcal{S}$-MZV is a $\Q$-linear combination of MZVs modulo $\pi^2\mathcal{Z}$.
The opposite statement, namely, the $\Q$-vector space
\[\overline{\mathcal{Z}}=\mathcal{Z}/\pi^2\mathcal{Z}\]
is generated by $\mathcal{S}$-MZVs, was shown by Yasuda \cite{Yasuda16}.

%\begin{lemma}\label{lem:Yasuda}
%Every MZV of weight $k$ is a $\Q$-linear combination of $\zeta^\ast_{\mathcal{S}}(\kk)$'s (resp.~$\zeta^\shuffle_{\mathcal{S}}(\kk)$'s) of weight $k$.
%\end{lemma}

%%%%%%%%%%%%%%%%%%%%%%%%%%%%%%%%%%
\subsection{$\mathcal{A}$-MZV vs. $\mathcal{S}$-MZV}
This paper deals with the following variant of the Kaneko--Zagier conjecture \cite{KanekoZagier}.

\begin{conjecture}\label{conj:Kaneko-Zagier}
Let $k$ be a positive integer.
For rational numbers $c_{\kk}$ we have
\begin{equation}\label{eq:finite_vs_symmetric}
\sum_{\wt(\kk)=k} c_{\kk} \zeta_{\mathcal{A}}(\kk) =0 \stackrel{?}{\Longleftrightarrow} \sum_{\wt(\kk)=k} c_{\kk} \zeta_{\mathcal{S}}(\kk)= 0 ,
\end{equation}
where the sum $\sum_{\wt(\kk)=k}$ runs over all indices $\kk$ of weight $k$.
\end{conjecture}

Conjecture \ref{conj:Kaneko-Zagier} implies that the $\Q$-linear map sending $\zeta_{\mathcal{A}}(\kk)$ to $\zeta_{\mathcal{S}}(\kk)$ for each index $\kk$ of weight $k$ is an isomorphism from $\mathcal{Z}_k^{\mathcal{A}}$ to $ \mathcal{Z}_k/\pi^2\mathcal{Z}_{k-2}$.
As evidences, several families of relations which are satisfied by both $\mathcal{A}$-MZVs and $\mathcal{S}$-MZVs in the same shape are known.
For one, as a counterpart of \eqref{eq:hoffman's_dual_finite}, the relation
\begin{equation}\label{eq:hoffman's_dual_symmetric}
\zeta_{\mathcal{S}}^\star(\kk)+\zeta_{\mathcal{S}}^\star(\kk^\vee)=0
\end{equation}
holds for any index $\kk$.

Note that \eqref{eq:finite_vs_symmetric} is equivalent to the statement for the star version, because by inclusion-exclusion, for $\mathcal{F}\in\{\mathcal{A},\mathcal{S}\}$ we have
\[
\zeta_{\mathcal{F}}(k_1,\ldots,k_d) =
\sum_{\substack{\footnotesize \square \ \mbox{is either a comma `,'}\\ \footnotesize \mbox{or a plus `$+$'}}} (-1)^{\sharp \{\square=+\}}
\zeta_{\mathcal{F}}^\star(k_1 \square \cdots \square k_d).\]

%%%%%%%%%%%%%%%%%%%%%%%%%%%%%%%%%%
\subsection{Multiple harmonic $q$-sum}

A \emph{multiple harmonic $q$-sum} is defined for a positive integer $m$ and an index $\kk=(k_1,\ldots,k_d)$ by
\begin{equation}\label{eq:MHqS}
H_m(\kk;q)=H_m(k_1,\ldots,k_d;q)=\sum_{m\ge m_1>\cdots>m_d>0}
\prod_{a=1}^d \frac{q^{(k_a-1)m_a}}{[m_a]^{k_a}},
\end{equation}
where for a positive integer $n$, we denote the $q$-integer by $[n]=[n]_q=(1-q^{n})/(1-q)$.
Similarly, its star version is defined by
\[H_m^\star (\kk;q)=H_m^\star (k_1,\ldots,k_d;q)=\sum_{m\ge m_1\ge \cdots\ge m_d>0}
\prod_{a=1}^d \frac{q^{(k_a-1)m_a}}{[m_a]^{k_a}}.\]
Note that by using
\begin{equation}\label{eq:partial_franction_q}
\frac{q^{(k_{1}-1)m}}{[m]^{k_{1}}}
\frac{q^{(k_{2}-1)m}}{[m]^{k_{2}}}=
\frac{q^{(k_{1}+k_{2}-1)m}}{[m]^{k_{1}+k_{2}}}+(1-q)
\frac{q^{(k_{1}+k_{2}-2)m}}{[m]^{k_{1}+k_{2}-1}},
\end{equation}
the star $H_m^\star (\kk;q)$ can be written as a $\Z[1-q]$-linear combination of non-star $H_m (\kk;q)$'s
(see also \cite[Remark 2.2]{BTT18}).

Replacing numerators with $q^{s_am_a} \ (s_a\in \Z)$ in \eqref{eq:MHqS}, we obtain many other variants of multiple harmonic $q$-sums which may play a role in this study (see \S\ref{sec:observations}).
In our results, the case $s_a=1 \ (1\le a\le d)$ is necessary and hence treated later (see Definition \ref{def:hatQ-MZV}).
Remark that the model \eqref{eq:MHqS} is studied by Bradley \cite[Definition 4]{Bradley05_2} as a $q$-analogue of the multiple harmonic sum $H_m(\kk)$.

In \cite{BTT18}, the authors obtained $\mathcal{A}$-MZVs and $\mathcal{S}$-MZVs as certain limits of values of the above multiple harmonic $q$-sums at $q$ being primitive roots of unity.
Here we summarize the results.

\begin{theorem}\label{thm:BTT}{\cite[Theorems 1.1 and 1.2]{BTT18}}
\begin{enumerate}
\item[(i)] For a prime $p$, denote by $\zeta_p$ a primitive $p$-th root of unity.
Then, under the identification $\Z[\zeta_p]/(1-\zeta_p)\Z[\zeta_p]\cong \Z/p\Z$, for each index $\boldsymbol{k}$ and $\bullet\in\{\emptyset,\star\}$ we have
\[ \big( H_{p-1}^\bullet (\boldsymbol{k};\zeta_p)\mod (1-\zeta_p)\Z[\zeta_p]\big)_p = \zeta_{\mathcal{A}}^\bullet (\boldsymbol{k}).\]
\item[(ii)] For each index $\boldsymbol{k}=(k_1,\ldots,k_d)$ and $\bullet\in\{\emptyset,\star\}$, we have
\[ \lim_{m\rightarrow \infty} H_{m-1}^\bullet (\boldsymbol{k};e^{\frac{2\pi i}{m}}) = \xi^\bullet (\boldsymbol{k})\equiv \zeta_{\mathcal{S}}^\bullet (\boldsymbol{k})\mod\pi i \mathcal{Z}[\pi i],\]
where
\[\xi(\boldsymbol{k})=\sum_{a=0}^{\dep(\kk)} (-1)^{\wt(\kk_a)}
\zeta^\ast\bigg(\overline{\kk_a};\frac{\pi i}{2}\bigg)\zeta^\ast\bigg(\kk^a;-\frac{\pi i}{2}\bigg) \]
and
\[\xi^\star (\kk) =
\sum_{\substack{\footnotesize \square \ \mbox{is either a comma `,'}\\ \footnotesize \mbox{or a plus `$+$'}}}
\xi(k_1 \square \cdots \square k_d),\]
which are elements in the polynomial ring $\mathcal{Z}[\pi i]$ over $\mathcal{Z}$.
\end{enumerate}
\end{theorem}

Let us illustrate an application of Theorem \ref{thm:BTT}.
Suppose that rational numbers $c_{\kk}$ satisfy
\begin{equation}\label{eq:duality_mhqs}
\sum_{\wt(\boldsymbol{k})=k} c_{\boldsymbol{k}} H_{p-1}(\kk;\zeta_p) =0
\end{equation}
for all large primes $p>0$ and any $\zeta_p$ primitive $p$-th roots of unity.
Then by Theorem \ref{thm:BTT}, we get $\Q$-linear relations among both $\mathcal{A}$-MZVs and $\mathcal{S}$-MZVs in the same shape:
\[ \sum_{\wt(\boldsymbol{k})=k} c_{\boldsymbol{k}} \zeta_{\mathcal{A}}(\kk) =0 \quad \mbox{and} \quad \sum_{\wt(\boldsymbol{k})=k} c_{\boldsymbol{k}} \zeta_{\mathcal{S}}(\kk) =0.\]
%This gives a partial evidence of Conjecture \ref{eq:finite_vs_symmetric}.

As an example of relations of the form \eqref{eq:duality_mhqs}, a kind of duality formula
\begin{equation}\label{eq:duality_H}
 H_{p-1}^\star (\kk;\zeta_p) + (-1)^{\wt(\kk)}H_{p-1}^\star(\overline{\kk^\vee};\zeta_p)=0
 \end{equation}
is shown in \cite[Theorem 1.3]{BTT18}.
This leads to
\begin{equation}\label{eq:F-duality}
\zeta^{\star}_{\mathcal{F}}(\boldsymbol{k}) + (-1)^{\wt(\kk)}\zeta^\star_{\mathcal{F}}(\overline{\kk^\vee})=0
\end{equation}
for $\mathcal{F}\in \{\mathcal{A},\mathcal{S}\}$, which by combining the reversal relation $ \zeta^{\star}_{\mathcal{F}}(\boldsymbol{k})= (-1)^{\wt(\kk)}\zeta^\star_{\mathcal{F}}(\overline{\kk})$ reprove \eqref{eq:hoffman's_dual_finite} and \eqref{eq:hoffman's_dual_symmetric}.

%%%%%%%%%%%%%%%%%%%%%%%%%%%%%%%%%%
\subsection{$\mathcal{Q}$-MZV}
Let $p$ be a prime.
Since $[p]=\prod_{a=1}^{p-1}(q-e^{2\pi ia/p})$, a polynomial $f(q)\in \Z[q]$ satisfies $f(\zeta_p)=0$ for any $\zeta_p$ primitive $p$-th roots of unity if and only if $f(q)\equiv 0\mod [p]$ (a stronger statement can be found in Remark \ref{rem:ev_injective}).
Hence, the relation \eqref{eq:duality_mhqs} can be treated in the following ring
\begin{align*}
\mathcal{Q}=\bigg( \prod_{p:{\rm prime}}\mathbb{Z}_{(p)}[q]\big/\big([p]\big)\bigg)\bigg/ \bigg(\bigoplus_{p:{\rm prime}}\mathbb{Z}_{(p)}[q]\big/\big([p]\big)\bigg),
\end{align*}
where $\Z_{(p)}[q]$ denotes the polynomial ring in $q$ over the ring $\Z_{(p)}$ of rational numbers
whose denominators are not divisible by $p$.
Here $\big([p]\big)$ denotes the ideal of $\Z_{(p)}[q]$ generated by the irreducible polynomial $[p]$ over $\mathbb{Z}_{(p)}$.
An element of $\mathcal{Q}$ is of the form $(a_{p})_p$, where $p$ runs over all primes and $a_{p}\in \mathbb{Z}_{(p)}[q]\big/\big([p]\big)$.
Similarly to the convention of $\mathcal{A}$, we will also regard $(f_p(q)\mod [p])_p $ as an element in $\mathcal{Q}$ for each sequence $\{f_n(q) \}_{n\ge0} \subset\Q[q]$ such that, for finitely many primes $p$, some of denominators of coefficients of $f_p(q)$ in $q$ may be divisible by $p$.
%Two elements $(f_{p})_p$ and $(g_{p})_p$ are identified if and only if $f_{p}=g_{p}$ for all but finitely many primes $p$.
With this, the map $\Q[q]\rightarrow \mathcal{Q}, f(q) \mapsto (f(q)\mod [p])_p$ is well-defined.
Component-wise addition, multiplication and scalar multiplication by $\Q[q]$
equip ${\mathcal{Q}}$ with the structure of a $\Q[q]$-algebra.

In this framework, we introduce the following object.
\begin{definition}
For each index $\kk$ and $\bullet\in\{\emptyset,\star\}$, we define $\zeta_{\mathcal{Q}}^\bullet(\kk)\in \mathcal{Q}$ by
\begin{align*}
 \zeta_{\mathcal{Q}}^\bullet (\kk) &= \big( H_{p-1}^\bullet (\kk;q) \mod [p]\big)_p.
\end{align*}
\end{definition}
Note that for each prime $p$ and positive integer $m$ coprime to $p$, since $\gcd([m],[p])=1$, the polynomial $[m]$ is invertible modulo $[p]$, and hence the above definition makes sense.
% (see also Lemma \ref{lem:[m]-inverse}).

For each weight $k$, one can observe that elements $\zeta_{{\mathcal{Q}}}(\kk)$'s satisfy numerous relations of the form
\begin{equation}\label{eq:rel's_Q-MZV}
\sum_{j=0}^k (1-q)^j \sum_{\substack{  \wt(\kk)=k-j}}c_{\kk,j}(\pp) \,\zeta_{{\mathcal{Q}}}(\kk) =0
\end{equation}
for some polynomials $c_{\kk,j}(x)\in \Q[x]$,
where $\pp=(p\mod [p])_p \in \mathcal{Q}$.
In other words, the elements $(1-q)^j \zeta_{{\mathcal{Q}}}(\kk)$ of weight $k$ satisfy linear relations over $\Q[\pp]$, where we set the term $1-q$ to be of weight 1.
For example, as a $q$-analogue of the well-known congruence $H_{p-1}(1)\equiv 0\mod p$, we have
\begin{equation}\label{eq:cong_wt1}
\zeta_{\mathcal{Q}}(1)-\frac{\pp-1}{2}(1-q)=0,
\end{equation}
which is due to Andrews \cite{Andrews99}.
One can also find other examples in \cite{HPT17} and \cite{Zhao13,Zhao16}.
Our duality formula \eqref{eq:duality_H} gives
\begin{equation*}\label{eq:duality_Q}
\zeta_{\mathcal{Q}}^\star (\kk) + (-1)^{\wt(\kk)}\zeta_{\mathcal{Q}}^\star(\overline{\kk^\vee})=0,
\end{equation*}
which by \eqref{eq:partial_franction_q} is also an example of \eqref{eq:rel's_Q-MZV}.
%We give all possible linear relations up to weight 3 below.

\begin{remark}
In \cite[\S3]{BTT18}, the authors introduce the cyclotomic analogue of finite multiple zeta value, which is slightly different from $\zeta_{\mathcal{Q}}(\kk)$.
In our terminology, the cyclotomic analogue is viewed as $\zeta_{\mathcal{Q}}(\kk) \mod \pp$ in the quotient ring $\mathcal{Q}/\pp\mathcal{Q}$.
\end{remark}

\begin{remark}\label{rem:QMZV}
From the definition, our $\zeta_{{\mathcal{Q}}}(\kk)$ may be viewed as a ``$q$-analogue" of the finite multiple zeta value $\zeta_{\mathcal{A}}(\kk)$.
However, since both $\zeta_{\mathcal{A}}(\kk)$ and $\zeta_{\mathcal{S}}(\kk)$ are obtained as certain realizations of $\zeta_{{\mathcal{Q}}}(\kk)$ (see \eqref{eq:recast_BTT}), it should have its own name, such as a \emph{$\mathcal{Q}$-multiple zeta value} (abbreviated by $\mathcal{Q}$-MZV).
In particular, since there are many variants of multiple harmonic $q$-sums, our $\zeta_{{\mathcal{Q}}}(\kk)$ would be called the $\mathcal{Q}$-MZV of Bradley--Zhao's model.
\end{remark}

%%%%%%%%%%%%%%%%%%%%%%%%%%%%%%%%%%
\subsection{From $\mathcal{Q}$-MZV to $\mathcal{A}$-MZV and $\mathcal{S}$-MZV}\label{sec:QtoAS}
We first define two $\Q$-algebra maps $\phi_{\mathcal{A}}:\mathcal{Q}\rightarrow \mathcal{A}$ and $\phi_{\mathcal{S}}: \mathcal{O}\rightarrow \C$, and then, recast Theorem \ref{thm:BTT} via $\mathcal{Q}$-MZVs.

For $\phi_{\mathcal{A}}$, we note that if two polynomials $f(q)$ and $g(q)$ over $\Z_{(p)}$ satisfies $f(q)\equiv g(q)\mod [p]$,
then $f(1)\equiv g(1) \mod p$.
Therefore, there is the natural projection
\[ \mathbb{Z}_{(p)}[q]\big/\big([p]\big)\longrightarrow \mathbb{Z}_{(p)}/p\mathbb{Z}_{(p)}
\simeq \mathbb{Z}/p\mathbb{Z}\]
that sends $q$ to $1$.
This induces the surjective $\Q$-algebra map
\begin{equation*}\label{eq:map_phi_A}
\phi_{{\mathcal{A}}}: {\mathcal{Q}} \longrightarrow {\mathcal{A}}, \quad (f_p(q)\mod [p])_p\longmapsto (f_p(1)\mod p)_p.
\end{equation*}

The $\Q$-algebra map $\phi_{\mathcal{S}}$ is defined as the composition of $\Q$-algebra maps $ lim $ and $ ev $ given as follows.
Let $\Q(\zeta_p)$ be the $p$th cyclotomic field.
Set
\[\mathcal{Q}^{\rm an}=  \bigg(\prod_{p:{\rm prime}} \Q(\zeta_p)\bigg)\bigg/ \bigg(\bigoplus_{p:{\rm prime}} \Q(\zeta_p)\bigg).\]
Since $[p]=\prod_{a=1}^{p-1}(q-e^{2\pi ia/p})$, the map
\begin{align*}
 ev : \mathcal{Q}&\longrightarrow \mathcal{Q}^{\rm an},\quad
 (f_p(q)\mod [p])_p\longmapsto (f_p(e^{\frac{2\pi i}{p}}))_p
 \end{align*}
is a well-defined $\Q$-algebra map.
Consider the $\Q$-subalgebra
%%%%%
%\[ \mathcal{O}^{\rm an}=\bigg\{(z_p)_p \in \prod_{p:{\rm prime}} \Q(\zeta_p) \, \bigg| \, %\lim_{p\rightarrow \infty} z_p <\infty\bigg\}\]
\[ \mathcal{O}^{\rm an}=\bigg\{(z_p)_p \in \prod_{p:{\rm prime}} \Q(\zeta_p) \, \bigg| \, \lim_{p\rightarrow \infty} z_p \,\, \hbox{converges} \bigg\}\]
%%%%%
of $\prod_{p:{\rm prime}} \Q(\zeta_p) $
and define
\begin{align*}
 lim : \mathcal{O}^{\rm an} &\longrightarrow \C,\quad
(z_p)_p \longmapsto \lim_{p\rightarrow \infty} z_p,
\end{align*}
which is also a $\Q$-algebra map.
Its kernel contains the $\Q$-subalgebra $\bigoplus_p \Q(\zeta_p)$.
Now let
\[ \mathcal{O}= ev ^{-1} \bigg(\mathcal{O}^{\rm an}\bigg/ \bigoplus_{p:{\rm prime}} \Q(\zeta_p)\bigg)\subset \mathcal{Q}\]
and define the $\Q$-algebra map
\[ \phi_{\mathcal{S}}=lim \circ  ev  : \mathcal{O}\longrightarrow \C,\]
which sends $ (f_p(q)\mod [p])_p$ to $\displaystyle\lim_{p\rightarrow\infty} f_p(e^{2\pi i/p})$.
%In the above map, since $\phi_{\mathcal{S}}(q)=1$, we can regard $\mathcal{O}$ as a $\Q[q]$-algebra and the map $\phi_{\mathcal{S}}$ is an algebra homomorphism.

Theorem \ref{thm:BTT} is now restated as follows.

\begin{theorem}\label{thm:AS_from_q}
For all index $\kk$ and $\bullet\in\{\emptyset,\star\}$, we have
\[ \phi_{{\mathcal{A}}} \big( \zeta_{{\mathcal{Q}}}^{\bullet}(\kk)\big) = \zeta_{{\mathcal{A}}}^\bullet(\kk)\]
and
\[ \phi_{{\mathcal{S}}} \big( \zeta_{{\mathcal{Q}}}^{\bullet}(\kk)\big) \equiv \zeta_{\mathcal{S}}^\bullet(\kk)\mod \pi i \mathcal{Z}[\pi i].\]
\end{theorem}

\begin{proof}
The first equation is immediate from the identity
\begin{align*}
\phi_{{\mathcal{A}}}\big((q^{(k-1)m}[m]^{-k}\mod [p])_p\big)=
(m^{-k}\mod p)_p
\end{align*}
for any $m\in\{1,2,\ldots,p-1\}$.
The second congruence is a consequence of Theorem \ref{thm:BTT} (ii).
We complete the proof.
\end{proof}

An application to the study of $\mathcal{A}$-MZVs and $\mathcal{S}$-MZVs is as follows.
Applying $\phi_{\mathcal{A}}$ to \eqref{eq:rel's_Q-MZV}, we obtain
\[ \sum_{\substack{ \wt(\kk)=k}} c_{\kk,0}(0)\zeta_{\mathcal{A}}(\kk)=0.\]
On the other hand, since $(1-q)^jc_{\kk,j}(p)\big|_{q=e^{2\pi i/p}} =O(p^{\deg c_{\kk,j}-j})$, the image of the left side of \eqref{eq:rel's_Q-MZV} under the map $\phi_{\mathcal{S}}$ diverges, if $\deg c_{\kk,j}>j$.
Thus, under the assumption that $\deg c_{\kk,j}\le j$ for all $\kk$ and $j$, the left side of \eqref{eq:rel's_Q-MZV} lies in $\mathcal{O}$ and taking $\phi_{\mathcal{S}}$ gives
\[ \sum_{\substack{\wt(\kk)=k}} c_{\kk,0}(0)\zeta_{\mathcal{S}}(\kk)=0,\]
because $\phi_{\mathcal{S}}\big((1-q)^jc_{\kk,j}(\pp)\zeta_{\mathcal{Q}}(\kk)\big) \in \Q (\pi i)^j \xi(\kk)$.
This shows that the relation \eqref{eq:rel's_Q-MZV} among $\mathcal{Q}$-MZVs with the above assumption gives rise to relations among both $\mathcal{A}$-MZVs and $\mathcal{S}$-MZVs in the same shape.
As a quick example, by \eqref{eq:cong_wt1}, we get
\[ \zeta_{\mathcal{A}}(1)=0 \quad \mbox{and}\quad \zeta_{\mathcal{S}}(1)=0.\]
%Also, the duality formulas \eqref{eq:F-duality} of both $\mathcal{A}$-MZSVs and $\mathcal{S}$-MZSVs are consequences of \eqref{eq:duality_Q} and the reversal relation.

\begin{remark}\label{rem:ev_injective}
One can show that the map $ ev $ is injective.
This can be checked as follows.
Suppose that a primitive $p$-th root of unity $\zeta_p$ is zero of $f_p(q)$ for all but finitely many primes $p$.
By the action of ${\rm Gal}(\Q(\zeta_p)/\Q)$, we see that $\zeta_p^a \ (1\le a\le p-1)$ is also zero of $f_p(q)$.
Hence, $f_p(q)$ is divisible by $[p]$.
It is worth mentioning that the injectivity of the map $ ev $ implies that the equation \eqref{eq:rel's_Q-MZV} holds if and only if
\[ \sum_{j=0}^k (1-\zeta_p)^j \sum_{\substack{ \kk \\ \wt(\kk)=k-j}}c_{\kk,j}(p) \,H_{p-1}(\kk;\zeta_p) =0\]
holds in $\Q(\zeta_p)$ for all large primes $p$.
%The latter can be used for numerical computations of relations among $\mathcal{Q}$-MZVs up to certain digits.
\end{remark}

%%%%%%%%%%%%%%%%%%%%%%%%%%%%%%%%%%
\section{A refined Kaneko-Zagier conjecture: $\widehat{\mathcal{A}}$-MZV and $\widehat{\mathcal{S}}$-MZV}

Further studies of the Kaneko--Zagier conjecture have been made by many authors \cite{AHY,Jarossay18,Jarossay19,OnoSekiYamamoto,Rosen15,Rosen,Yasuda19}.
In this section, from their works we review a theory of $\widehat{\mathcal{A}}$-MZV and $\widehat{\mathcal{S}}$-MZV together with a conjectural relationship between them (which is called a refined version of the Kaneko--Zagier conjecture).
%Relationships between conjectures we describe are also mentioned.
% that ``a refined version implies the Kaneko--Zagier conjecture" under the assumption that Rosen's lifting conjecture holds.

%%%%%%%%%%%%%%%%%%%%%%%%%%%%%%%%%%
\subsection{Supercongruences for multiple harmonic sums: an overview}

Multiple harmonic sums satisfy mod $p^n$ congruences relations.
A systematic study of these relations is
%%%%%%%
%due to by Rosen \cite{Rosen15}.
due to Rosen \cite{Rosen15}.
%%%%%%%
As a brief overview of his works, let us begin with a version of Rosen's lifting conjecture \cite[Conjecture A]{Rosen15}.

\begin{conjecture}\label{conj:lifting}
Let $k$ be a positive integer and $p_1$ be a prime greater than $k+1$.
For a finite set $\{c_{\kk}\mid \wt(\kk)=k\}$ of rational numbers, suppose that the congruence
\[ \sum_{\wt(\kk)=k}  c_{\kk}H_{p-1}(\kk)\equiv 0 \mod p\]
 holds for all primes $p\ge p_1$.
Then, for each $n\ge2$, there exists a prime $p_n$ and a finite set $\{c_{\kk}^{(l)} \mid 1\le l\le n-1,\, \wt(\kk)=k+l\}$ of rational numbers such that $p_n\ge p_{n-1}$ and
the supercongruence
\begin{equation}\label{eq:supercongruence_Hp}
\sum_{\wt(\kk)=k}  c_{\kk} H_{p-1}(\kk) + \sum_{l=1}^{n-1}p^l\sum_{\wt(\kk)=k+l} c_{\kk}^{(l)} H_{p-1} (\kk)\equiv 0\mod p^n
\end{equation}
holds for all primes $p\ge p_n$.
\end{conjecture}

Roughly, Conjecture \ref{conj:lifting} says that mod $p$ congruences of $H_{p-1}(\kk)$'s, which hold for all but finitely many primes $p$, can be lifted into mod $p^n$ congruences of $p^l H_{p-1}(\kk)$'s for each $n$, which also hold for all but finitely many primes $p$ (note that the number of exceptions of primes in the latter congruence may  be strictly bigger than the former ones for some $n$).

Let us illustrate two examples of lifts of mod $p$ congruences.
For each positive integer $n$, Wolfstenholme's theorem \cite{Wolstenholme62}, $H_{p-1}(1)\equiv 0\mod p^2$, can be lifted into
\begin{equation}\label{eq:wt1_Hp}
H_{p-1}(1)+ \sum_{l=0}^{n-2}p^l \, H_{p-1}(1+l)  \equiv 0 \mod p^n \quad (\forall p\gg 0:\mbox{prime}),
\end{equation}
which is a special case of Corollary 4.3.3 in \cite{RosenPHD}.
Hoffman's duality \eqref{eq:hoffman's_dual} extends to
\begin{equation}\label{eq:seki's_dual}
\sum_{l=0}^{n-1} p^l \bigg(H_{p-1}^\star(\underbrace{1,\ldots,1}_l,\kk) +H_{p-1}^\star(\underbrace{1,\ldots,1}_l,\kk^\vee)\bigg) \equiv 0 \mod{p^n} \quad (\forall p:\mbox{prime})
\end{equation}
(see \cite{Seki19}).
Note that without fixing a basis, the way of extension is of course not unique.
For example, one has
\[ H_{p-1}(1) + \frac13 p^2H_{p-1}(2,1)-\frac16 p^4 H_{p-1}(4,1)-\frac19 p^5 H_{p-1}(4,1,1) \equiv 0 \mod{p^6},\]
which again recovers Wolfstenholme's theorem.

\begin{remark}
One can show that the injectivity of the $\mathcal{A}$-valued period map ${\rm per}_{\mathcal{A}}:{\rm Fil}^0\mathcal{P}^{\mathfrak{dr}}\big/{\rm Fil}^1\mathcal{P}^{\mathfrak{dr}} \rightarrow \mathcal{A}$, introduced by Rosen \cite{Rosen}, implies Conjecture \ref{conj:lifting}, where ${\rm Fil}^\bullet$ is a decreasing filtration coming from the Hodge filtration on algebraic de Rham cohomology.
In our case, we consider $\mathcal{P}^{\mathfrak{dr}}$ as the ring of de Rham periods of mixed Tate motives over $\Z$.
The map ${\rm per}_{\mathcal{A}}$ is an analogue of the period map, and is expected to be injective \cite[Theorem 6.4]{Rosen}.
\end{remark}

%%%%%%%%%%%%%%%%%%%%%%%%%%%%%%%%%%
\subsection{$\widehat{\mathcal{A}}$-MZV}
To deal with congruences of type \eqref{eq:supercongruence_Hp},
Rosen \cite{Rosen15} introduces the $\Q$-algebra
\begin{align*}
\mathcal{A}_n=
\bigg(\prod_{p:{\rm prime}} \Z/p^n \Z\bigg) \bigg/\bigg( \bigoplus_{p:{\rm prime}} \Z/p^n \Z\bigg).
\end{align*}
For each index $\kk$, let
\[ \zeta_{\mathcal{A}_n} (\kk)=\big( H_{p-1}(\kk)\mod p^n\big)_p \in \mathcal{A}_n\]%,\quad \zeta_{\mathcal{A}_n}^\star (\kk)=\big( H_{p-1}^\star(\kk)\mod p^n\big)_p \in \mathcal{A}_n\]
and set $\pp=\big( p\mod p^n\big)_p$.
In $\mathcal{A}_n$, \eqref{eq:supercongruence_Hp} can be written as an identity
\[ \sum_{\wt(\kk)=k}  c_{\kk} \zeta_{\mathcal{A}_n}(\kk) + \sum_{l=1}^{n-1}\pp^l\sum_{\wt(\kk)=k+l} c_{\kk}^{(l)} \zeta_{\mathcal{A}_n} (\kk)=0.\]
Relations among $\zeta_{\mathcal{A}_n}(\kk)$ are of particular interest and discussed in \cite{MuraharaOnozukaSeki20,OnoSakuradaSeki}.

As a completion of $\mathcal{A}_n$, Rosen \cite[Definition 3.2]{Rosen15} defines the $\Q$-algebra $\widehat{\mathcal{A}}$ to be
the projective limit
\begin{equation}\label{eq:def_hatA}
\widehat{\mathcal{A}}=\varprojlim_n \mathcal{A}_n,
\end{equation}
where for each $n\ge1$ the transition map $\varphi_n:\mathcal{A}_{n+1}\rightarrow\mathcal{A}_n$ is induced from the natural projection $\Z/p^{n+1}\Z \to \Z/p^n\Z$.
Note that $\widehat{\mathcal{A}}$ is complete under the $\pp$-adic topology,
where $\pp=\big((p\mod p^n)_p\big)_n \in \widehat{\mathcal{A}}$.
Any element of $\widehat{\mathcal{A}}$ is written as
$\big((a_{p, n})_{p}\big)_{n}$ with
$(a_{p, n})_{p} \in \mathcal{A}_n$ satisfying $\varphi_{n}\big((a_{p, n+1})_{p} \big)=(a_{p, n})_p$ for all $n\ge1$.
Then two elements $\big((a_{p,n})_p \big)_n$ and $\big((b_{p,n})_p \big)_n$ of $\widehat{\mathcal{A}}$
are identified if and only if there exists an increasing sequence $\{p_n\}_{n\ge1}$ of primes
such that $a_{p,n}\equiv b_{p,n}\mod p^n$ for all primes $p\ge p_n$ and $n\ge1$.
We remark that there exists a natural surjective homomorphism (see \cite[\S3.2]{Rosen15})
\begin{equation*}%\label{eq:surj_[p]-adic}
\prod_p \Z_p \rightarrow \widehat{\mathcal{A}}.
\end{equation*}

The following definition was first made by Rosen \cite{Rosen15}
(we use Seki's modification \cite{Seki19}).

\begin{definition}\label{def:hatA-MZV}
For an index $\kk$, we define the \emph{$\widehat{\mathcal{A}}$-MZV} $\zeta_{\widehat{\mathcal{A}}}(\kk) \in \widehat{\mathcal{A}}$ by
\[ \zeta_{\widehat{\mathcal{A}}}(\kk)= \big( (H_{p-1}(\kk) \mod{p^n})_p\big)_n \]
and the \emph{$\widehat{\mathcal{A}}$-MZSV} $\zeta_{\widehat{\mathcal{A}}}^\star(\kk) \in \widehat{\mathcal{A}}$ by
\[\zeta_{\widehat{\mathcal{A}}}^\star(\kk)=\left( (H_{p-1}^\star(\kk) \mod{p^n})_p\right)_n.\]
\end{definition}

For each weight $k$, $\widehat{\mathcal{A}}$-MZVs satisfy numerous relations of the form
\[ \sum_{l\ge0}\pp^l\sum_{\wt(\kk)=k+l} c_{\kk}^{(l)} \zeta_{\widehat{\mathcal{A}}} (\kk)=0\]
with rational numbers $c_{\kk}^{(l)}$.
%, where we set $\pp$ to be of weight $-1$.
For example, the relations \eqref{eq:wt1_Hp} and \eqref{eq:seki's_dual} are in $\widehat{\mathcal{A}}$ expressed as
\begin{equation}\label{eq:wt1}
\zeta_{\widehat{\mathcal{A}}}(1)+ \sum_{l\ge 0}\pp^l \, \zeta_{\widehat{\mathcal{A}}}(1+l)=0
\end{equation}
and
\begin{equation}\label{eq:hoffman_dual_hatA}
\sum_{l\ge0}\pp^l\, \big( \zeta_{\widehat{\mathcal{A}}}^\star(\underbrace{1,\ldots,1}_l,\kk) +\zeta_{\widehat{\mathcal{A}}}^\star(\underbrace{1,\ldots,1}_l,\kk^\vee)\big)=0,
\end{equation}
respectively.

%%%%%%%%%%%%%%%%%%%%%%%%%%%%%%%%%%
\subsection{$\widehat{\mathcal{S}}$-MZV}
As a counterpart of $\widehat{\mathcal{A}}$-MZVs, for each index $\kk$ we define
\[\zeta^\bullet_{\widehat{\mathcal{S}}}(\kk)=\sum_{a=0}^{\dep(\kk)} (-1)^{\wt(\kk_a)}
\sum_{\bl\in \Z_{\ge0}^a} t^{\wt(\bl)} b\binom{\kk_a}{\bl} \zeta^\bullet (\overline{\kk_a+\bl};T)\zeta^\bullet(\kk^a;T) \in \mathcal{Z}[[t]]\]
with the regularization $\bullet\in\{\ast,\shuffle\}$.
The above right side does not depend on $T$
(see, for example, \cite[Proposition 2.3]{OnoSekiYamamoto}).
Hence, for each index $\kk$ and $\bullet\in \{\ast,\shuffle\}$,
the $\zeta^\bullet_{\widehat{\mathcal{S}}}(\kk)$ lies in the formal power series ring $\mathcal{Z}[[t]]$ over $\mathcal{Z}$.
It was shown in \cite[Proposition 3.2.4]{Jarossay19} (see also \cite[Proposition 2.1]{OnoSekiYamamoto}) that the equality
\[ \zeta^\ast_{\widehat{\mathcal{S}}}(\kk) \equiv
\zeta^\shuffle_{\widehat{\mathcal{S}}}(\kk) \mod \pi i \mathcal{Z}[\pi i][[t]]\]
holds for each index $\kk$.
Hence,
the following definition is independent of the choice of the regularization $\bullet\in\{\ast,\shuffle\}$.

\begin{definition}\label{def:hatS-MZV}
For an index $\kk$, we define the \emph{$\widehat{\mathcal{S}}$-MZV} $\zeta_{\widehat{\mathcal{S}}}(\kk)\in \overline{\mathcal{Z}[\pi i]}[[t]]$ by
\[\zeta_{\widehat{\mathcal{S}}}(\kk)=\zeta^\ast_{\widehat{\mathcal{S}}}(\kk) \mod \pi i \mathcal{Z}[\pi i][[t]],\]
where we set
\[\overline{\mathcal{Z}[\pi i]}[[t]] =\big(\mathcal{Z}[\pi i]/\pi i \mathcal{Z}[\pi i]\big)[[t]].\]
The \emph{$\widehat{\mathcal{S}}$-MZSV} is defined by
\[ \zeta^\star_{\widehat{\mathcal{S}}}(k_1,\ldots,k_d) =
\sum_{\substack{\footnotesize \square \ \mbox{is either a comma `,'}\\ \footnotesize \mbox{or a plus `$+$'}}}
\zeta_{\widehat{\mathcal{S}}}(k_1 \square \cdots \square k_d).\]
\end{definition}

Note that the constant term of $\zeta_{\widehat{\mathcal{S}}}(\kk)$ in $t$ coincides with $\zeta_{{\mathcal{S}}}(\kk)$.
Studies of the $\widehat{\mathcal{S}}$-MZV are initiated by Hirose, Rosen and Jarossay (he called $\widehat{\mathcal{S}}$-MZV, without taking modulo $\pi i$, the $\Lambda$-adic adjoint MZV), independently.
Relations among $\widehat{\mathcal{S}}$-MZVs of the form
\[ \sum_{l\ge0}t^l\sum_{\wt(\kk)=k+l} c_{\kk}^{(l)} \zeta_{\widehat{\mathcal{S}}} (\kk)=0\]
are particular interest in this study.
For example, as a counterpart of \eqref{eq:hoffman_dual_hatA}, Hirose obtained
\begin{equation}\label{eq:hoffman_dual_hatS}
\sum_{l=0}^{\infty} t^l\,\big(\zeta_{\widehat{\mathcal{S}}}^\star(\underbrace{1,\ldots,1}_l,\kk) + %\sum_{l=0}^{\infty }
\zeta_{\widehat{\mathcal{S}}}^\star(\underbrace{1,\ldots,1}_l,\kk^\vee)\big)=0
\end{equation}
in his unpublished work.

%%%%%%%%%%%%%%%%%%%%%%%%%%%%%%%%%%
\subsection{$\widehat{\mathcal{A}}$-MZV vs.~$\widehat{\mathcal{S}}$-MZV}
Similarly to the Kaneko--Zagier conjecture (Conjecture \ref{conj:Kaneko-Zagier}), there is a conjectural correspondence between $\widehat{\mathcal{A}}$-MZVs and $\widehat{\mathcal{S}}$-MZVs.
%This leads to the following conjecture.

\begin{conjecture}[A refined version of the Kaneko--Zagier conjecture]\label{conj:refined_KZ}
Let $k$ be a positive integer.
For rational numbers $\{c_{\kk}^{(l)} \in \Q \mid \wt(\kk)=k+l,\ l\ge0\}$, we have
\begin{align*}
&\sum_{l\ge0} \pp^{l} \left( \sum_{\wt(\kk)=k+l} c_{\kk}^{(l)}   \zeta_{\widehat{\mathcal{A}}}(\kk) \right) =0\  \mbox{in}\ \widehat{\mathcal{A}}\\
\stackrel{?}{\Longleftrightarrow}& \sum_{l\ge0} t^{l}\left( \sum_{\wt(\kk)=k+l} c_{\kk}^{(l)}   \zeta_{\widehat{\mathcal{S}}}(\kk) \right)=0 \  \mbox{in}\ \overline{\mathcal{Z}}[[t]].
\end{align*}
\end{conjecture}

Similar statements with Conjecture \ref{conj:refined_KZ} can be found in \cite[Conjecture 5.3.2]{Jarossay19}, \cite[Conjecture 2.3]{Rosen19} and \cite[Conjecture 4.3]{OnoSekiYamamoto}.
Several evidence of Conjecture \ref{conj:refined_KZ} are known.
The duality formulas \eqref{eq:hoffman_dual_hatA} and \eqref{eq:hoffman_dual_hatS} are one of examples.
Also, a kind of double shuffle relations for both $\widehat{\mathcal{A}}$-MZVs and $\widehat{\mathcal{S}}$-MZVs in the same shape is obtained in \cite{Jarossay19,OnoSekiYamamoto}.

We mention a relationship with the Kaneko--Zagier conjecture (Conjecture \ref{conj:Kaneko-Zagier}).

\begin{proposition}\label{prop:implication}
Under the assumption that Conjecture \ref{conj:lifting} holds, Conjecture \ref{conj:refined_KZ} implies Conjecture \ref{conj:Kaneko-Zagier}.
%(i) If Conjectures \ref{conj:lifting} and \ref{conj:refined_KZ} are valid, then Conjecture \ref{conj:Kaneko-Zagier} holds.\\
%(ii) If Conjectures \ref{conj:Kaneko-Zagier} and \ref{conj:lifting} are valid, then Conjecture \ref{conj:refined_KZ} holds.
\end{proposition}

\begin{proof}
Assume that we have
\begin{equation}\label{eq:A_rel}
\sum_{\wt(\kk)=k} c_{\kk}^{(0)} \zeta_{\mathcal{A}}(\kk) =0.
\end{equation}
Conjecture \ref{conj:lifting} says that there exist rational numbers $c_{\kk}^{(l)} \ (l\ge1)$ such that
\begin{equation}\label{eq:hatA_rel}
\sum_{l\ge0}\pp^l \sum_{\wt(\kk)=k+l} c_{\kk}^{(l)} \zeta_{\widehat{\mathcal{A}}}(\kk) =0.
\end{equation}
Using Conjectures \ref{conj:refined_KZ}, we obtain
\begin{equation}\label{eq:hatS_rel}
\sum_{l\ge0}t^l \sum_{\wt(\kk)=k+l} c_{\kk}^{(l)} \zeta_{\widehat{\mathcal{S}}}(\kk) =0.
\end{equation}
Comparing the constant terms of both sides, one has
\begin{equation}\label{eq:S_rel}
\sum_{\wt(\kk)=k} c_{\kk}^{(0)} \zeta_{\mathcal{S}}(\kk) =0,
\end{equation}
so, $\eqref{eq:A_rel} \Rightarrow \eqref{eq:S_rel}$.
For the opposite implication $\eqref{eq:S_rel} \Rightarrow \eqref{eq:A_rel}$,
Yasuda's result \cite{Yasuda16}, stating that all MZVs of weight $k$ are $\Q$-linear combinations of $\mathcal{S}$-MZVs of weight $k$, determines a lift of \eqref{eq:S_rel} to \eqref{eq:hatS_rel}, inductively (the same argument is used in the proof of \cite[Proposition 4.4]{OnoSekiYamamoto}).
Using Conjecture \ref{conj:refined_KZ} and then applying to \eqref{eq:hatA_rel} the natural isomorphism $\widehat{\mathcal{A}}\big/\pp\widehat{\mathcal{A}}\cong \mathcal{A} $ given by $\zeta_{\widehat{\mathcal{A}}}(\kk) \mod \pp\mapsto \zeta_{\mathcal{A}}(\kk)$, we get \eqref{eq:A_rel}.
We complete the proof.
\end{proof}

%%%%%%%%%%%%%%%%%%%%%%%%%%%%%%%%%%
\section{$\widehat{\mathcal{Q}}$-MZV}
In this section, we introduce a $\widehat{\mathcal{Q}}$-MZV and give a generalization of Theorem \ref{thm:AS_from_q}.

%%%%%%%%%%%%%%%%%%%%%%%%%%%%%%%%%%
\subsection{The $\Q$-algebra $\widehat{\mathcal{Q}}$}
Let us first extend the framework of $\mathcal{Q}$-MZVs.

For each $n\ge1$ and a prime $p$,
let $\big([p]^n \big)$ be the ideal of $\Z_{(p)}[q]$ generated by the polynomial $[p]^n$,
where $\Z_{(p)}[q]$ denotes the polynomial ring in $q$ over the ring $\Z_{(p)}$ of rational numbers
whose denominators are not divisible by $p$.
Consider the rings
\begin{align*}
Z_{p, n}=\mathbb{Z}_{(p)}[q]\big/\big([p]^{n}\big)\quad \mbox{and}\quad
\mathcal{Q}_{n}=\bigg(\prod_{p}Z_{p, n}\bigg)\bigg/\bigg(\bigoplus_{p}Z_{p,n}\bigg),
\end{align*}
where $p$ runs over all primes.
Note that $\mathcal{Q}=\mathcal{Q}_1$.
An element of $\mathcal{Q}_{n}$ is of the form $(f_{p,n})_p$, where $p$ runs over all primes and $f_{p,n}\in Z_{p,n}$.
Two elements $(f_{p,n})_p$ and $(g_{p,n})_p$ are identified if and only if $f_{p,n}=g_{p,n}$ for all but finitely many primes $p$.
One can embed $\Q[q]$ into $\mathcal{Q}_{n}$ diagonally.
Component-wise addition, multiplication and scalar multiplication by $\Q[q]$
equip ${\mathcal{Q}}_n$ with the structure of a $\Q[q]$-algebra.

Similarly to the $\Q$-algebra $\widehat{\mathcal{A}}$,
we define the $\Q[q]$-algebra $\widehat{\mathcal{Q}}$ to be
the projective limit of the system of rings $\{\mathcal{Q}_n\}$:
\begin{equation*}\label{eq:def_hatQ}
\widehat{\mathcal{Q}}=\varprojlim_n \mathcal{Q}_n,
\end{equation*}
where for each $n\ge1$ the transition map $\varphi_n:\mathcal{Q}_{n+1}\rightarrow\mathcal{Q}_n$ is induced from the natural projection $Z_{p, n+1} \to Z_{p, n}$.
It is endowed with the projective limit topology induced from
the discrete topology on $\mathcal{Q}_{n}$.
We also have a natural surjective homomorphism
\begin{equation*}%\label{eq:surj_[p]-adic}
\prod_p Z_p \rightarrow \widehat{\mathcal{Q}},
\end{equation*}
where $Z_p=\varprojlim_n Z_{p,n}$ denotes the ring of $[p]$-adic integers (cf.~\cite[Lemma 2.3]{Seki19}).

Any element of $\widehat{\mathcal{Q}}$ is written as
$\big((f_{p, n})_{p}\big)_{n}$ with
$(f_{p, n})_{p} \in \mathcal{Q}_n$ satisfying $\varphi_{n}\big((f_{p, n+1})_{p} \big)=(f_{p, n})_p$ for all $n\ge1$.
Then two elements $\big((f_{p,n})_p \big)_n$ and $\big((g_{p,n})_p \big)_n$ of $\widehat{\mathcal{Q}}$
are identified if and only if there exists an increasing sequence $\{p_n\}_{n\ge1}$ of primes
such that $f_{p,n}=g_{p,n}$ in $Z_{p, n}$ for all primes $p\ge p_n$ and $n\ge1$.

Define the element $[\pp]$ of $\widehat{\mathcal{Q}}$ by
\begin{align*}
[\pp]=\big(([p] \mod [p]^{n})_{p}\big)_{n}.
\end{align*}
The continuous homomorphism $\widehat{\mathcal{Q}}\rightarrow \mathcal{Q}_n$ induces $\mathcal{Q}_n\cong \widehat{\mathcal{Q}}/[\pp]^n\widehat{\mathcal{Q}}$.
The projective limit topology on $\widehat{\mathcal{Q}}$ coincides with the $[\pp]$-adic topology and $\widehat{\mathcal{Q}}$ is complete under the $[\pp]$-adic topology (cf.~\cite[Lemma 2.5]{Seki19}).

%%%%%%%%%%%%%%%%%%%%%%%%%%%%%%%%%%
\subsection{Definition of $\widehat{\mathcal{Q}}$-MZV}

We begin with the following lemma.

\begin{lemma}\label{lem:[m]-inverse}
Let $p$ be a prime and $p>m\ge 1$.
For any $n \ge 1$, the $q$-integer $[m]$ is invertible in $Z_{p, n}$.
More explicitly its inverse is given by
\begin{align*}
\frac{1}{[m]}\equiv [l]_{q^{m}}\sum_{j=0}^{n-1}
\left(-q[\alpha]_{q^{p}}[p]\right)^{j} \quad \mod [p]^{n},
\end{align*}
where $l$ and $\alpha$ are the integers satisfying
$p>l\ge 1, \alpha \ge 0$ and $ml-p \alpha=1$.
\end{lemma}

\begin{proof}
Using the relation $ml=1+p \alpha$, we see that
$[m][l]_{q^{m}}=1+q[\alpha]_{q^{p}}[p]$.
Hence we have
\begin{align*}
[m][l]_{q^{m}}\sum_{j=0}^{n-1}
\left(-q[\alpha]_{q^{p}}[p]\right)^{j}=1-(-q[\alpha]_{q^{p}}[p])^{n} \equiv 1
\quad \mod [p]^{n},
\end{align*}
which completes the proof.
\end{proof}

{}From Lemma \ref{lem:[m]-inverse},
the following definition makes sense.

\begin{definition}\label{def:hatQ-MZV}
For an index $\kk$ and $\bullet\in\{\emptyset,\star\}$, we define
\[ \zeta_{\widehat{\mathcal{Q}}}^\bullet (\kk)= \big( (H_{p-1}^\bullet(\kk;q) \mod{[p]^n})_p\big)_n  \]
and
\[ \overline{\zeta}_{\widehat{\mathcal{Q}}}^\bullet(\kk)=\left(
(\overline{H}_{p-1}^\bullet(\kk;q) \mod{[p]^n})_p\right)_n \]
as elements in $\widehat{\mathcal{Q}}$,
where for an index $\kk=(k_1,\ldots,k_d)$, we set
\begin{align*}
\overline{H}_m(\kk;q)&=\sum_{m\ge m_1>\cdots>m_d>0} \prod_{a=1}^d \frac{q^{m_a}}{[m_a]^{k_a}} \quad \mbox{and}\quad
\overline{H}_m^\star (\kk;q)=\sum_{m\ge m_1\ge \cdots\ge m_d>0} \prod_{a=1}^d \frac{q^{m_a}}{[m_a]^{k_a}}.
\end{align*}
\end{definition}

In the same vein with Remark \ref{rem:QMZV}, we call the above objects \emph{$\widehat{\mathcal{Q}}$-multiple zeta values} (abbreviated by $\widehat{\mathcal{Q}}$-MZVs).
%The $\zeta_{\widehat{\mathcal{Q}}}(\kk)$ is Bradley--Zhao's model.
%The model $ \overline{\zeta}_{\widehat{\mathcal{Q}}}(\kk)$ is studied in \cite{Schlesinger,Zudilin03}, so is called Schlesinger--Zudilin's model (see also \cite{Ebrahimi-FardManchonSinger16}).
The use of ``bar'' on $\zeta_{\widehat{\mathcal{Q}}}$ seems confusing with the complex conjugate, but it comes into play in the study of relations later.

In what follows, similarly to \S\ref{sec:QtoAS}, we prepare algebra maps $\phi_{\widehat{\mathcal{A}}}:\widehat{\mathcal{Q}}\rightarrow \widehat{\mathcal{A}}$ and  $\phi_{\widehat{\mathcal{S}}}:\widehat{\mathcal{O}}\rightarrow \C[[t]]$, and then, give our main result stating that $\widehat{\mathcal{A}}$-MZVs and $\widehat{\mathcal{S}}$-MZVs are obtained as images of $\widehat{\mathcal{Q}}$-MZVs under these maps.
Since the setting of $\phi_{\widehat{\mathcal{S}}}$ needs more spaces, we do this in the following two subsections and the proof of our main result is postponed to Appendix \ref{app:proof_of_main_result}.

%%%%%%%%%%%%%%%%%%%%%%%%%%%%%%%%%%
\subsection{From $\widehat{\mathcal{Q}}$-MZV to $\widehat{\mathcal{A}}$-MZV}
The map $\phi_{\widehat{\mathcal{A}}}$ is defined in a similar manner to $\phi_{\mathcal{A}}$ in \S\ref{sec:QtoAS}.
If two polynomials $f(q)$ and $g(q)$ over $\Z_{(p)}$ satisfies $f(q)\equiv g(q)\mod [p]^n$,
then $f(1)\equiv g(1) \mod p^n$.
Therefore, there is the natural projection
\[ Z_{p,n}\longrightarrow \mathbb{Z}_{(p)}/p^{n}\mathbb{Z}_{(p)}
\simeq \mathbb{Z}/p^{n}\mathbb{Z}\]
that sends $q$ to $1$.
This induces the continuous surjective algebra map
\begin{equation*}\label{eq:map_phi}
\phi_{\widehat{\mathcal{A}}}: \widehat{\mathcal{Q}} \longrightarrow \widehat{\mathcal{A}},\quad \big((f_{p,n}(q) \mod [p]^n)_p\big)_n\longmapsto \big((f_{p,n}(1) \mod p^n)_p\big)_n.
\end{equation*}

\begin{theorem}\label{thm:hatA_from_q}
For all index $\kk$ and $\bullet\in\{\emptyset,\star\}$, we have
\[ \phi_{\widehat{\mathcal{A}}} \big( \zeta_{\widehat{\mathcal{Q}}}^\bullet(\kk)\big) =  \phi_{\widehat{\mathcal{A}}} \big(\overline{ \zeta}_{\widehat{\mathcal{Q}}}^\bullet(\kk)\big) = \zeta_{\widehat{\mathcal{A}}}^\bullet(\kk).\]
\end{theorem}

\begin{proof}
This is immediate from the identity
\begin{align*}
\phi_{\widehat{\mathcal{A}}}\big(\big(([m]^{-1}\mod [p]^n)_p\big)_n\big)=
\big((m^{-1}\mod p^n)_p\big)_n
\end{align*}
for any $m\in\{1,2,\ldots,p-1\}$.
\end{proof}

%%%%%%%%%%%%%%%%%%%%%%%%%%%%%%%%%%
\subsection{From $\widehat{\mathcal{Q}}$-MZV to $\widehat{\mathcal{S}}$-MZV}
The map $\phi_{\widehat{\mathcal{S}}}$ is defined to be the composition of two algebra maps $\widehat{lim}$ and $\widehat{ ev }$, as in the same vein as $\phi_{\mathcal{S}}$ in \S\ref{sec:QtoAS}.

First, let us define $\widehat{ ev }$.
For a prime $p$ and $n \ge 1$, set
\begin{align*}
Z_{p, n}^{\rm an}=\mathbb{Q}(\zeta_{p})[[t]]/(t^{n}) \quad \mbox{and}\quad
\mathcal{Q}^{\rm an}_n=\bigg(\prod_{p:{\rm prime}}Z_{p, n}^{\rm an}\bigg)\bigg/\bigg(\bigoplus_{p:{\rm prime}}Z_{p,n}^{\rm an}\bigg).
\end{align*}
We also write $(f_{p,n})_p$ for an element of $\mathcal{Q}^{\rm an}_n$, where $p$ runs over all primes and $f_{p,n}\in Z_{p,n}^{\rm an}$.
We apply the same construction with $\widehat{\mathcal{Q}}$ to
\[ \widehat{\mathcal{Q}}^{\rm an}=\varprojlim_n \mathcal{Q}_n^{\rm an},\]
where the transition map $ \mathcal{Q}_{n+1}^{\rm an}\rightarrow \mathcal{Q}_n^{\rm an}$ is induced by the natural projection $Z_{p, n+1}^{\rm an}\rightarrow Z_{p, n}^{\rm an}$.
The $\Q$-algebra $\widehat{\mathcal{Q}}^{\rm an}$ is complete under the $\boldsymbol{t}$-adic topology, where
\[ \boldsymbol{t}=\big( (t\mod t^n)_p\big)_n \in \widehat{\mathcal{Q}}^{\rm an}.\]
%endowed with
%the projective limit topology induced from the discrete topology on
%$ \prod_{p}Z_{p, n}^{\rm an}\bigg/\bigoplus_{p}Z_{p,n}^{\rm an}$.

There exists a unique formal power series $q_m(t)\in\Q(\zeta_m)[[t]]$ such that
\begin{equation}\label{eq:qmt}
q_m(0)=e^{\frac{2\pi i}{m}} \quad \mbox{and}\quad [m]_{q_m(t)}=t.
\end{equation}
For example, solving the above equations, one gets
\[ q_m(t)=e^{\frac{2\pi i}{m}}+(1-e^{-\frac{2\pi i}{m}})\frac{t}{m}+O(t^2).\]
Alternatively, the power series $q_m(t)$ can be viewed as an inverse function
of $t=[m]_{q}$, which is holomorphic in a neighborhood of $t=0$
because $dt/dq\neq 0$ at $q=e^{2\pi i/m}$.
See Remark \ref{rem:taylor_exp_qmt} for the explicit Taylor expansion.
For each prime $p$, we use the power series $q_p(t)$ to get the $\mathbb{Z}_{(p)}$-homomorphism
\begin{equation*}\label{eq:injective_hom}
Z_{p, n} \longrightarrow Z_{p, n}^{\rm an}, \qquad
f(q) \,\, \hbox{mod}\,\, [p]^{n} \mapsto
f(q_{p}(t))	\,\, \hbox{mod}\,\, t^{n}.
\end{equation*}
It induces the continuous $\mathbb{Q}$-algebra map
\[\widehat{ ev }: \widehat{\mathcal{Q}} \longrightarrow \widehat{\mathcal{Q}}^{\rm an},\quad \big((f_{p,n}(q)\mod [p]^n)_p\big)_n\longmapsto \big((f_{p,n}(q_p(t)) \mod t^n)_p\big)_n.\]

Secondly, we introduce $\widehat{lim}$.
Let $\mathcal{O}_{n}^{\rm an}$ be the $\mathbb{Q}$-subalgebra of $\prod_{p}Z_{p, n}^{\rm an}$ given by
\begin{align*}
\mathcal{O}_{n}^{\rm an}=\left\{
\big( \sum_{l=0}^{n-1}z_{p,l}\,t^{l} \, \hbox{mod}\, t^{n}\big)_{p} \in\prod_{p}Z_{p, n}^{\rm an}
\, \bigg| \,
\lim_{p \to \infty}z_{p, l} \,\, \hbox{converges for all} \,\, 0\le l \le n-1
\right\}.
\end{align*}
Note that
$\bigoplus_{p}Z_{p, n}^{\rm an} \subset \mathcal{O}_{n}^{\rm an} \subset \prod_{p}Z_{p, n}^{\rm an}$.
We define the $\mathbb{Q}$-subalgebra $\widehat{\mathcal{O}}^{\rm an}$ of $\widehat{\mathcal{Q}}^{\rm an}$ by
\begin{align*}
\widehat{\mathcal{O}}^{\rm an}=\lim_{\substack{\longleftarrow \\ n}}
\bigg(\mathcal{O}_{n}^{\rm an}\bigg/\bigoplus_{p}Z_{p,n}^{\rm an}\bigg).
\end{align*}
The transition map $ \mathcal{O}_{n+1}^{\rm an}/\bigoplus_{p}Z_{p,n+1}^{\rm an}\rightarrow \mathcal{O}_n^{\rm an}/\bigoplus_{p}Z_{p,n}^{\rm an}$ is also induced by the natural projection $Z_{p,n+1}^{\rm an}\rightarrow Z_{p,n}^{\rm an}$.
It can be shown that $\widehat{\mathcal{O}}^{\rm an}$ is a closed subset of $\widehat{\mathcal{Q}}^{\rm an}$. %with respect to the projective limit topology.
We note that any element $z$ of $\widehat{\mathcal{O}}^{\rm an}$ is written in the form
\begin{align*}
z=\big((\textstyle{\sum_{l=0}^{n-1}z_{p, l}^{(n)}\, t^{l} \,\, \hbox{mod} \,\, t^{n}})_{p}
\,\, \hbox{mod} \oplus_{p}Z_{p,n}^{\mathrm{an}}\big)_{n}
\end{align*}
with convergent sequences $\{z_{p, l}^{(n)}\}_{p}$ satisfying the condition that
there exists an increasing sequence $\{p_{n}\}_{n\ge1}$ of primes such that,
for all $n\ge1$ and primes $p \ge p_{n}$,
the equality $z_{p,l}^{(n+1)}=z_{p, l}^{(n)}$ holds for each $l\in\{0,1,\ldots,n-1\}$.
Taking $p \to \infty$ we see that
$\displaystyle\lim_{p \to \infty}z_{p,l}^{(n+1)}=\displaystyle\lim_{p \to \infty}z_{p, l}^{(n)}$
for all $0\le l<n$.
Set $z_{l}=\displaystyle\lim_{p \to \infty}z_{p,l}^{(n)}$ for $l \ge 0$,
which is independent on the choice of $n$ greater than $l$.
Thus we obtain the $\Q$-algebra map
\begin{align*}
\widehat{lim}: \widehat{\mathcal{O}}^{\rm an} \longrightarrow \mathbb{C}[[t]], \quad
z \mapsto \sum_{l\ge0}z_{l}\, t^{l}.
\end{align*}
It is continuous with respect to the $t$-adic topology on $\mathbb{C}[[t]]$.

Finally, we define $\phi_{\widehat{\mathcal{S}}}$.
Let us denote by $\widehat{\mathcal{O}}$ the closed $\Q$-subalgebra of $\widehat{\mathcal{Q}}$ given by
\begin{equation*}\label{eq:O}
\widehat{\mathcal{O}}=\{z\in \widehat{\mathcal{Q}}\mid  ev (z)\in \widehat{\mathcal{O}}^{\rm an}\}
\, \big(= ev ^{-1}\big(\widehat{\mathcal{O}}^{\rm an}\big)\big).
\end{equation*}
The map $\phi_{\widehat{\mathcal{S}}}$ is then defined by
\begin{align*}
\phi_{\widehat{\mathcal{S}}}=\widehat{lim} \circ \, \widehat{ ev }:
\widehat{\mathcal{O}} \longrightarrow \mathbb{C}[[t]],
\end{align*}
which is a continuous $\mathbb{Q}$-algebra map.
It is explicitly described as follows.
Let $f=\big((f_{p,n}(q)\,\, \hbox{mod}\, [p]^{n})_p\big)_n$ be
an element of $\widehat{\mathcal{O}}$.
Set $f_{p,n}(q_p(t))=\sum_{l=0}^{n-1}z_{p,l}^{(n)}t^l+O(t^{n})$.
Then one has $\phi_{\widehat{\mathcal{S}}}\big( f\big)= \sum_{l\ge0} z_l t^l$ with $z_l=\displaystyle\lim_{p\rightarrow \infty}z_{p,l}^{(n)} \,\, (n>l)$.
%By definition, we have $\phi_{\widehat{\mathcal{S}}}\big([\pp]\big)=t$.

\begin{theorem}\label{thm:hatS_from_q}
For any index $\kk$ and $\bullet\in\{\emptyset,\star\}$, we have
$\zeta_{\widehat{\mathcal{Q}}}^\bullet (\kk) \in \widehat{\mathcal{O}}$ and
\begin{equation}\label{eq:hatS_from_q}
\phi_{\widehat{\mathcal{S}}} \big( \zeta_{\widehat{\mathcal{Q}}}^\bullet (\kk)\big) =\widehat{\xi}^\bullet (\kk) \,
\end{equation}
where $\widehat{\xi}(\kk),\widehat{\xi}^\star(\kk)\in \mathcal{Z}[\pi i][[t]]$ are given by
\[\widehat{\xi}(\kk)=\sum_{a=0}^{\dep(\kk)} (-1)^{\wt(\kk_a)}
\sum_{\bl\in \Z_{\ge0}^a} t^{\wt(\bl)}\, b\binom{\kk_a}{\bl} \zeta^\ast \left(
\overline{\kk_a+\bl};\frac{\pi i}{2}
\right)
\zeta^\ast \left(\kk^a;-\frac{\pi i}{2}\right)
\]
and
\[
\widehat{\xi}^{\star}(k_{1}, \ldots , k_{d})=\sum_{\substack{\footnotesize \square \
\hbox{is either a comma `,'}\\ \footnotesize \mbox{or a plus `$+$'}}}
\widehat{\xi}(k_{1} \square \cdots \square k_{d}).
\]
Also, $\overline{\zeta}_{\widehat{\mathcal{Q}}}(\kk)$
and $\overline{\zeta}_{\widehat{\mathcal{Q}}}^{\star}(\kk)$ belong to $\widehat{\mathcal{O}}$ and
it holds that
\begin{equation}\label{eq:hatS_from_q2}
\phi_{\widehat{\mathcal{S}}}\big( \overline{\zeta}_{\widehat{\mathcal{Q}}}(\kk)\big)=
\overline{\widehat{\xi}(\kk)} \quad \mbox{and} \quad \phi_{\widehat{\mathcal{S}}}\big( \overline{\zeta}^\star_{\widehat{\mathcal{Q}}}(\kk)\big)=
\overline{\widehat{\xi}^{\star}(\kk)},
\end{equation}
where the bar on the right sides means taking the complex conjugate of
each coefficient in a formal power series over $\C$.
In particular, for any index $\kk$ and $\bullet\in\{\emptyset,\star\}$, we have
\begin{align*}
\phi_{\widehat{\mathcal{S}}} \big( \zeta_{\widehat{\mathcal{Q}}}^\bullet(\kk)\big)\equiv \phi_{\widehat{\mathcal{S}}} \big( \overline{\zeta}_{\widehat{\mathcal{Q}}}^\bullet(\kk)\big) \equiv \zeta_{\widehat{\mathcal{S}}}^\bullet (\kk)   \mod \pi i\mathcal{Z}[\pi i][[t]].
\end{align*}
\end{theorem}

Our proof of Theorem \ref{thm:hatS_from_q} can be found in Appendix \ref{app:proof_of_main_result}.

\begin{remark}
Similarly to the map $ ev $, one can prove the injectivity of $\widehat{ev}$ as follows.
For a rational function $f(q)$ in $q$ over $\Z_{(p)}$ whose denominator is coprime to $[p]$, it suffices to show that if $f(q_p(t))\equiv 0 \mod t^n$, then $f(q)\equiv 0 \mod [p]^n$.
Since $q_p(t)-e^{2\pi i/p}\equiv 0\mod t$, we have $f(q_p(t))/(q_p(t)-e^{2\pi i/p})^n \in \Q(\zeta_p)[[t]]$.
Hence $\frac{d^{n-1}}{dt^{n-1}}f(q_p(t))\equiv 0\mod t$, which shows that $e^{2\pi i/p}$ is zero of $f(q)$ of multiplicity $n$.
Notice that the congruence in question is stable under the action of ${\rm Gal}(\Q(\zeta_p)/\mathbb{Q})$.
Thus, all primitive $p$-th roots of unity are zero of $f(q)$ of multiplicity $n$.
Since $[p]=\prod_{a=1}^{p-1}(q-e^{2\pi ia/p})$, we conclude $f(q)\equiv 0 \mod{[p]^n}$.
\end{remark}

%%%%%%%%%%%%%%%%%%%%%%%%%%%%%%%%%%
\subsection{Application}
Let us illustrate an application to the study of relations of $\widehat{\mathcal{A}}$-MZVs and $\widehat{\mathcal{S}}$-MZVs.
We extend Andrews' relation \eqref{eq:cong_wt1} to $\widehat{\mathcal{Q}}$-MZVs and deduce corresponding relations for $\widehat{\mathcal{A}}$-MZVs and $\widehat{\mathcal{S}}$-MZVs. %\eqref{eq:wt1}, which is an extension of Wolstenholme's theorem \eqref{eq:wolstenholme}, and its $\widehat{\mathcal{S}}$-MZV version in the same time.

An extension of Andrews' relation \eqref{eq:cong_wt1} is as follows.
\begin{proposition}\label{prop:wt1_hatQ}
We have
\[\zeta_{\widehat{\mathcal{Q}}}(1)-\frac{\pp-1}{2}(1-q) +\frac12 \sum_{l\ge1}[\pp]^l \overline{\zeta}_{\widehat{\mathcal{Q}}}(1+l) =0.\]
\end{proposition}
\begin{proof}
As a rational function in $q$, it is shown in the proof of Theorem 1 in \cite{ShiPan07} that we have the identities
\begin{align*}
H_{p-1}(1;q)-\frac{p-1}{2}(1-q) &=\frac{1-q}{2}(1-q^p)\sum_{m=1}^{p-1}\frac{q^m}{(1-q^m)(q^m-q^p)}\\
&=-\frac12 [p] \sum_{m=1}^{p-1}\frac{q^m}{[m]}\left([m]-[p]\right)^{-1}.
\end{align*}
Using
\begin{equation*}
\frac{1}{[m]-[p]}=\sum_{l=0}^{n-1}\frac{1}{[m]^{l+1}}[p]^{l}\mod [p]^n,
\label{eq:m-p-inv}
\end{equation*}
we obtain the desired formula.
\end{proof}

To apply the map $\phi_{\widehat{\mathcal{F}}}$, we note
\begin{align*}
\phi_{\widehat{\mathcal{F}}}( [\pp])=\Lambda
\end{align*}
holds for $\mathcal{F}=\mathcal{A}$ and $\mathcal{S}$,
where $\Lambda$ is given by
\begin{align}
\Lambda=\left\{
\begin{array}{ll}
\boldsymbol{p} & (\mathcal{F}=\mathcal{A}), \\
t & (\mathcal{F}=\mathcal{S}).
\end{array}
\right.
\label{eq:def-Lambda}
\end{align}
Using Proposition \ref{prop:1-q} below, for $\mathcal{F}\in\{\mathcal{A},\mathcal{S}\}$, we have
\begin{align*}
\phi_{\widehat{\mathcal{F}}}(q)=1
\label{eq:imageof1-q}
\end{align*}
and
\[ \phi_{\widehat{\mathcal{S}}} (\pp (1-q)) = -2\pi i, \quad \phi_{\widehat{\mathcal{A}}} (\pp (1-q)) = 0.\]
Now, applying $\phi_{\widehat{\mathcal{S}}}$ to Proposition \ref{prop:wt1_hatQ}, we obtain
\[\widehat{\xi}(1)+\pi i + \frac12 \sum_{l\ge1} t^l \overline{\widehat{\xi}(1+l)}=0,\]
which gives
\[\zeta_{\widehat{\mathcal{S}}}(1)+ \sum_{l\ge 0}t^l \, \zeta_{\widehat{\mathcal{S}}}(1+l)=0.\]
On the other hand, applying $\phi_{\widehat{\mathcal{A}}}$ to Proposition \ref{prop:wt1_hatQ}, we get \eqref{eq:wt1}, which is the same shape with the above. % follows from Theorem \ref{thm:hatA_from_q} and \eqref{eq:phi-Ahat-images}.

\begin{remark}
One of key ingredients of the study of relationships between $\zeta_{\widehat{\mathcal{A}}}(\kk)$ and $\zeta_{\widehat{\mathcal{S}}}(\kk)$ is the formula
\[H_{p-1}(\kk)=\sum_{a=0}^{\dep(\kk)} (-1)^{\wt(\kk_a)}  \sum_{\bl\in \Z_{\ge0}^a}  p^{\wt(\bl)}\, b\binom{\kk_a}{\bl}\zeta^{\rm De}_p (\overline{\kk_a+\bl}) \, \zeta^{\rm De}_p(\kk^a),\]
expressing the multiple harmonic sum $H_{p-1}(\kk)$ in terms of Deligne's $p$-adic multiple zeta values (see \cite[Definition 2.7]{Furusho07} for the definition of $ \zeta^{\rm De}_p(\kk)$, where they use the opposite convention, namely, our $\zeta^{\rm De}_p(\kk)$ corresponds to their $\zeta^{\rm De}_p(\overline{\kk})$).
Remark that the above formula was obtained by Akagi-Hirose-Yasuda \cite{AHY} and first proved by Jarossay \cite{Jarossay18}.
It is of interest to find a $q$-analogue of the above formula.
\end{remark}

%%%%%%%%%%%%%%%%%%%%%%%%%%%%%%%%%%
\section{Reversal, Duality and Cyclic sum formulas for $\widehat{\mathcal{Q}}$-MZVs}\label{sec:rel}
In this section, we will see that reversal, duality and cyclic sum formulas for both $\widehat{\mathcal{A}}$-MZVs and $\widehat{\mathcal{S}}$-MZVs can be extended to $\widehat{\mathcal{Q}}$-MZVs.
Proofs use well-known techniques in $q$-analogues, so are postponed to Appendix \ref{app:proofs_of_relations}.

%%%%%%%%%%%%%%%%%%%%%%%%%%%%%%%%%%
\subsection{Reversal formula}

Since $q^{p}=1-(1-q)[p]$, for all primes $p$ and $n\ge1$, we have $(q^{p})^{-1}\equiv \sum_{j=0}^{n-1}((1-q)[p])^{j}\mod [p]^n$.
Hence, letting
\begin{align*}
q^{\pm \pp}=(((q^{p})^{\pm 1} \mod [p]^{n})_{p})_{n} \in \widehat{\mathcal{Q}},
\end{align*}
we obtain
\begin{align*}
q^{\pp}=1-(1-q)[\pp], \qquad
q^{-\pp}=\sum_{l \ge 0}((1-q)[\pp])^{l}.
\end{align*}

\begin{theorem}\label{thm:q-reversal}
For any index $\kk$ and $\bullet\in\{\emptyset,\star\}$, it holds that
\begin{equation*}%\label{eq:q-reversal}
\begin{aligned}
\overline{\zeta}_{\widehat{\mathcal{Q}}}^\bullet(\kk)
&=(-q^{-\pp})^{\wt(\kk)}(q^{\pp})^{\dep(\kk)}
\sum_{\bl \in \mathbb{Z}_{\ge 0}^{\dep(\kk)}}(q^{-\pp}[\pp])^{\wt(\bl)}\,
b\binom{\kk}{\bl}
\zeta_{\widehat{\mathcal{Q}}}^\bullet(\overline{\kk+\bl}).
%\overline{\zeta}_{\widehat{\mathcal{Q}}}^{\star}(\kk)
%&=(-q^{-\pp})^{\wt(\kk)}(q^{\pp})^{\dep(\kk)}
%\sum_{\bl \in \mathbb{Z}_{\ge 0}^{\dep(\kk)}}(q^{-\pp}[\pp])^{\wt(\bl)}\,
%b\binom{\kk}{\bl}
%\zeta_{\widehat{\mathcal{Q}}}^\star(\overline{\kk+\bl}).
\end{aligned}
\end{equation*}
\end{theorem}

\begin{corollary}\label{cor:reversal}
For all index $\kk$, $\bullet\in\{\emptyset,\star\}$ and $\mathcal{F}\in\{\mathcal{A},\mathcal{S}\}$, we have
\begin{equation*}\label{eq:reversal-1}
\begin{aligned}
\zeta_{\widehat{\mathcal{F}}}^\bullet(\kk)&=(-1)^{\wt(\kk)}
\sum_{\bl \in \mathbb{Z}_{\ge 0}^{\dep(\kk)}}\Lambda^{\wt(\bl)}\,
b\binom{\kk}{\bl}
\zeta_{\widehat{\mathcal{F}}}^\bullet(\overline{\kk+\bl}),
\end{aligned}
\end{equation*}
where $\Lambda$ is defined by \eqref{eq:def-Lambda}.
We also have
\begin{equation*}\label{eq:reversal-2}
\begin{aligned}
\overline{\widehat{\xi}^\bullet(\kk)}&=(-1)^{\wt(\kk)}
\sum_{\bl \in \mathbb{Z}_{\ge 0}^{\dep(\kk)}}t^{\wt(\bl)}\,
b\binom{\kk}{\bl}
\widehat{\xi}^\bullet(\overline{\kk+\bl}),
\end{aligned}
\end{equation*}
where the bar on the left means the complex conjugate.
\end{corollary}

\begin{remark}
The mod $[p]$ version of Theorem \ref{thm:q-reversal} is proved in
\cite[Theorem 3.1]{HPT17} with the same method.
Corollary \ref{cor:reversal} for $\mathcal{F}=\mathcal{A}$ is proved in \cite[Theorem 4.1]{Rosen15}.
Corollary \ref{cor:reversal} for $\zeta_{\widehat{\mathcal{S}}}$ is
a special case of the shuffle relation,
which can be found in, e.g., \cite{Jarossay19} and \cite{OnoSekiYamamoto}.
\end{remark}

\subsection{Duality formula}

%For an index $\kk=(k_1,\ldots,k_d)$, its Hoffman's dual,
%denoted by $\kk^\vee$, is the index obtained
%by writing each component of 1 and then
%interchanging commas `,' and plus signs `$+$'.
%For example, if $\kk=(2, 3, 1)=(1+1, 1+1+1, 1)$,
%then $\kk^{\vee}=(1,1+1,1,1,1+1)=(1,2,1,2)$.

For all odd primes $p$ and $n\ge1$, we have
\[q^{p(p+1)/2}=(1-(1-q)[p])^{(p+1)/2}\equiv
\sum_{l=0}^{n-1} \binom{\frac{p+1}{2}}{l} (-(1-q)[p])^l \mod [p]^n.\]
Let
\begin{align*}
q^{\pp(\pp+1)/2}=((q^{p(p+1)/2} \mod [p]^{n})_{p})_{n} \in
\widehat{\mathcal{Q}}.
\end{align*}

\begin{theorem}\label{thm:Hoffman's-duality}
For any index $\kk$, it holds that
\begin{align*}
q^{\pp(\pp+1)/2}
\sum_{l\ge 0}[\pp]^{l}
\zeta_{\widehat{\mathcal{Q}}}^{\star}(\{1\}^{l}, \kk)+\sum_{l \ge 0}
(q^{-\pp}[\pp])^{l}\, \overline{\zeta}_{\widehat{\mathcal{Q}}}^{\star}(\{1\}^{l}, \kk^{\vee})=0,
\end{align*}
where $\{1\}^l=\underbrace{1,\ldots,1}_l$.
\end{theorem}

\begin{corollary}\label{cor:duality_F}
Let $\kk$ be an index. We have
\begin{align}
\sum_{l\ge 0}\Lambda^{l}\,
\zeta_{\widehat{\mathcal{F}}}^{\star}(\{1\}^{l}, \kk)+\sum_{l \ge 0}
\Lambda^{l}\, \zeta_{\widehat{\mathcal{F}}}^{\star}(\{1\}^{l}, \kk^{\vee})=0
\label{eq:Hoffman-dual-1}
\end{align}
for $\mathcal{F}=\mathcal{A}$ and $\mathcal{S}$,
where $\Lambda$ is given by \eqref{eq:def-Lambda},
and
\begin{align}
e^{\pi i t}\sum_{l\ge 0}t^{l}\,\widehat{\xi}^{\star}(\{1\}^{l}, \kk)+\sum_{l \ge 0}
t^{l}\, \overline{\widehat{\xi}^{\star}(\{1\}^{l}, \kk^{\vee})}=0.
\label{eq:Hoffman-dual-2}
\end{align}
\end{corollary}

\begin{remark}
The duality formula \eqref{eq:Hoffman-dual-1} for $\widehat{\mathcal{A}}$-MZVs
is proved by Seki \cite{Seki19}, and
its proof is simplified by Shuji Yamamoto.
Our proof of Theorem \ref{thm:Hoffman's-duality} is a $q$-analogue of
the simplified proof.
We thank Shin-ichiro Seki and Shuji Yamamoto for informing their results to us.
The duality formula \eqref{eq:Hoffman-dual-2}
is announced by Minoru Hirose.
\end{remark}

\subsection{Cyclic sum formula}

Let $\mathfrak{S}_{d}$ be the symmetric group on a set of $d$ elements.
The subgroup of $\mathfrak{S}_{d}$ generated by
the cyclic permutation $\sigma=(1, 2, \ldots , d)$ acts on the set of
indices of weight $k$ and depth $d$ by
$\sigma (k_{1}, k_{2}, \ldots , k_{d})=(k_{2}, \ldots , k_{d}, k_{1})$.
We denote the set of orbits of the action by $\Pi(k, d)$.
For an orbit $\alpha \in \Pi(k,d)$ we denote its cardinality by $|\alpha|$.

\begin{theorem}\label{thm:cyclic-sum}
For any orbit $\alpha \in \Pi(k, d)$ it holds that
\begin{equation}\label{eq:cyc_sum_q}
\begin{aligned}
&\sum_{\kk \in \alpha}
\sum_{s=0}^{\kk_{1}-2}
\zeta_{\widehat{\mathcal{Q}}}(\kk_{1}-s, \kk^1, s+1)\\
&=\sum_{\kk \in \alpha}
\zeta_{\widehat{\mathcal{Q}}}(\kk_{1}+1, \kk^1)+\sum_{\kk \in \alpha}\sum_{l\ge0}
(q^{-\pp}[\pp])^{l}\bigg\{
\zeta_{\widehat{\mathcal{Q}}}(\kk^1, \kk_{1},l+1)+
\zeta_{\widehat{\mathcal{Q}}}(\kk^1, \kk_{1}+ l+1)
\bigg\} \\
&+(1-q)\sum_{\kk \in \alpha}\sum_{l\ge0}(q^{-\pp}[\pp])^{l}
\zeta_{\widehat{\mathcal{Q}}}(\kk^1, \kk_{1}+l)
\end{aligned}
\end{equation}
and
\begin{equation}\label{eq:cyc_sum_q_star}
\begin{aligned}
&
\sum_{\kk \in \alpha}\sum_{s=0}^{\kk_{1}-2}
\zeta_{\widehat{\mathcal{Q}}}^{\star}(\kk_{1}-s, \kk^{1}, s+1) \\
&=\frac{k}{d}|\alpha|
\zeta_{\widehat{\mathcal{Q}}}^{\star}(k+1)+
|\alpha|\sum_{j=1}^{d}
\left(\frac{k}{j}-1\right)\binom{d}{j}
(1-q)^{j} \zeta_{\widehat{\mathcal{Q}}}^\star(k+1-j)\\
&+\sum_{\kk \in \alpha}\sum_{l\ge 0}(q^{-\pp}[\pp])^{l}
\zeta_{\widehat{\mathcal{Q}}}^{\star}(\kk^{1}, \kk_{1},l+1).
\end{aligned}
\end{equation}
\end{theorem}

%Applying the maps $\phi_{\widehat{\mathcal{A}}}$ and $\phi_{\widehat{\mathcal{S}}}$,
%we obtain the following corollary.
\begin{corollary}\label{cor:cyclic_sum}
For any $\alpha \in \Pi(k, d)$ and $\mathcal{F}\in\{\mathcal{A},\mathcal{S}\}$ we have
\begin{align*}
&
\sum_{\kk \in \alpha}\sum_{s=0}^{\kk_{1}-2}
\zeta_{\widehat{\mathcal{F}}}(\kk_{1}-s, \kk^{1}, s+1) \\
&=\sum_{\kk \in \alpha}
\zeta_{\widehat{\mathcal{F}}}(\kk_{1}+1, \kk^{1})+\sum_{\kk \in \alpha}
\sum_{l\ge 0}\Lambda^{l}
\left(
\zeta_{\widehat{\mathcal{F}}}(\kk^{1}, \kk_{1},l+1)+
\zeta_{\widehat{\mathcal{F}}}(\kk^{1}, \kk_{1}+l+1)
\right)
\end{align*}
and
\begin{align*}
\sum_{\kk \in \alpha}\sum_{s=0}^{\kk_{1}-2}
\zeta_{\widehat{\mathcal{F}}}^{\star}(\kk_{1}-s, \kk^{1}, s+1)=\frac{k}{d}\left|\alpha\right|
\zeta_{\widehat{\mathcal{F}}}^\star(k+1)+\sum_{\kk \in \alpha}
\sum_{l\ge 0}\Lambda^{l}
\zeta_{\widehat{\mathcal{F}}}^{\star}(\kk^{1}, \kk_{1},l+1),
\end{align*}
where $\Lambda$ is defined by \eqref{eq:def-Lambda}.
The above equality with $\zeta_{\widehat{\mathcal{F}}}^{\bullet}$ and $\Lambda$ replaced
by $\widehat{\xi}^{\bullet}$ and $t$, respectively,
also holds for $\bullet \in \{\emptyset, \star\}$.
\end{corollary}

\begin{remark}
Our proof of Theorem \ref{thm:cyclic-sum} is a $q$-analogue of
the proof of Corollary \ref{cor:cyclic_sum} for $\mathcal{F}=\mathcal{A}$
due to Kawasaki \cite{Kawasaki} (see also \cite{KawasakiOyama}).
Corollary \ref{cor:cyclic_sum} for $\mathcal{F}=\mathcal{S}$ is shown in \cite{HiroseMuraharaOno} and Corollary \ref{cor:cyclic_sum} for $\widehat{\xi}^\bullet \mod t$ is announced by Nobuo Sato and Minoru Hirose.
\end{remark}

%%%%%%%%%%%%%%%%%%%%%%%%%%%%%%%%%%
\section{Discussions on dimensions}\label{sec:observations}

%%%%%%%%%%%%%%%%%%%%%%%%%%%%%%%%%%
\subsection{Enumeration of relations among $\mathcal{Q}$-MZVs}
This subsection discusses relations of $\zeta_{\mathcal{Q}}(\kk)$'s of the form \eqref{eq:rel's_Q-MZV}, based on some experimental works.

For $k\ge1$, we denote by $\mathbb{I}_k$ the set of all indices of weight $k$.
Consider the $\Q$-vector space $\mathcal{Z}^{\mathcal{O}}_k$ spanned by the set
\[ \{ \pp^h(1-q)^j \zeta_{\mathcal{Q}}(\kk)\mid 0\le h\le j\le k,\ \kk\in \mathbb{I}_{k-j}\} ,\]
which is a $\Q$-vector subspace of the $\Q[\pp]$-module generated by elements $(1-q)^j \zeta_{\mathcal{Q}}(\kk)$ of weight $k$.
For example, the space $\mathcal{Z}^{\mathcal{O}}_2$ is generated by
\[ \zeta_{\mathcal{Q}}(2), \zeta_{\mathcal{Q}}(1,1),(1-q)\zeta_{\mathcal{Q}}(1),\pp (1-q)\zeta_{\mathcal{Q}}(1),(1-q)^2,\pp(1-q)^2,\pp^2(1-q)^2.\]
We let $\mathcal{Z}^{\mathcal{O}}_0=\Q$ as usual.
The reason why we consider the $\Q$-vector space $\mathcal{Z}^{\mathcal{O}}_k$, not the $\Q[\pp]$-module, is that the whole space
\[\mathcal{Z}^{\mathcal{O}}=\sum_{k\ge0} \mathcal{Z}^{\mathcal{O}}_k\]
with the product given by the $q$-stuffle product (see \cite{Bradley05}) is a subalgebra of $\mathcal{O}$ defined in \S\ref{sec:QtoAS}. % (remember that $p^h(1-q)^j\big|_{q=e^{2\pi i/p}}=O(p^{h-j})$ as $p\rightarrow \infty$).
Moreover, by Theorem \ref{thm:AS_from_q} (and \cite[Proposition 16]{Hirose}), we have
\[\phi_{\mathcal{A}} \big(\mathcal{Z}^{\mathcal{O}}_k\big)= \mathcal{Z}^{\mathcal{A}}_k \quad \mbox{and} \quad \phi_{\mathcal{S}} \big(\mathcal{Z}^{\mathcal{O}}_k\big)= \mathcal{Z}_k\oplus \pi i\mathcal{Z}_{k-1}.\]

Our numerical implementation (by \cite{PARI}) is as follows.
Let us fix a finite set $S$ of primes.
For each index $\kk$ and $0\le h\le j$, we first compute a polynomial $n(h,j,\kk;q) \in \Q[q]\big/ \big(\prod_{p\in S}[p]\big)$ such that $n(h,j,\kk;q)\equiv p^h (1-q)^j H_{p-1}(\kk;q)\mod [p]$ for all $p\in S$ (use the Chinese remainder theorem), and then, find numerical $\Q$-linear relations among numerical values $n(h,j,\kk;q)$'s at $q$ being a fixed transcendental number.
These numerical data provide possible $\Q$-linear relations of the generator of $\mathcal{Z}^{\mathcal{O}}_k$, and in this way, we can count the number of all linearly independent (possible) relations over $\Q$.
For simplicity, letting $V_k = (1-q)\mathcal{Z}^{\mathcal{O}}_{k-1}+\pp (1-q) \mathcal{Z}^{\mathcal{O}}_{k-1}\subset \mathcal{Z}^{\mathcal{O}}_k$, we below give a numerical dimension of
\[  \widetilde{\mathcal{Z}}^{\mathcal{O}}_k = \mathcal{Z}^{\mathcal{O}}_k\big/ V_k.\]
%The result, which should be compared with dimensions of $\mathcal{Z}^{\mathcal{A}}_k$, is as follows.

\begin{table}[hbtp]
  \caption{Table of numerical dimensions}
  \label{dim}
$\begin{array}{c|cccccccccccccccccccccccccc}
k & 1&2&3&4&5&6&7&8&9&10&11&12\\ \hline
%\sharp\{\kk \mid \wt(\kk)\le k\} & 1&2&2^2&2^3&2^4&2^5&2^6&2^7&2^8&2^9&2^{10}&2^{11}&2^{12}\\
\dim_{\Q} \widetilde{\mathcal{Z}}^{\mathcal{O}}_k   & 0&0&1&0&2&1&3&4 &5 &10 &11 &19\\
%\tilde{d}_k^{\mathcal{Q}}& 1&0&0&1&0&2&1&3&4&5&10&11&19\\
\dim_{\Q} \mathcal{Z}^{\mathcal{A}}_k  & 0&0&1&0&1&1&1&2&2&3&4&5\\
\end{array}$
\end{table}

From the above table, since $\phi_{\mathcal{A}}\big(V_k\big)=\{0\}$ and $\phi_{\mathcal{S}}\big(V_k\big)\subset \pi i \mathcal{Z}_{k-1}$, we see that up to weight 4
all linear relations among $\mathcal{A}$-MZVs and $\mathcal{S}$-MZVs are obtained from
linear relations of $\pp^h(1-q)^j \zeta_{\mathcal{Q}}(\kk)$'s.
In weight 5, there is a relation which does not come from relations of $\zeta_{\mathcal{Q}}(\kk)$'s.
Such relation is already detected in \cite[\S3.4]{BTT18}; for $\mathcal{F}\in \{\mathcal{A},\mathcal{S}\}$, it is
\begin{equation}\label{eq:missing_weight5}
\zeta_{\mathcal{F}}(4,1)-2\zeta_{\mathcal{F}}(3,1,1)=0.
\end{equation}
In other words, we expect that $\zeta_{\mathcal{Q}}(4,1)-2\zeta_{\mathcal{Q}}(3,1,1)\stackrel{?}{\not\in}(1-q)\mathcal{Z}^{\mathcal{O}}_{4} +\pp (1-q) \mathcal{Z}^{\mathcal{O}}_4$.

%%%%%%%%%%%%%%%%%%%%%%%%%%%%%%%%%%
\subsection*{Examples of relations}

Using results we obtained in \S\ref{sec:rel}, we illustrate some examples of explicit relations and
compute upper bounds of the dimension of $\widetilde{\mathcal{Z}}^{\mathcal{O}}_k$ for $k=1,2,3$.
We remark that for each $k\ge1$ there exists a polynomial $c_k(x)\in \Q[x]$ of degree at most $k$ such that $\zeta_{\mathcal{Q}}(k)= c_k(\pp)(1-q)^k $
(see, e.g., \cite[Corollary 9.5.5]{Zhao16} and \cite[Remark 2.3]{BTT18}).

In weight 1, the identity \eqref{eq:cong_wt1} shows
\begin{equation}\label{eq:wt1_including}
\zeta_{\mathcal{Q}}(1)\in V_1=(1-q)\mathcal{Z}_0^{\mathcal{O}}+\pp (1-q)\mathcal{Z}^{\mathcal{O}}_0,
\end{equation}
so $\dim\widetilde{\mathcal{Z}}^{\mathcal{O}}_1= 0$.

In weight 2, a generator of $\widetilde{\mathcal{Z}}^{\mathcal{O}}_2$ is given by $\{\zeta_{\mathcal{Q}}(2),\zeta_{\mathcal{Q}}(1,1)\}$ and these satisfy
\begin{align}
\label{eq:wt2dep1} 0&=\zeta_{\mathcal{Q}}(2)+\frac{\pp^2-1}{12}(1-q)^2,\\
\notag 0&=2\zeta_{\mathcal{Q}}(2)+\zeta_{\mathcal{Q}}(1,1)+(1-q)\zeta_{\mathcal{Q}}(1) ,
\end{align}
where the last relation is obtained from Theorem \ref{thm:cyclic-sum} for the case $k=d=1$.
These relations imply $\dim \widetilde{\mathcal{Z}}^{\mathcal{O}}_2=0$.

In weight 3 we get
\begin{align*}
0&=\zeta_{\mathcal{Q}}(3)-\frac{\pp^2-1}{24}(1-q)^3,\\
0&=\zeta_{\mathcal{Q}}(2,1)+\zeta_{\mathcal{Q}}(1,2)+\zeta_{\mathcal{Q}}(1,1,1)+(1-q)\zeta_{\mathcal{Q}}(1,1),\\
0&=\zeta_{\mathcal{Q}}(3)+\zeta_{\mathcal{Q}}(2,1)+\zeta_{\mathcal{Q}}(1,2)+(1-q)\zeta_{\mathcal{Q}}(2)+\frac{\pp^2-1}{12}(1-q)^2\zeta_{\mathcal{Q}}(1),
\end{align*}
where the second relation is the case $k=d=2$ in Theorem \ref{thm:cyclic-sum} modulo $[\pp]$ and
the last relation can be deduced from \eqref{eq:wt2dep1}
by multipling by $\zeta_{\mathcal{Q}}(1)$ and using the $q$-stuffle product
\begin{align*}
\zeta_{\mathcal{Q}}(k_1)\zeta_{\mathcal{Q}}(k_2)=
\zeta_{\mathcal{Q}}(k_1,k_2)+\zeta_{\mathcal{Q}}(k_2,k_1)+\zeta_{\mathcal{Q}}(k_1+k_2)+
(1-q)\zeta_{\mathcal{Q}}(k_1+k_2-1).
\end{align*}
Hence, one gets $\dim \widetilde{\mathcal{Z}}^{\mathcal{O}}_3\le 1$.

\begin{remark}
Let us count the number of linearly independent relations obtained from the duality formula \eqref{eq:duality_Q} and the $q$-stuffle product.
For this, define the \emph{stuffle-star product}
$\star : \mathfrak{H}^1\otimes \mathfrak{H}^1 \rightarrow \mathfrak{H}^1$
inductively by
\begin{equation*}
y_kw\star y_lw'=y_k(w \star y_lw')+y_l(y_kw \star w')-y_{k+l}(w\star w')
\end{equation*}
for $w,w' \in \mathfrak{H}^1$ and $k,l\ge1$, with the initial condition $w \star 1=w=1\star w$.
By the $q$-stuffle-star product (see \cite[\S3.2]{BTT18} for the definition), for indices $\kk$ and $\boldsymbol{l}$ with $\wt(\kk)+\wt(\boldsymbol{l})=k$ it holds that
\[ \zeta_{\mathcal{Q}}^\star(y_{\kk})\zeta_{\mathcal{Q}}^\star(y_{\boldsymbol{l}})- \zeta_{\mathcal{Q}}^\star(y_{\kk}\star y_{\boldsymbol{l}})\in (1-q)\mathcal{Z}^{\mathcal{O}}_{k-1},\]
where $\zeta_{\mathcal{Q}}^\star$ is regarded as a $\Q$-linear map that sends $y_{\kk}$ to $\zeta_{\mathcal{Q}}^\star(\kk)$.
Hence, by \eqref{eq:duality_Q}, for any indices $\kk$ and $\boldsymbol{l}$ with $\wt(\kk)+\wt(\boldsymbol{l})=k$, one obtains
\[ \zeta_{\mathcal{Q}}^\star\big(y_{\kk}\star y_{\boldsymbol{l}} - (-1)^k y_{\overline{\kk^\vee}}\star y_{\overline{\boldsymbol{l}^\vee}}\big)\in (1-q)\mathcal{Z}^{\mathcal{O}}_{k-1}.\]
Furthermore, using \eqref{eq:wt1_including}, for any index $\kk$ of weight $k-1$, we see that
\[ \zeta_{\mathcal{Q}}^\star\big(y_{1}\star y_{\kk}\big) \in V_k.\]
Now consider the $\Q$-vector subspace of $\mathfrak{H}^1_k=\langle y_{\kk}\mid \kk\in\mathbb{I}_k\rangle_{\Q}$ given by
\begin{align*}
\mathfrak{R}_k&=\left\langle y_{\kk}\star y_{\boldsymbol{l}} - (-1)^k y_{\overline{\kk^\vee}}\star y_{\overline{\boldsymbol{l}^\vee}} \, \middle|\, (\kk,\boldsymbol{l})\in \bigcup_{l=1}^{k-1} \mathbb{I}_l\times \mathbb{I}_{k-l}\right\rangle_{\Q}\\
&+\langle y_{1}\star y_{\kk}\mid \kk\in \mathbb{I}_{k-1}\rangle_{\Q} +\left\langle y_{\kk}+(-1)^{\wt(\kk)} y_{\overline{\kk^\vee}} \, \middle|\, \kk \in \mathbb{I}_{k}\right\rangle_{\Q}.
\end{align*}
Since $\zeta_{\mathcal{Q}}^\star(w)\in V_k$ for any $w\in \mathfrak{R}_k$, it follows that
\[\dim \widetilde{\mathcal{Z}}_k^{\mathcal{O}} \le \dim \mathfrak{H}^1_k\big/ \mathfrak{R}_k .\]
The exact dimension of the right side can be computed up to certain weights and the list, which coincides with the above numerical dimension of $\widetilde{\mathcal{Z}}_k^{\mathcal{O}}$, is as follows.
\begin{table}[hbtp]
  \caption{Table of numerical dimensions}
  \label{dim}
$\begin{array}{c|cccccccccccccccccccccccccc}
k & 1&2&3&4&5&6&7&8&9&10&11&12\\ \hline
%\sharp\{\kk \mid \wt(\kk)\le k\} & 1&2&2^2&2^3&2^4&2^5&2^6&2^7&2^8&2^9&2^{10}&2^{11}&2^{12}\\
 \dim_{\Q} \mathfrak{H}^1_k\big/ \mathfrak{R}_k    & 0&0&1&0&2&1&3&4 &5 &10 &11 &19\\
%\tilde{d}_k^{\mathcal{Q}}& 1&0&0&1&0&2&1&3&4&5&10&11&19\\
%\dim_{\Q} \mathcal{Z}^{\mathcal{A}}_k  & 0&0&1&0&1&1&1&2&2&3&4&5\\
\end{array}$
\end{table}
%As pointed out by Julian Rosen, if $\Q[\pp]$-module $\mathcal{Z}^{\mathcal{Q}}_k\big/ (1-q)\mathcal{Z}^{\mathcal{Q}}_{k-1}$ is free, then its rank over $\Q[\pp]$ coincides with the dimension of $\mathcal{Z}^{\mathcal{Q}}_k\big/\big((1-q)\mathcal{Z}^{\mathcal{Q}}_{k-1} +\pp \mathcal{Z}^{\mathcal{Q}}_k\big)$ over $\Q$.
\end{remark}

%%%%%%%%%%%%%%%%%%%%%%%%%%%%%%%%%%
\subsection{Variants of $\mathcal{Q}$-MZVs}\label{subsec:variants}
In this subsection, we introduce variants of $\zeta_{\mathcal{Q}}(\kk)$ and discuss how we obtain \eqref{eq:missing_weight5} from our theory of $\mathcal{Q}$-MZVs.

For an index $\kk=(k_1,\ldots,k_d)$ and $\boldsymbol{s}=(s_1,\ldots,s_d)\in \Z_{\ge0}^d$, we set
\[H_m(\kk;\boldsymbol{s};q)=\sum_{m\ge m_1>\cdots>m_d>0}
\prod_{a=1}^d \frac{q^{s_am_a}}{[m_a]^{k_a}},\]
and define
\[ \zeta_{\mathcal{Q}}(\kk;\boldsymbol{s})=\big( H_{p-1}(\kk;\boldsymbol{s};q)\mod [p]\big)_p\in \mathcal{Q}.\]
Note that $\zeta_{\mathcal{Q}}(k_1,\ldots,k_d;k_1-1,\ldots,k_d-1)=\zeta_{\mathcal{Q}}(k_1,\ldots,k_d)$ and $\zeta_{\mathcal{Q}}(k_1,\ldots,k_d;1,\ldots,1)=\overline{\zeta}_{\mathcal{Q}}(k_1,\ldots,k_d)$ .
For $k\ge1$, let $\mathcal{Z}^{\mathcal{Q}}_k$ be the $\Q$-vector space spanned by the set
\[ \{ \pp^h(1-q)^j \zeta_{\mathcal{Q}}(\kk;\boldsymbol{s})\mid 0\le h\le j\le k,\ \kk\in\mathbb{I}_{k-j},\ \boldsymbol{s}\in \Z^{\dep(\kk)}_{\ge0} ,\ \boldsymbol{s}\le \kk\} ,\]
where $(s_1,\ldots,s_d)\le (k_1,\ldots,k_d)$ means $s_i\le k_i$ for all $1\le i\le d$.
For example, the space $\mathcal{Z}^{\mathcal{Q}}_2$ is generated by
\begin{gather*}
 \zeta_{\mathcal{Q}}(2;0),\zeta_{\mathcal{Q}}(2;1),\zeta_{\mathcal{Q}}(2;2), \zeta_{\mathcal{Q}}(1,1;0,0),\zeta_{\mathcal{Q}}(1,1;0,1),\zeta_{\mathcal{Q}}(1,1;1,0),\zeta_{\mathcal{Q}}(1,1;1,1),\\
 (1-q)\zeta_{\mathcal{Q}}(1;0), \pp (1-q)\zeta_{\mathcal{Q}}(1;0),(1-q)\zeta_{\mathcal{Q}}(1;1),\pp (1-q)\zeta_{\mathcal{Q}}(1;1),\\
 (1-q)^2,\pp(1-q)^2,\pp^2(1-q)^2.
 \end{gather*}
As is indicated by the upper $\mathcal{Q}$, $\mathcal{Z}^{\mathcal{Q}}_k$ is not in $\mathcal{O}$ (for one, from \cite[Remark 2.11]{BTT18} we see that $\phi_{\mathcal{S}}(\zeta_{\mathcal{Q}}(1,1;0,1)) =\infty$),
while $\phi_{\mathcal{A}}\big(\zeta_{\mathcal{Q}}(\kk;\boldsymbol{s})\big)=\zeta_{\mathcal{Q}}(\kk)$ holds for any $\boldsymbol{s}\in \Z_{\ge0}^{\dep(\kk)}$.

Using these objects, we can observe
\[ \zeta_{\mathcal{Q}}(4,1;3,0)-\zeta_{\mathcal{Q}}(3,1,1;2,1,0)-\zeta_{\mathcal{Q}}(3,1,1;2,0,1)\stackrel{?}{\in}(1-q)\mathcal{Z}^{\mathcal{Q}}_{4} +\pp (1-q) \mathcal{Z}^{\mathcal{Q}}_4,\]
from which we can reprove \eqref{eq:missing_weight5} for the case $\mathcal{F}=\mathcal{A}$.
Although some of individual terms in the above relation do not converge under the map $\phi_{\mathcal{S}}$, we can check that main terms in the asymptotic expansion of the above relation are canceled and that the image under the map $\phi_{\mathcal{S}}$ (modulo $\pi i$) coincides with \eqref{eq:missing_weight5} for the case $\mathcal{F}=\mathcal{S}$.

%720 zetaQ[0, 0, {4, 1}, {3, 0}] -
% 720 zetaQ[0, 0, {3, 1, 1}, {2, 0, 1}] -
% 720 zetaQ[0, 0, {3, 1, 1}, {2, 1, 0}] -
% 720 zetaQ[0, 1, {2, 1, 1}, {0, 1, 0}] -
% 300 zetaQ[0, 2, {2, 1}, {1, 0}] -
% 1440 zetaQ[0, 2, {1, 1, 1}, {1, 0, 0}] -
% 1050 zetaQ[0, 3, {1, 1}, {1, 0}] + 360 zetaQ[1, 3, {1, 1}, {1, 0}] -
% 6 zetaQ[1, 5, {}, {}] + 60 zetaQ[2, 2, {2, 1}, {1, 0}] -
% 30 zetaQ[2, 3, {1, 1}, {1, 0}] + 5 zetaQ[2, 5, {}, {}] +
% 5 zetaQ[3, 5, {}, {}] - 5 zetaQ[4, 5, {}, {}] + zetaQ[5, 5, {}, {}]

For comparison, we exhibit the numerical dimension of
\[  \widetilde{\mathcal{Z}}^{\mathcal{Q}}_k = \mathcal{Z}^{\mathcal{Q}}_k\big/ \big((1-q)\mathcal{Z}^{\mathcal{Q}}_{k-1}+\pp (1-q) \mathcal{Z}^{\mathcal{Q}}_{k-1}\big).\]
Here is a table of dimensions of $\widetilde{\mathcal{Z}}^{\mathcal{Q}}_k$ over $\Q$ up to weight 5.
\begin{center}
$\begin{array}{c|cccccccccccccccccccccccccc}
k & 1&2&3&4&5&6&7 \\ \hline%&6&7&8&9&10&11&12\\ \hline
\dim_{\Q} \widetilde{\mathcal{Z}}^{\mathcal{Q}}_k   & 0&1&2&2&6&8&16%&1&3&4 &5 &10 &11 &19\\
\end{array}$
\end{center}

\begin{remark}
In a similar manner to the proof of Theorem \ref{thm:BTT} (ii), one can show that $\zeta_{\mathcal{Q}}(\kk;\{0\}^{\dep(\kk)})$ of weight $k$ lies in $\mathcal{Z}_k^{\mathcal{O}}$ (see also \cite[Theorem 3.8]{Tasaka21} with the case $N=1$).
Since
\[\frac{q^{sm}}{[m]^{k}}=\frac{q^{(s-1)m}}{[m]^{k}}-(1-q)\frac{q^{(s-1)m}}{[m]^{k-1}},\]
we see that elements $\zeta_{\mathcal{Q}}(\kk;\boldsymbol{s})$ of weight $k$ with $(\{0\}^{\dep(\kk)})\le \boldsymbol{s}< \kk$ lies in $\mathcal{Z}^{\mathcal{O}}_k$, where $(s_1,\ldots,s_d)< (k_1,\ldots,k_d)$ means $s_i<k_i$ for all $1\le i\le d$.
\end{remark}

%%%%%%%%%%%%%%%%%%%%%%%%%%%%%%%%%%
\subsection{Relations of $\mathcal{Q}_2$-MZVs}
In this subsection, we give examples of relations among $\zeta_{\mathcal{Q}_2}(\kk)$'s, which are defined by
\[ \zeta_{\mathcal{Q}_2}(\kk) = \big( H_{p-1}(\kk;q)\mod [p]^2\big )_p \in \mathcal{Q}_2\cong \widehat{\mathcal{Q}}\big/[\pp]^2\widehat{\mathcal{Q}}.\]

As examples of relations, from Proposition \ref{prop:wt1_hatQ}, we obtain
\[ \zeta_{\mathcal{Q}_2}(1)- \frac{\pp-1}{2}(1-q) + \frac{1}{2}[\pp] \overline{\zeta}_{\mathcal{Q}_2}(2)=0.\]
Also, we see from Theorems \ref{thm:q-reversal} and \ref{thm:Hoffman's-duality} that for each index $\kk$
\begin{gather*}
(-1)^{\wt(\kk)}\overline{\zeta}_{\mathcal{Q}_2}(\kk) - \zeta_{\mathcal{Q}_2}(\overline{\kk})  -  [\pp] \sum_{\substack{\boldsymbol{l}\in \Z_{\ge0}^{\dep(\kk)}\\\wt(\boldsymbol{l})=1}} b\binom{\kk}{\boldsymbol{l}}\zeta_{\mathcal{Q}_2}(\overline{\kk+\boldsymbol{l}})=0,\\
 \zeta_{\mathcal{Q}_2}^\star(\kk)+\zeta_{\mathcal{Q}_2}^\star(\kk^\vee)+[\pp]\bigg( \zeta_{\mathcal{Q}_2}^\star(1,\kk)+\zeta_{\mathcal{Q}_2}^\star(1,\kk^\vee)-\frac{\pp+1}{2}(1-q)\zeta_{\mathcal{Q}_2}^\star(\kk)\bigg)=0.
\end{gather*}

In order to count the number of relations, for $k\ge1$, let $\mathcal{Z}_k^{\mathcal{O}_2}$ be the $\Q$-vector space spanned by the sets
\[ \{ \pp^h(1-q)^j \zeta_{\mathcal{Q}_2}(\kk)\mid 0\le h\le j\le k,\ \kk\in \mathbb{I}_{k-j}\}\]
and
\[ \{ [\pp] \pp^h(1-q)^j \zeta_{\mathcal{Q}_2}(\kk)\mid 0\le h\le j\le k+1,\ \kk\in \mathbb{I}_{k+1-j}\}.\]
For simplicity, we consider
\[  \widetilde{\mathcal{Z}}^{\mathcal{O}_2}_k = \mathcal{Z}^{\mathcal{O}_2}_k\big/ \big((1-q)\mathcal{Z}^{\mathcal{O}_2}_{k-1}+\pp (1-q) \mathcal{Z}^{\mathcal{O}_2}_{k-1}\big).\]
The list of numerical dimensions of $\widetilde{\mathcal{Z}}^{\mathcal{O}_2}_k$ up to weight 7, which is compared with the dimension of the $\Q$-vector space $\mathcal{Z}_k^{\mathcal{A}_2}$ spanned by $\zeta_{\mathcal{A}_2}(\kk)=(H_{p-1}(\kk)\mod p^2)_p\in \mathcal{A}_2$ of weight $k$ (see \cite[p.261]{Zhao16}), is as follows.

\begin{center}
$\begin{array}{c|cccccccccccccccccccccccccc}
k & 1&2&3&4&5&6&7 \\ \hline
\dim_{\Q} \widetilde{\mathcal{Z}}^{\mathcal{O}_2}_k   & 0&1&1&2&3&4&7\\
\dim_{\Q} {\mathcal{Z}}^{\mathcal{A}_2}_k   & 0&1&1&1&2&2&3
\end{array}$
\end{center}

%%%%%%%%%%%%%%%%%%%%%%%%%%%%%%%%%%%%%%%%%%%%%%%%%%%%%%%%%%%%%%%%%%%%%%%%%%%%%%%%%%%%%%%%%%%%%%%%%%%%%%%%%%%%%%%%%%%%%%%%%%%%%%
\appendix

%%%%%%%%%%%%%%%%%%%%%%%%%%%%%%%%%%
\section{Proof of Theorem \ref{thm:hatS_from_q}}\label{app:proof_of_main_result}

In this appendix, we give a proof of Theorem \ref{thm:hatS_from_q}.
Our first goal is to prove
\begin{equation}\label{eq:lim_formula}
\lim_{m\rightarrow \infty}  H_{m-1}(\kk;q_m(t)) = \widehat{\xi}(\kk),
\end{equation}
from which the equation \eqref{eq:hatS_from_q} for $\bullet =\emptyset$ follows, where $q_m(t)$ was defined to be the solution to \eqref{eq:qmt}.
The equation \eqref{eq:hatS_from_q} for $\bullet =\star$ is also obtained from \eqref{eq:lim_formula}, together with \eqref{eq:partial_franction_q} and Proposition \ref{prop:1-q}, so its proof is omitted.
Finally, we prove \eqref{eq:hatS_from_q2} in much the same way.
Throughout this section, for simplicity of notations, we write
\[\zeta_m=e^{2\pi i/m}.\]

%%%%%%%%%%%%%%%%%%%%%%%%%%%%%%%%%%
\subsection{Asymptotic formulas}
Note that for each index $\kk$, one obtains the power series
\[  H_{m-1}(\kk;q_m(t))=\sum_{l \ge 0}\alpha_{l}(\kk; m)t^{l}\in\Q(\zeta_m)[[t]].\]
Our first task is to give an asymptotic formula for $\alpha_l(\kk;m)$ as $m\rightarrow \infty$.

%For each index $\kk$, one obtains the power series
%\[  H_{m-1}(\kk;q_m(t))\in\Q(\zeta_m)[[t]].\]
%To compute the limit $\displaystyle\lim_{m\rightarrow\infty} H_{m-1}(\kk;q_m(t))$, let $\alpha_l(\kk;m)$ be
For this, we remark that for each index $\kk$ the formula
\begin{equation}\label{eq:BTT}
\begin{aligned}
\alpha_0(\kk;m)&=H_{m-1}(\kk;\zeta_m)\\
&= \sum_{a=0}^{\dep(\kk)}(-1)^{\wt(\kk_a)}  \zeta^\ast\left(\overline{\kk_a};\overline{\gamma_m}\right)  \zeta^\ast\left(\kk^a;\gamma_m\right)+ O\left(\frac{\log^J m}{m} \right)
\end{aligned}
\end{equation}
is shown in \cite[\S2.3.2]{BTT18},
where $\gamma_{m}$ is defined
in terms Euler's constant $\gamma$ by
\begin{align*}
\gamma_m=\log\left(\frac{m}{\pi}\right)+\gamma -\frac{\pi i }{2}
\end{align*}
and $\overline{\gamma_{m}}$ is its complex conjugate.
We consider similar asymptotics for $l\ge1$.

Let us begin with a general result on explicit formulas for the coefficients of $t^l \ (l\ge0)$
in the Taylor expansion of $f(q_m(t))$ at $t=0$ for
any holomorphic function $f(q)$ at $q=\zeta_m$.

\begin{proposition}\label{prop:expansion}
For $l, j \ge 0$, we define the rational number $B_{l, j}$ by
\begin{align*}
 B_{l, j}=\frac{1}{j!}\frac{d^{j}}{dy^{j}}\left(\frac{y}{e^{y}-1}\right)^{l}\bigg|_{y=0}
\end{align*}
and set
\begin{align}
S_{l,j}(m)=\frac{1}{j!}\frac{d^{j}}{dy^{j}}(\zeta_{m} e^{y/m}-1)^{l}\bigg|_{y=0}.
\label{eq:def-SNL}
\end{align}
Suppose that a function $f(q)$ is holomorphic in a neighborhood of $q=\zeta_{m}$.
Set $f(q_{m}(t))=\sum_{l \ge 0}a_{l}t^{l}$, where $a_{l} \in \C$.
Then $a_{0}=f(\zeta_{m})$ and, for $l\ge 1$, it holds that
\begin{align*}
a_{l}=\frac{1}{l}\sum_{\substack{j_{1}+j_{2}+j_{3}=l-1 \\ j_{1}, j_{2}, j_{3} \ge 0}}
B_{l, j_{1}}\, S_{l, j_{2}}(m)\, \frac{1}{j_{3}!m^{j_{3}+1}} (\theta_{q}^{j_{3}+1}f)(\zeta_{m}),
\end{align*}
where $\theta_{q}$ is the theta operator
\begin{align*}
\theta_{q}=q\frac{d}{dq}.
\end{align*}
\end{proposition}

\begin{proof}
By definition we have $a_{0}=f(q_{m}(0))=f(\zeta_{m})$.
Suppose that $l \ge 1$.
Consider the function $\psi_{m}(y)=(1-e^{y})/(1-\zeta_{m}e^{y/m})$.
Since $d[m]_q/dq$ is non-zero at $q=\zeta_m$, we have
$q_{m}(\psi_m(y))=\zeta_{m}e^{y/m}$.
Note that $\psi_{m}(0)=0$ and $\psi_{m}'(0)\not=0$.
Changing the variable $t$ to $y$ by $t=\psi_{m}(y)$, we see that
\begin{align*}
a_{l}=
\mathrm{Res}_{t=0}\, t^{-l-1}f(q_{m}(t)) \, dt=
\mathrm{Res}_{y=0}\, \psi'_{m}(y)\psi_{m}(y)^{-l-1}f(\zeta_{m}e^{y/m})\, dy.
\end{align*}
The right side is equal to
\begin{align*}
\mathrm{Res}_{y=0}\, (-\psi_{m}(y)^{-l}/l)' \, f(\zeta_{m}e^{y/m})\, dy=
\frac{1}{l}\,
\mathrm{Res}_{y=0}\, \left\{
\psi_{m}(y)^{-l} \, \left(\frac{d}{dy}f(\zeta_{m}e^{y/m})\right)
\right\}dy
\end{align*}
Since $\psi_{m}(y)^{-l}$ has an $l$-th order pole at $y=0$, it is equal to
\begin{align*}
&
\frac{1}{l!}\frac{d^{l-1}}{dy^{l-1}}
\left\{
y^{l} \left(\frac{\zeta_{m}e^{y/m}-1}{e^{y}-1}\right)^{l} \,
\left( \frac{d}{dy} f(\zeta_{m} e^{y/m}) \right)
\right\}\bigg|_{y=0} \\
&=\frac{1}{l}
\sum_{\substack{j_{1}+j_{2}+j_{3}=l-1 \\ j_{1}, j_{2}, j_{3} \ge 0}}
B_{l, j_{1}}\, S_{l, j_{2}}(m)\,\frac{1}{j_{3}!}
\left( \frac{d^{j_{3}+1}}{dy^{j_{3}+1}} f(\zeta_{m}e^{y/m}) \right)\bigg|_{y=0}.
\end{align*}
Now the desired formula follows from the relation
\begin{align*}
\frac{d^{k}}{dy^{k}} f(\zeta_{m}e^{y/m})=\frac{1}{m^k}\, (\theta_{q}^{k}f)(\zeta_{m}e^{y/m})
\end{align*}
for $k \ge 0$.
\end{proof}

\begin{proposition}\label{prop:1-q}
Set
\begin{align}
1-q_{m}(t)=-\frac{2\pi i}{m}\sum_{l \ge 0}\alpha_{l}(m)t^{l},
\label{eq:1-q}
\end{align}
where $\alpha_{l}(m) \in \C$.
Then it holds that
$\alpha_{0}(m)=1+O(m^{-1})$ and
$\alpha_{l}(m)=O(m^{-l})$ for $l \ge 1$ as $m \to \infty$.
\end{proposition}

For the proof, we use the following lemma.
\begin{lemma}\label{lem:SNL}
For $l \ge j \ge 0$, it holds that
\begin{align*}
S_{l, j}(m)=\frac{1}{m^l}\binom{l}{j}(2\pi i)^{l-j}(1+O(m^{-1})) \qquad
(m \to \infty).
\end{align*}
\end{lemma}

\begin{proof}
Substituting $q=\zeta_me^{y/m}$ into \eqref{eq:def-SNL}, we have
\begin{align*}
S_{l, j}(m)=\frac{1}{j! m^j}\theta_{q}^{j}(q-1)^{l}\big|_{q=\zeta_{m}}.
\end{align*}
It holds that
\begin{align*}
\theta_{q}^{j}=\sum_{n=1}^{j}
\stirlingtwo{j}{n}q^{n}\frac{d^{n}}{dq^{n}},
\end{align*}
where $\stirlingtwo{j}{n}$ is the Stirling number of the second kind defined by
\begin{align*}
\stirlingtwo{j}{0}=\stirlingtwo{0}{j}=\delta_{j, 0}, \qquad
\stirlingtwo{j+1}{n}=\stirlingtwo{j}{n-1}+n \stirlingtwo{j}{n}.
\end{align*}
Note that $\stirlingtwo{j}{j}=1$.
Therefore
\begin{align*}
S_{l, j}(m)&=\frac{1}{j! m^j}\sum_{n=1}^{j}
\stirlingtwo{j}{n}q^{n}\frac{d^{n}}{dq^{n}}(q-1)^{l}\big|_{q=\zeta_{m}}\\
&=\frac{1}{j! m^j}\sum_{n=1}^{j}
\stirlingtwo{k}{n}q^{n}l(l-1)\cdots (l-n+1)(q-1)^{l-n}\big|_{q=\zeta_{m}}\\
&=\frac{1}{j! m^j}\sum_{n=1}^{j}n!
\stirlingtwo{j}{n}\binom{l}{n}\zeta_m^{n}(\zeta_m-1)^{l-n}.
\end{align*}
Since $\zeta_m=1+\frac{2\pi i}{m} +O(m^{-2})$ as $m\rightarrow \infty$ and the coefficient of $\zeta_m^{n}(\zeta_m-1)^{l-n}$ in the above does not depend on $m$, we have
\begin{align*}
&=\frac{1}{j! m^j}\sum_{n=1}^{j}n!\stirlingtwo{j}{n}
\binom{l}{n}\left(\frac{2\pi i}{m}\right)^{l-n}(1+O(m^{-1}))\\
&=\frac{1}{m^j}
\binom{l}{j}\left(\frac{2\pi i}{m}\right)^{l-j}(1+O(m^{-1})),
\end{align*}
from which the statement follows.
\end{proof}

\begin{proof}[Proof of Proposition \ref{prop:1-q}]
Setting $t=0$ in \eqref{eq:1-q},
we obtain $1-\zeta_{m}=-\frac{2\pi i}{m} \alpha_{0}(m)$.
Hence $\alpha_{0}(m)=1+O(m^{-1})$.
For $l \ge 1$, applying Proposition \ref{prop:expansion} to the case $f(q)=1-q$,
we see that
\begin{align*}
-\frac{2\pi i }{m} \alpha_{l}(m)=\frac{1}{l}
\sum_{\substack{j_{1}+j_{2}=l-1 \\ j_{1}, j_{2} \ge 0}}
B_{l, j_{1}}\,S_{l, j_{2}}(m)\, \frac{-\zeta_m}{m}.
\end{align*}
The above expression and Lemma \ref{lem:SNL} imply $\alpha_{l}(m)=O(m^{-l})$.
\end{proof}

\begin{remark}\label{rem:taylor_exp_qmt}
In this paper, we only use the asymptotic formula for $q_m(t)$ described in Proposition \ref{prop:1-q}.
However, using the standard technique, one can compute the Taylor expansion of $q_m(t)$ at $t=0$.
It is given by
\[ q_m(t)= \zeta_m \sum_{l\ge0} t^l \sum_{j=0}^l \frac{(-\zeta_m)^j}{(j+1)!(l-j)!}\left(-\frac{j+1}{m}\right)_l,\]
where $\left(a\right)_l=a(a+1)\cdots (a+l-1)$.
\end{remark}

For an index $\kk=(k_{1}, \ldots , k_{d})$ and a positive integer $m$, we set
\begin{align*}
\tilde{H}_{m-1}(\kk; q)&=\left(-\frac{m}{2\pi i}(1-q)\right)^{-\wt(\kk)}H_{m-1}(\kk;q)\\
&=\sum_{m>m_{1}>\cdots >m_{d}>0}
\prod_{a=1}^{d}\left(-\frac{2\pi i}{m}\right)^{k_{a}}
\frac{q^{(k_{a}-1)m_{a}}}{(1-q^{m_{a}})^{k_{a}}}
\end{align*}
and write
\begin{align*}
\tilde{H}_{m-1}(\kk; q_{m}(t))=\sum_{l=0}^{\infty}\tilde{\alpha}_{l}(\kk; m)t^{l}.
\end{align*}
Proposition \ref{prop:1-q} implies that
\begin{align}
H_{m-1}(\kk;q_m(t))=\tilde{H}_{m-1}(\kk; q_{m}(t))
\left(1+\sum_{l \ge 0}O(m^{-1})\, t^{l} \right).
\label{eq:z-to-ztilde}
\end{align}
Thus, we have
\[ \alpha_0(\kk;m)=\tilde{\alpha}_0(\kk;m)\big(1+O(m^{-1})\big) \quad (m\rightarrow \infty)\]
and \eqref{eq:BTT} gives the asymptotic for $\tilde{\alpha}_0(\kk;m)$.
It also follows from \eqref{eq:z-to-ztilde} that, if the limit $\displaystyle\lim_{m\rightarrow\infty}\tilde{\alpha}_{l}(\kk; m)$ exists, we get
\begin{equation}\label{eq:lim_alpha}
\lim_{n\rightarrow\infty}\tilde{\alpha}_{l}(\kk; m)= \lim_{n\rightarrow\infty}\alpha_{l}(\kk; m).
\end{equation}

Now consider the asymptotics of $\tilde{\alpha}_{l}(\kk; m)$ for $l \ge 1$.
{}Using Proposition \ref{prop:expansion} and then Lemma \ref{lem:SNL},
we see that
\begin{equation}
\begin{aligned}
\tilde{\alpha}_{l}(\kk; m)&=\frac{1}{l}
\sum_{\substack{j_{1}+j_{2}+j_{3}=l-1 \\ j_{1}, j_{2}, j_{3} \ge 0}}
B_{l, j_{1}} \,S_{l, j_{2}}(m)\, \frac{1}{j_{3}!m^{j_3+1}}
(\theta_{q}^{j_{3}+1}\tilde{H}_{m-1})(\kk; \zeta_{m})\\
&=\frac{1}{l}
\sum_{\substack{j_{1}+j_{2}+j_{3}=l-1 \\ j_{1}, j_{2}, j_{3} \ge 0}}
B_{l, j_{1}} \,\frac{1}{m^l}\binom{l}{j_2}(2\pi i)^{l-j_2}(1+O(m^{-1})) \\
 &\times \frac{j_3+1}{m^{j_3+1}}  \sum_{\substack{l_1+\cdots+l_d=j_3+1\\l_1,\ldots,l_d\ge0}}
\sum_{m>m_1>\cdots >m_d>0} \prod_{a=1}^{d}\left(-\frac{2\pi i}{m}\right)^{k_{a}}\frac{1}{l_a!}\theta_q^{l_a}
\frac{q^{(k_{a}-1)m_{a}}}{(1-q^{m_{a}})^{k_{a}}}\bigg|_{q=\zeta_m}.
\end{aligned}
\label{eq:ztilde-expansion1}
\end{equation}

\begin{lemma}\label{lem:theta-[m]}
For $l \ge 0$ and $k, m \ge 1$, it holds that
\begin{align*}
\frac{1}{l!}\,\theta_{q}^{l}\frac{q^{(k-1)m}}{(1-q^{m})^{k}}=m^{l}\sum_{s=0}^{l}T_{s, l}(k)
\binom{s+k-1}{s}
\frac{q^{(k+s-1)m}}{(1-q^{m})^{k+s}},
\end{align*}
where
\begin{align*}
T_{s, l}(k)=\frac{s!}{l!}\sum_{a=s}^{l}\binom{l}{a}\stirlingtwo{a}{s}(k-1)^{l-a}.
\end{align*}
\end{lemma}

\begin{proof}
We may assume that $|q|<1$.
We calculate the generating function:
\begin{align*}
&
\sum_{l \ge 0}\frac{X^{l}}{l!}\theta_{q}^{l}\frac{q^{(k-1)m}}{(1-q^{m})^{k}}=
\sum_{l, j \ge 0}\frac{X^{l}}{l!}\binom{k+j-1}{j} \theta_{q}^{l} q^{(k+j-1)m} \\
&=\sum_{l, j \ge 0}\frac{((k+j-1)mX)^{l}}{l!}\binom{k+j-1}{j} q^{(k+j-1)m} \\
&=\sum_{j \ge 0}\binom{k+j-1}{j} q^{(k+j-1)m}e^{(k+j-1)mX}=
e^{(k-1)mX}\frac{q^{(k-1)m}}{(1-q^{m}e^{mX})^{k}}.
\end{align*}
Now compute
\begin{align*}
\frac{q^{(k-1)m}}{(1-q^{m}e^{mX})^{k}}&=
\frac{q^{(k-1)m}}{(1-q^{m})^{k}}\left(1-\frac{q^{m}}{1-q^{m}}(e^{mX}-1)\right)^{-k} \\
&=\sum_{s=0}^{\infty}\binom{k+s-1}{s}\frac{q^{(k+s-1)m}}{(1-q^{m})^{k+s}}
(e^{mX}-1)^{s}  \\
&=\sum_{s=0}^{\infty}\binom{k+s-1}{s}\frac{q^{(k+s-1)m}}{(1-q^{m})^{k+s}}
\sum_{j=s}^{\infty}\stirlingtwo{j}{s}\frac{s!}{j!}(mX)^{j} \\
&=\sum_{j=0}^{\infty}(mX)^{j}
\sum_{s=0}^{j}\binom{k+s-1}{s}\frac{q^{(k+s-1)m}}{(1-q^{m})^{k+s}}
\stirlingtwo{j}{s}\frac{s!}{j!},
\end{align*}
where for the third equality we have used \cite[Eq.(7)]{ArakawaIbukiyamaKaneko}.
Thus, we obtain
\begin{align*}
\frac{1}{l!}\,\theta_{q}^{l}\frac{q^{(k-1)m}}{(1-q^{m})^{k}}
&=\sum_{\substack{j_{1}+j_{2}=l \\ j_{1}, j_{2} \ge 0}}
\frac{((k-1)m)^{j_{1}}}{j_{1}!}
m^{j_{2}}
\sum_{s=0}^{j_{2}}\binom{k+s-1}{s}\frac{q^{(k+s-1)m}}{(1-q^{m})^{k+s}}
\stirlingtwo{j_{2}}{s}\frac{s!}{j_{2}!} \\
&=m^{l}\sum_{j=0}^{l}\frac{1}{l!}\binom{l}{j}
(k-1)^{l-j}
\sum_{s=0}^{j}\binom{k+s-1}{s}\frac{q^{(k+s-1)m}}{(1-q^{m})^{k+s}}
\stirlingtwo{j}{s}s! \\
&=m^{l}\sum_{s=0}^{l}
\binom{k+s-1}{s}\frac{q^{(k+s-1)m}}{(1-q^{m})^{k+s}}
T_{s, l}(k),
\end{align*}
which is the desired formula.
\end{proof}

Applying Lemma \ref{lem:theta-[m]} to \eqref{eq:ztilde-expansion1},
we see that
\begin{equation}
\begin{aligned}
 \tilde{\alpha}_{l}(\kk; m)&=\sum_{\substack{j_{1}, j_{2} \ge 0\\ \bl=(l_1,\ldots,l_d)\in \Z_{\ge0}^d\\ j_{1}+j_{2}+\wt(\bl)=l  }}
\frac{l-j_{1}-j_{2}}{l}B_{l, j_{1}}\,
\frac{1}{m^{l}}\binom{l}{j_{2}}(2\pi i)^{l-j_{2}}(1+O(m^{-1}))
 \\
&\times
\sum_{\substack{\bs =(s_1,\ldots,s_d)\in \Z_{\ge0}^d \\ l_{a} \ge s_{a} \ge 0 \\ (1\le a \le d)}}
\left(-\frac{m}{2\pi i}\right)^{\wt(\bs)}
\left\{\prod_{a=1}^{d}\binom{k_{a}+s_{a}-1}{s_{a}}T_{s_{a}, l_{a}}(k_{a})\right\}
Z_{m-1}(\bl; \kk+\bs),
\end{aligned}
\label{eq:ctilde-1}
\end{equation}
where $Z_{m-1}(\bl; \kk)$ is defined by
\begin{equation}\label{eq:def_Z}
Z_{m-1}(\bl; \kk)=\sum_{m-1\ge m_{1}>\cdots >m_{d}>0}
\prod_{a=1}^{d}
\left(\frac{m_{a}}{m}\right)^{l_{a}}
\left(-\frac{2\pi i}{m}\right)^{k_{a}}
\frac{\zeta_{m}^{(k_{a}-1)m_{a}}}{(1-\zeta_{m}^{m_{a}})^{k_{a}}}
\end{equation}
for $\bl=(l_{1}, \ldots , l_{d}) \in \Z_{\ge 0}^{d}$ and
$\kk=(k_{1}, \ldots , k_{d}) \in \Z_{\ge 1}^{d}$.

\begin{lemma}\label{lem:B-log}
For $\bl \in \Z_{\ge 0}^{d}$ and an index $\kk \in \Z_{\ge 1}^{d}$,
we have $Z_{m-1}(\bl; \kk)=O((\log{m})^{d})$ as $m \to \infty$.
\end{lemma}

\begin{proof}
We see that
\begin{align}
\left|Z_{m-1}(\bl; \kk)\right|\le
\sum_{m-1\ge m_{1}>\cdots >m_{d}>0}
\prod_{a=1}^{d}\left|\left(-\frac{2\pi i}{m}\right)^{k_{a}}
\frac{\zeta_{m}^{(k_{a}-1)m_{a}}}{(1-\zeta_{m}^{m_{a}})^{k_{a}}}
\right|.
\label{eq:B-log-1}
\end{align}
For an index $\kk=(k_{1}, \ldots , k_{d})$ define
\begin{align*}
\widetilde{Z}_{m}(\kk)=
\sum_{m/2 \ge m_{1}>\cdots >m_{d}>0}
\prod_{a=1}^{d}
\left|\left(-\frac{2\pi i}{m}\right)^{k_{a}}
\frac{\zeta_{m}^{(k_{a}-1)m_{a}}}{(1-\zeta_{m}^{m_{a}})^{k_{a}}}
\right|.
\end{align*}
By a similar argument to \cite[\S 2.3.2]{BTT18},
the right side of \eqref{eq:B-log-1} is bounded from above by
$\sum_{a=0}^{d} \widetilde{Z}_{m}(\overline{\kk_a})\widetilde{Z}_{m}(\kk^a)$.
Hence it suffices to show that
$\widetilde{Z}_{m}(\kk)=O((\log{m})^{d})$ for $\kk \in \Z_{\ge 1}^{d}$.

For $k \ge 1$, we define the function $\varphi_{k}(x)$ by
\begin{align}
\varphi_{k}(x)=(-\pi i x)^{k}\frac{e^{(k-1)\pi i x}}{(1-e^{\pi i x})^{k}}.
\label{eq:def-varphi}
\end{align}
Then
\begin{align*}
\widetilde{Z}_{m}(\kk)=
\sum_{m/2 \ge m_{1}>\cdots >m_{d}>0}
\prod_{a=1}^{d}
\left|
m_{a}^{-k_{a}}\varphi_{k}(2m_{a}/m)
\right|.
\end{align*}
Since $\varphi_{k}(x)$ is bounded in the unit disc $|x|\le 1$,
there exists a positive constant $C_{\kk}$ such that
\begin{align*}
\widetilde{Z}_{m}(\kk) \le C_{\kk}
\sum_{m/2 \ge m_{1}>\cdots >m_{d}>0}\prod_{a=1}^{d}m_{a}^{-k_{a}}.
\end{align*}
Hence $\widetilde{Z}_{m}(\kk)=O((\log{m})^{d})$ because
$\kk \in \Z_{\ge 1}^{d}$.
\end{proof}

Lemma \ref{lem:B-log} implies that
the summand of \eqref{eq:ctilde-1} is $O((\log{m})^{J}/m)$ for some $J \ge 1$ unless $\wt(\bs)=l$,
in which case $\bs=\bl$ and $j_{1}=j_{2}=0$.
Since $B_{l, 0}=1$, for an index $\kk$ we have
\begin{equation}\label{eq:tilde_alpha}
\tilde{\alpha}_{l}(\kk; m)=(-1)^{l}
\sum_{\substack{\bl\in \Z_{\ge0}^{\dep(\kk)}\\ \wt(\bl)=l}}b\binom{\kk}{\bl}
Z_{m-1}(\bl; \kk+\bl)+O((\log{m})^{J_{\kk, l}}/m)
\end{equation}
for some $J_{\kk, l} \ge 1$, where we have used the notation in \eqref{eq:ind_notation}.

For $\bl=(l_1,\ldots,l_d) \in \Z_{\ge 0}^{d}$ and
$\kk=(k_1,\ldots,k_d) \in \Z_{\ge 1}^{d}$, we set
\begin{align}
A_{m}^{-}(\bl; \kk)=\sum_{m/2>m_{1}>\cdots >m_{d}>0}
\prod_{a=1}^{d}\left(\frac{m_{a}}{m}\right)^{l_{a}}
\left(-\frac{2\pi i}{m}\right)^{k_{a}}
\frac{\zeta_{m}^{(k_{a}-1)m_{a}}}{(1-\zeta_{m}^{m_{a}})^{k_{a}}}
\label{eq:def-A+}
\end{align}
and define $A_{m}^{+}(\bl; \kk)$ by \eqref{eq:def-A+} where the range of summation
$m/2>m_{1}>\cdots >m_{d}>0$ is replaced by $m/2 \ge m_{1}>\cdots >m_{d}>0$.
We set $A_{m}^{\pm}(\varnothing; \varnothing)=1$.
%If $m$ is odd,
Then it holds that
\begin{align*}
Z_{m-1}(\bl; \kk+\bl)=\sum_{a=0}^{d}
&
\left\{\sum_{\substack{l_{j} \ge s_{j} \ge 0 \\ (1\le j \le a)}}
\left(\prod_{j=1}^{a}(-1)^{k_{j}+l_{j}+s_{j}}\binom{l_{j}}{s_{j}}\right)
\overline{A_{m}^{+}(s_{a}, \ldots , s_{1}; \overline{\kk_a+\bl_a})}
\right\} \\
&\times
A_{m}^{-}(\bl^a; \kk^a+\bl^a),
\end{align*}
where the bar on $A_{m}^{+}$ means the complex conjugate.

\begin{lemma}\label{lem:A_asymp}
Let $\bs\in\Z_{\ge0}^d$ and $\kk\in\Z_{\ge1}^d$.
Suppose that $d\ge 1$ and $k_{a}-1 \ge s_{a} \ge 0$ for $1\le a \le r$.
If $\bs\not=(0, \ldots , 0)$, it holds that
$A_{m}^{\pm}(\bs; \kk)=O((\log{m})^{d}/m)$.
If $\bs=(0, \ldots , 0)$, we have
\begin{align*}
A_{m}^{\pm}(\{0\}^{d}; \kk)=\zeta^{*}(\kk; \gamma_m)+
O((\log{m})^{J_{\kk}}/m)
\end{align*}
for some $J_{\kk} \ge 1$.
\end{lemma}

\begin{proof}
Since $A_{m}^{+}(\bs; \kk)=A_{m}^{-}(\bs; \kk)$ if $m$ is odd and
\begin{align*}
A_{m}^{+}(\bs; \kk)=-\frac{1}{2^{s_{1}}}\left( \frac{\pi i}{m} \right)^{k_{1}}
A_{m}^{-}(\bs^{1}; \kk^{1})+A_{m}^{-}(\bs;\kk)
\end{align*}
if $m$ is even,
it suffices to show the statement for $A_{m}^{-}(\bs;\kk)$.
For $\bs=(0, \ldots , 0)$ it is proved in \cite[Proposition 2.9]{BTT18}.
Using the function $\varphi_{k}(x)$ defined by \eqref{eq:def-varphi},
we rewrite as
\begin{align*}
A_{m}^{-}(\bs; \kk)=m^{-\wt(\bs)}\sum_{m/2>m_{1}>\cdots >m_{d}>0}\prod_{a=1}^{d}
\frac{\varphi_{k_{a}}(2m_{a}/m)}{m_{a}^{k_{a}-s_{a}}}.
\end{align*}
By the same argument as in the proof of Lemma \ref{lem:B-log},
we see that the sum in the right side is $O((\log{m})^{d})$.
Hence, if $\wt(\bs)\ge 1$, it holds $A_{m}^{-}(\bs; \kk)=O((\log{m})^{d}/m)$.
\end{proof}

\begin{theorem}\label{thm:asym_alpha}
For all $l \ge 1$ and each index $\kk$ of depth $d$, we have
\begin{equation}\label{eq:asym_alpha}
\begin{aligned}
\tilde{\alpha}_{l}(\kk; m)&=\sum_{a=1}^{d}(-1)^{ \wt(\kk_a)}
\sum_{\substack{\bl\in \Z_{\ge0}^a\\ \wt(\bl)=l}}b\binom{\kk_a}{\bl}\zeta^\ast\left(\overline{\kk_a+\bl}; \overline{\gamma_m}\right)\,
\zeta^\ast\left(\kk^a; \gamma_m\right) +O\left( \frac{(\log{m})^{J_{\kk, l}}}{m}\right)
\end{aligned}
\end{equation}
for some integer $J_{\kk, l}\ge 1$.
\end{theorem}
\begin{proof}
Applying Lemma \ref{lem:A_asymp} to \eqref{eq:tilde_alpha}, we obtain the desired result.
\end{proof}

%%%%%%%%%%%%%%%%%%%%%%%%%%%%%%%%%%
\subsection{Proof of \eqref{eq:lim_formula}}
In order to check that the limit $\displaystyle\lim_{m\rightarrow \infty}H_{m-1}(\kk;q_m(t))$ exists,
we need the following lemma, from which, by Theorem \ref{thm:asym_alpha}, \eqref{eq:lim_formula} follows.

\begin{lemma}\label{lem:indep_T}
For any index $\kk$, the sum
\[\sum_{a=0}^{\dep(\kk)} (-1)^{\wt(\kk_a)}
\sum_{\bl\in \Z_{\ge0}^a}  t^{\wt(\bl)}\,
b\binom{\kk_a}{\bl} \zeta^\ast \left(
\overline{\kk_a+\bl};T+\frac{\pi i}{2}
\right)
\zeta^\ast \left(\kk^a;T-\frac{\pi i}{2}\right)
.\]
does not depend on $T$ and coincides with $\widehat{\xi}(\kk)\in \mathcal{Z}[\pi i][[t]]$.
\end{lemma}

Our proof of Lemma \ref{lem:indep_T} is done by a similar computation with \cite{Hirose,Jarossay19}.
We repeat it for convenience.

Recall the notation from \S2.1.
Consider the algebraic projection
\begin{equation*}\label{eq:def_reg0}
{\rm reg}_0 : \mathfrak{h}_\shuffle\rightarrow \mathfrak{h}^1_\shuffle
\end{equation*}
which sends $x_0$ to 0.
It is well-known that
\begin{equation}\label{eq:reg0}
{\rm reg}_0(x_{\kk}x_0^l) =(-1)^l\sum_{\substack{\bl\in \Z_{\ge0}^d \\ \wt(\bl)=l}} b\binom{\kk}{\bl} x_{\kk+\bl}
\end{equation}
holds for any index $\kk$ of depth $d$ and $l\ge0$ (see e.g., \cite[p.955]{Brown12}).
Denote by $\{x_0,x_1\}^\times$ the set of words consisting of $x_0$ and $x_1$.
We assume that the set $\{x_0,x_1\}^\times$ contains the empty word $\varnothing$.
Let us consider the two non-commutative generating series
\begin{align*}
\Phi^\shuffle(T)=\sum_{w\in\{x_0,x_1\}^\times}
Z^\shuffle ({\rm reg}_0(w)) w \quad \mbox{and}\quad
\Phi^\ast(T)=\Lambda(x_1)\Phi^\shuffle(T) ,
\end{align*}
which lie in $ \R[T]\langle\langle x_0,x_1\rangle\rangle$,
where we set $Z^{\shuffle}(\varnothing)=1$ and
\[\Lambda(x)=\exp\left( \displaystyle\sum_{n\ge2} \frac{(-1)^{n-1}}{n}\zeta(n)x^n\right).\]

For a non-commutative power series $S\in \R[T]\langle\langle x_0,x_1\rangle\rangle$,
we denote by $S[w]$ the coefficient of a word $w$ in $S$.

\begin{lemma}{\cite[Proposition 3.2.2]{Jarossay19}}\label{lem:shuffle_l_part}
For each index $\kk$ and $l\ge0$ we have
\begin{equation}\label{eq:refined_reg_thm}
\Phi^\ast(T)[x_{\kk}x_0^l]=(-1)^l\sum_{\substack{\bl\in \Z_{\ge0}^{\dep(\kk)} \\ \wt(\bl)=l}} b\binom{\kk}{\bl} \zeta^\ast(\kk+\bl;T).
\end{equation}
\end{lemma}

\begin{proof}
Denote by $\{y_k\mid k\ge1\}^\times$ the set of words consisting of $y_k \ (k\ge1)$.
Let
\begin{align*}
\Psi^\ast(T) & = \sum_{w\in\{y_k\mid k\ge1\}^\times} Z^\ast (w) w,
\end{align*}
where we set $Z^\ast(\varnothing)=1$.
Define the $ \R[T]$-linear map
\begin{align*}
\pi_Y : \R[T]\langle\langle x_0,x_1\rangle\rangle&\longrightarrow  \R[T]\langle\langle y_k\mid k\ge1\rangle\rangle
\end{align*}
by
$\pi_{Y}(x_{\kk})=y_{\kk}$ for $\kk\in \Z_{\ge1}^d \, (d\ge0)$ and
$\pi_{Y}(wx_0)=0$ for $w\in \mathfrak{H}^1$.
Then the regularization theorem by Racinet \cite{Racinet02} (see also \cite{IharaKanekoZagier06})
states the equality
\begin{equation}\label{eq:reg_theorem}
\Psi^\ast(T)=\pi_Y\big(\Lambda(x_1)\Phi^\shuffle(T)).
\end{equation}
This shows that
\begin{equation}\label{eq:reg_thm}
\Phi^\ast(T)[x_{\kk}]=\Lambda(x_1)\Phi^\shuffle(T)[x_{\kk}]=\Psi^\ast(T)[y_{\kk}]
\end{equation}
holds for any index $\kk$.

Let $l\ge1$.
For an admissibe index $\kk$, by using \eqref{eq:reg0} and \eqref{eq:reg_thm},
we see
\begin{align*}
&
\Phi^\ast(T)[x_{\kk}x_0^l]=\Lambda(x_1)\Phi^\shuffle(T)[x_{\kk}x_0^l]=
\Phi^\shuffle(T)[x_{\kk}x_0^l]=
(-1)^l\sum_{\substack{\bl\in \Z_{\ge0}^{\dep(\kk)} \\ \wt(\bl)=l}}
b\binom{\kk}{\bl} \Phi^\shuffle(T)[x_{\kk+\bl}]\\
&= (-1)^l\sum_{\substack{\bl\in \Z_{\ge0}^{\dep(\kk)} \\ \wt(\bl)=l}} b\binom{\kk}{\bl} \left(\Lambda(x_1)\Phi^\shuffle(T)\right)[x_{\kk+\bl}]=
(-1)^l\sum_{\substack{\bl\in \Z_{\ge0}^{\dep(\kk)} \\ \wt(\bl)=l}}
b\binom{\kk}{\bl} \Psi^\ast(T)[y_{\kk+\bl}],
\end{align*}
from which \eqref{eq:refined_reg_thm} follows.

Let us consider the case $\kk=(\underbrace{1,\ldots,1}_a,k_{a+1},\ldots,k_d)$ with $k_{a+1}\ge2$ and $a\ge1$.
Set $\Lambda(x)=\displaystyle\sum_{j\ge0}\lambda_j x^j$.
Using \eqref{eq:reg0}, one computes
\begin{align*}
\Phi^\ast(T)[x_{\kk}x_0^l]&= \sum_{j=0}^a \lambda_{j} \Phi^\shuffle(T)[ x_{\kk^j}x_0^l]=\sum_{j=0}^a \lambda_{j}  (-1)^l\sum_{\substack{\bl\in \Z_{\ge0}^{d-j}\\ \wt(\bl)=l}} b\binom{\kk^j}{\bl} \Phi^\shuffle(T)[x_{\kk^j+\bl}]\\
&=(-1)^l\sum_{j=0}^a \lambda_{j}  \sum_{\substack{\bl=(l_{1},\ldots,l_d)\in \Z_{\ge0}^{d}\\ \wt(\bl)=l\\l_1=\cdots=l_j=0}} b\binom{\kk}{\bl} \Phi^\shuffle(T)[x_{\kk^j+\bl^j}].
\end{align*}
For $0\le s \le a$ we set
\begin{align*}
I_{s}=\{\bl \in \Z_{\ge 0}^{d} \mid \wt(\bl)=l, \, l_{1}=\cdots =l_{s}=0\}, \qquad
I_{s}^{+}=\{\bl \in I_{s} \mid l_{s+1}\ge1\}.
\end{align*}
{}From the above calculation, we see that
\begin{align*}
(-1)^{l}\Phi^\ast(T)[x_{\kk}x_0^l]=\sum_{j=0}^{a}\lambda_{j}\sum_{\bl \in I_{j}}
b\binom{\kk}{\bl} \Phi^\shuffle(T)[x_{\kk^j+\bl^j}].
\end{align*}
Using $I_{j}=I_{j}^{+} \sqcup \cdots \sqcup I_{a-1}^{+} \sqcup I_{a}$,
we decompose the right side into the two parts
\begin{align*}
&
R_{1}=\sum_{j=0}^{a-1}\Lambda_{j}\sum_{s=j}^{a-1}\sum_{\bl \in I_{s}^{+}}
b\binom{\kk}{\bl} \Phi^\shuffle(T)[x_{\kk^j+\bl^j}], \\
&
R_{2}=\sum_{j=0}^{a}\Lambda_{j}\sum_{\bl \in I_{a}}
b\binom{\kk}{\bl} \Phi^\shuffle(T)[x_{\kk^j+\bl^j}].
\end{align*}
For $j\le s \le a-1$ and $\bl \in I_{s}^{+}$ we have
\begin{align*}
b\binom{\kk}{\bl}=b\binom{\kk^{s}}{\bl^{s}}, \qquad
x_{\kk^{j}+\bl^{j}}=x_{1}^{s-j}x_{\kk^{s}+\bl^{s}}
\end{align*}
and $\kk^{s}+\bl^{s}$ is admissible.
Hence
\begin{align*}
R_{1}&=\sum_{s=0}^{a-1}\sum_{\bl \in I_{s}^{+}}b\binom{\kk^{s}}{\bl^{s}}
\sum_{j=0}^{s}\Lambda_{j}
\Phi^{\shuffle}(T)[x_{1}^{s-j}x_{\kk^{s}+\bl^{s}}]	 \\
&=\sum_{s=0}^{a-1}\sum_{\bl \in I_{s}^{+}}b\binom{\kk^{s}}{\bl^{s}}
\Phi^{*}(T)[x_{1}^{s}x_{\kk^{s}+\bl^{s}}]=
\sum_{s=0}^{a-1}\sum_{\bl \in I_{s}^{+}}b\binom{\kk}{\bl}
\Phi^{*}(T)[x_{\kk+\bl}].
\end{align*}
Similarly we see that
\begin{align*}
R_{2}=\sum_{\bl \in I_{a}}b\binom{\kk}{\bl}
\Phi^{*}(T)[x_{\kk+\bl}].
\end{align*}
Since the direct sum $I_{0}^{+}\sqcup \cdots \sqcup I_{a-1}^{+} \sqcup I_{a}$ is equal to
the set of $\bl \in \Z_{\ge 0}^{d}$ with $\wt(\bl)=l$, we find that
\begin{align*}
R_{1}+R_{2}=
\sum_{\substack{\bl\in \Z_{\ge0}^d \\ \wt(\bl)=l}} b\binom{\kk}{\bl} \Phi^{*}(T)[x_{\kk+\bl}]=
\sum_{\substack{\bl\in \Z_{\ge0}^d \\ \wt(\bl)=l}} b\binom{\kk}{\bl} \Psi^{*}(T)[y_{\kk+\bl}].
\end{align*}
This completes the proof.
\end{proof}

Define the $\R[T]$-linear map $\sigma:\R[T]\langle\langle x_0,x_1\rangle\rangle \rightarrow\R[T]\langle\langle x_0,x_1\rangle\rangle $ by
\[ \sigma(x_{a_1}\cdots x_{a_k})=(-1)^{k}x_{a_k}\cdots x_{a_1}.\]
This gives the antipode of the shuffle coalgebra $\R[T]\langle\langle x_0,x_1\rangle\rangle$.
Thus, for any algebra homomorphism $\varphi:\mathfrak{h}_\shuffle \rightarrow \R[T]$,
its non-commutative generating series $S=\sum_{w} \varphi(w)w$ is group-like
with respect to the shuffle coproduct and satisfies $\sigma(S)S=1$.
We let
\[ \Xi(T_1,T_2)=\sigma\big(\Phi^\ast(T_1)  \big )x_1 \Phi^\ast(T_2) .\]

\begin{proposition}\label{pror:generating_function}
For any index $\kk$ and $l\ge1$, we have
\begin{align*}
 \Xi(T_1,T_2)[x_0^lx_1x_{\kk}]&=\sum_{a=1}^{\dep(\kk)}(-1)^{ \wt(\kk_a)}
\sum_{\substack{\bl\in \Z_{\ge0}^a\\ \wt(\bl)=l}}
b\binom{\kk_a}{\bl}
\zeta^\ast\left(\overline{\kk_a+\bl_a}; T_1\right)
\zeta^\ast\left(\kk^a; T_2\right)
\end{align*}
\end{proposition}
\begin{proof}
Let $\kk=(k_1,\ldots,k_a)$.
One computes
\begin{align*}
 \Xi(T_1,T_2)[x_0^lx_1x_{\kk}]&=\sum_{a=0}^{\dep(\kk)}\sigma\big(\Phi^\ast(T_1)  \big )
 [x_0^lx_1 x_0^{k_1-1}\cdots x_1x_0^{k_a-1}]
 \Phi^\ast(T_2) [x_{\kk^a}]\\
 &=\sum_{a=1}^{\dep(\kk)}(-1)^{ l+\wt(\kk_a)}
 \Phi^\ast(T_1) [x_0^{k_a-1}x_1\cdots x_0^{k_1-1}x_1x_0^l] \Phi^\ast(T_2) [x_{\kk^a}].
\end{align*}
Thus, the desired formula follows from Lemma \ref{lem:shuffle_l_part}.
\end{proof}

\begin{proposition}{\cite[Theorem 9]{Hirose}}\label{prop:xi}
Let $\Phi_{\exp}(T)= \sigma\big(\Phi^\shuffle(0)\big) \exp(T x_1)\Phi^\shuffle(0)$.
Then we have
\[ \Xi(T_1,T_2)=\frac{1}{2\pi i} \left(\Phi_{\exp}(\pi i+T_2-T_1)-\Phi_{\exp}(-\pi i+T_2-T_1)\right).\]
In particular, it holds that
\[ \Xi(T+\frac{\pi i}{2},T-\frac{\pi i}{2})=\frac{1}{2\pi i}  \left(
\Phi_{\exp}(0)-\Phi_{\exp}(-2\pi i)
\right), \]
which is independent of $T$.
\end{proposition}

\begin{proof}
Since $\Phi^\shuffle(T)=\exp(T x_1)\Phi^\shuffle(0)$, one computes
\begin{align*}
\Xi(T_1,T_2)&=\sigma(\Lambda(x_1)\Phi^\shuffle(T_1))x_1 \Lambda(x_1)\Phi^\shuffle(T_2)\\
&=\sigma(\Phi^\shuffle(T_1))\Lambda(-x_1)x_1 \Lambda(x_1)\Phi^\shuffle(T_2)\\
&=\sigma(\Phi^\shuffle(0))\exp(-T_1x_1)\Lambda(-x_1)x_1 \Lambda(x_1)\exp(T_2x_1)\Phi^\shuffle(0)\\
&=\sigma(\Phi^\shuffle(0))\frac{\sin \pi x_1}{\pi} \exp ((T_2-T_1)x_1)\Phi^\shuffle(0)\\
&=\sigma(\Phi^\shuffle(0))\big( \exp((\pi i+T_2-T_1)x_1)-\exp((-\pi i+T_2-T_1)x_1)\big)\Phi^\shuffle(0),
\end{align*}
from which the statement follows.
\end{proof}

\begin{proof}[Proof of Lemma \ref{lem:indep_T}]
Lemma \ref{lem:indep_T} follows from Propositions \ref{pror:generating_function} and \ref{prop:xi}.
\end{proof}

%%%%%%%%%%%%%%%%%%%%%%%%%%%%%%%%%%
\subsection{Proof of \eqref{eq:hatS_from_q2}}
We turn to the proof of \eqref{eq:hatS_from_q2}.
Define $\beta_{l}(\kk; m) \in \mathbb{C} \,\, (l\ge 0, m\ge1)$ by
\begin{align*}
\left(-\frac{m}{2\pi i}(1-q_{m}(t))\right)^{-\wt(\kk)}\overline{H}_{m-1}(\kk;q_{m}(t))=
\sum_{l \ge 0}\beta_{l}(\kk; m)\,t^{l}.
\end{align*}
We show the limit of $\beta_{l}(\kk; m)$ as $m \to +\infty$ converges.
Then we have
\begin{align*}
\phi_{\widehat{\mathcal{S}}}(\overline{\zeta}_{\widehat{\mathcal{Q}}}(\kk))=
\sum_{l \ge 0}(\lim_{m \to \infty}\beta_{l}(\kk; m))t^{l}.
\end{align*}

Using the equality
\begin{align*}
\frac{1}{l!}\theta_{q}^{l}\frac{q^{m}}{(1-q^{m})^{k}}=m^{l}
\sum_{s=0}^{l}(-1)^{s+l}T_{s,l}(k)\binom{s+k-1}{s}\frac{q^{m}}{(1-q^{m})^{k+s}}
\end{align*}
for $l \ge 0$ and $k, m \ge 1$, which is shown by the same way as Lemma \ref{lem:theta-[m]},
we see that
\begin{align*}
&
\beta_{l}(\kk; m)=\sum_{\substack{j_{1}, j_{2} \ge 0\\ \bl=(l_1,\ldots,l_d)\in \Z_{\ge0}^d\\ j_{1}+j_{2}+\wt(\bl)=l  }}
\frac{l-j_{1}-j_{2}}{l}B_{l, j_{1}}
\frac{1}{m^{l}}\binom{l}{j_{2}}(2\pi i)^{l-j_{2}}(1+O(m^{-1}))
 \\
&\times
\sum_{\substack{\bs =(s_1,\ldots,s_d)\in \Z_{\ge0}^d \\ l_{a} \ge s_{a} \ge 0 \\ (1\le a \le d)}}
\left(-\frac{m}{2\pi i}\right)^{\wt(\bs)}(-1)^{\wt(\bs)+\wt(\bl)}\,
b\binom{\kk}{\bs}
\prod_{a=1}^{d}T_{s_{a}, l_{a}}(k_{a}) \, \overline{Z_{m-1}(\bl; \kk+\bs)},
\end{align*}
where $Z_{m-1}(\bl; \kk)$ is the complex number defined by \eqref{eq:def_Z}.
Now the desired equality is obtained from the above expression in a similar manner to
the proof of \eqref{eq:hatS_from_q}.

For the star version, we use
\[\frac{q^{m}}{[m]^{k_{1}}}
\frac{q^{m}}{[m]^{k_{2}}}=
\frac{q^{m}}{[m]^{k_{1}+k_{2}}}+(q-1)
\frac{q^{m}}{[m]^{k_{1}+k_{2}-1}}\]
to obtain the relations of the following form for any index $\kk=(k_{1}, \ldots , k_{d})$:
\begin{align*}
\overline{H}^{\star}_{m}(\kk;q)=\sum_{\bl}\overline{H}_{m}(\bl;q)+\sum_{\substack{\kk' \\ \wt(\kk')<\wt(\kk)}}
(q-1)^{\wt(\kk)-\wt(\kk')}c_{\kk,\kk'}\overline{H}_{m}(\kk';q),
%\label{eq:Hbarstar-to-Hbar}
\end{align*}
where the sum $\sum_{\bl}$ is over all indices $\bl$ of the form
$(k_{1} \square k_{2} \square \cdots \square k_{d})$ in which
$\square$ is either '$+$' (plus) or ',' (comma),
and $c_{\kk,\kk'}$ is an integer independent on $m$.
Then the result is a consequence of the limit formula
$\displaystyle\lim_{m \to \infty}\overline{H}_{m}(\kk;q_{m}(t))=\overline{\widehat{\xi}(\kk)}$
and the definition of $\widehat{\xi}^\star(\kk)$.
So we are done.

%%%%%%%%%%%%%%%%%%%%%%%%%%%%%%%%%%
\section{Proofs of relations}\label{app:proofs_of_relations}

%%%%%%%%%%%%%%%%%%%%%%%%%%%%%%%%%%
\subsection{Proof of reversal formulas}

\begin{lemma}\label{prop:[p]-expansions}
Suppose that $p$ is a prime, $p>m\ge 1$ and $k \ge 1$.
It holds that
\begin{align}
&
\frac{1}{[p-m]^{k}}=\left(-(q^{p})^{-1}\frac{q^{m}}{[m]}\right)^{\! k}\,
\sum_{l=0}^{n-1}\binom{k+l-1}{l}
\left( (q^{p})^{-1}[p]\frac{q^{m}}{[m]} \right)^{l}
\label{eq:[p-m]-inv}
\end{align}
in $Z_{p, n}$ for all $n\ge1$.
\end{lemma}

\begin{proof}
Note that there exists $f_{k,n}(x)\in \Z[x]$ such that
\begin{align*}
(1-x)^{k}\sum_{l=0}^{n-1}\binom{k+l-1}{l}x^{l} = 1+x^nf_{k,n}(x).
\end{align*}
Set $x=(q^{p})^{-1}[p]q^{m}/[m]$ and use
\begin{align*}
1-(q^{p})^{-1}[p]\frac{q^{m}}{[m]}=(q^{p})^{-1}\frac{q^{m}}{[m]}
\left( q^{p-m}[m]-[p] \right)=-(q^{p})^{-1}\frac{q^{m}}{[m]}[p-m].
\end{align*}
Then we obtain \eqref{eq:[p-m]-inv}.
\end{proof}

\begin{proof}[Proof of Theorem \ref{thm:q-reversal}]
It suffices to show the identity
\begin{align*}
\overline{H}_{p-1}^{\bullet}(\kk;q)=
(-(q^{p})^{-1})^{\wt(\kk)}q^{\dep(\kk)p}
\sum_{\substack{\bl \in \mathbb{Z}_{\ge 0}^{\dep(\kk)} \\ \wt(\bl)<n }}
((q^{p})^{-1}[p])^{\wt(\bl)}\,
b\binom{\kk}{\bl}
H_{p-1}^{\bullet}(\overline{\kk+\bl};q)
\end{align*}
in $Z_{p, n}$ for $\bullet \in \{\emptyset, \star\}$, any prime $p$ and $n \ge 1$.
It is obtained by changing the summation variable $m_{a}$ to $p-m_{a}$ in
$\overline{H}_{p-1}^{\bullet}(\kk;q)$ and using \eqref{eq:[p-m]-inv}.
\end{proof}

\begin{proof}[Proof of Corollary \ref{cor:reversal}]
Since $\phi_{\widehat{\mathcal{F}}}(q^{\pm \boldsymbol{p}}) =1$, applying Theorems \ref{thm:hatA_from_q} and \ref{thm:hatS_from_q}
%Remark \ref{rem:complex_conj_hat_A} and \eqref{eq:power_of_qmt_lim}
to Theorem \ref{thm:q-reversal}, we get the desired formulas.
\end{proof}

%%%%%%%%%%%%%%%%%%%%%%%%%%%%%%%%%%
\subsection{Proof of duality formulas}
\begin{proof}[Proof of Theorem \ref{thm:Hoffman's-duality}]
Set $\kk=(k_{1}, \ldots , k_{d})$ and $\kk^{\vee}=(k_{1}', \ldots , k_{s}')$.
We start from a $q$-analogue of the duality formula due to Bradley \cite{Bradley05_2} (see also \cite{Kawashima10}):
\begin{align*}
&
\sum_{n\ge m_{1} \ge \cdots \ge m_{d}>0}
(-1)^{m_{1}-1}q^{m_{1}(m_{1}-1)/2}\qbinom{n-1}{m_{1}-1}
\prod_{a=1}^{d}\frac{q^{(k_{a}-1)m_{a}}}{[m_{a}]^{k_{a}}} \\
&=\sum_{n=m_{1}\ge m_{2} \ge \cdots \ge m_{s}>0}
\frac{1}{[m_{1}]^{k_{1}'}}
\prod_{a=2}^{s}\frac{q^{m_{a}}}{[m_{a}]^{k_{a}'}},
\end{align*}
where $\qbinom{n}{m}$ is the $q$-binomial coefficient
\begin{align*}
\qbinom{n}{m}=\prod_{j=1}^{m}\frac{[n-(j-1)]}{[j]}.
\end{align*}
Multiply the both sides by $(-1)^{n}q^{(p-n)(p-n-1)/2}\qbinom{p-1}{n}$ and
take the sum over $1\le n \le p-1$ using the equalities
\begin{align*}
\qbinom{p-1}{n-1}\qbinom{n-1}{m_{1}-1}=\qbinom{p-1}{m_{1}-1}\qbinom{p-m_{1}}{n-m_{1}}
\end{align*}
and
\begin{align*}
\sum_{n=m_{1}}^{p-1}(-1)^{n}q^{(p-n)(p-n-1)/2}\qbinom{p-m_{1}}{n-m_{1}}=(-1)^{p-1}.
\end{align*}
Then we obtain
\begin{align*}
&
(-1)^{p-1}\sum_{p> m_{1}\ge	\cdots \ge m_{d}>0}
(-1)^{m_{1}-1}q^{m_{1}(m_{1}-1)/2}\qbinom{p-1}{m_{1}-1}
\prod_{a=1}^{d}\frac{q^{(k_{a}-1)m_{a}}}{[m_{a}]^{k_{a}}} \\
&=\sum_{p> m_{1}\ge	\cdots \ge m_{s}>0}
(-1)^{m_{1}}q^{(p-m_{1})(p-m_{1}-1)/2}\qbinom{p-1}{m_{1}-1}
\frac{1}{[m_{1}]^{k_{1}'}}
\prod_{a=2}^{s}\frac{q^{m_{a}}}{[m_{a}]^{k_{a}'}}.
\end{align*}
It is an identity of rational functions in $q$.
Using Proposition \ref{prop:[p]-expansions},
we see that the following equalities hold in $Z_{p, n}$ for any $n \ge 1$:
\begin{align*}
&
(-1)^{m-1}q^{m(m-1)/2}\qbinom{p-1}{m-1}=(-1)^{p-1}q^{p(p-1)/2}
\prod_{a=m}^{p-1}\frac{[a]}{[a]-[p]} \\
&=(-1)^{p-1}q^{p(p-1)/2}
\left(
1+\sum_{l=1}^{n-1}[p]^{l}\sum_{p>m_{1}\ge \cdots \ge m_{l} \ge m}\prod_{j=1}^{l}\frac{1}{[m_{j}]}
\right)
\end{align*}
and
\begin{align*}
&
(-1)^{m}q^{(p-m)(p-m-1)/2}\qbinom{p-1}{m-1}=
(-1)^{m}q^{(p-m)(p-m-1)/2}\prod_{a=m}^{p-1}\frac{[a]}{[p-a]} \\
&=(-1)^{p}(q^{p})^{-1}q^{m}
\left(
1+\sum_{l=1}^{n-1}((q^{p})^{-1}[p])^{l}
\sum_{p>m_{1}\ge \cdots \ge m_{l} \ge m}\prod_{j=1}^{l}\frac{q^{m_{j}}}{[m_{j}]}
\right).
\end{align*}
Hence for any prime $p$ we obtain
\begin{align*}
(-1)^{p-1}q^{p(p+1)/2}
\sum_{l=0}^{n-1}[p]^{l}
H_{p-1}^{\star}(\{1\}^{l}, \kk;q)=-\sum_{l = 0}^{n-1}
(q^{-p}[p])^{l}\, \overline{H}_{p-1}^{\star}(\{1\}^{l}, \kk^{\vee};q) \mod [p]^n,
\end{align*}
which completes the proof.
\end{proof}

\begin{proof}[Proof of Corollary \ref{cor:duality_F}]
The relation \eqref{eq:Hoffman-dual-1} for $\mathcal{F}=\mathcal{A}$ is obtained by applying $\phi_{\widehat{\mathcal{A}}}$
%Remark \ref{rem:complex_conj_hat_A}
to Theorem \ref{thm:Hoffman's-duality}.
Since the equality \eqref{eq:Hoffman-dual-2} implies
\eqref{eq:Hoffman-dual-1} for $\mathcal{F}=\mathcal{S}$,
it suffices to show \eqref{eq:Hoffman-dual-2}.
To obtain it from Theorem \ref{thm:Hoffman's-duality},
we need to show that $(-1)^{m-1}q_{m}(t)^{m(m+1)/2} \to \exp{(\pi i t)}$ as $m \to \infty$.

Set $(-1)^{m-1}q_{m}(t)^{m(m+1)/2}=1+\sum_{l\ge 1}\alpha_{l}(m)t^{l}$.
{}From Proposition \ref{prop:expansion} and Lemma \ref{lem:SNL},
we see that
\begin{align*}
\alpha_{l}(m)&=\frac{1}{l}
\sum_{\substack{j_{1}+j_{2}+j_{3}=l-1 \\ j_{1}, j_{2}, j_{3} \ge 0}}
B_{l, j_{1}}\frac{(2\pi i)^{l-j_{2}}}{m^{l}}\binom{l}{j_{2}}(1+O(m^{-1}))
\frac{1}{j_{3}! \, m^{j_{3}+1}}\left(\frac{m(m+1)}{2}\right)^{j_{3}+1} \\
&=\frac{1}{l}
\sum_{\substack{j_{1}+j_{2}+j_{3}=l-1 \\ j_{1}, j_{2}, j_{3} \ge 0}}
B_{l, j_{1}}
\frac{2^{j_{1}}(\pi i)^{l-j_{2}}}{j_{3}!}
\frac{(m+1)^{j_{3}+1}}{m^{l}}
\binom{l}{j_{2}}(1+O(m^{-1})) \\
&=\frac{(\pi i)^{l}}{l!}(1+O(m^{-1})).
\end{align*}
Thus we find that
$\displaystyle\lim_{m \to \infty}(-1)^{m}q_{m}(t)^{m(m+1)/2}=\exp{(\pi i t)}$.
\end{proof}

%%%%%%%%%%%%%%%%%%%%%%%%%%%%%%%%%%
\subsection{Proof of cyclic sum formulas}

\begin{proof}[Proof of Theorem \ref{thm:cyclic-sum}]
Our proof is adapted from \cite[\S5]{Bradley05},
\cite[\S2]{OhnoOkuda07} and \cite[\S5]{Kawasaki}.
We first show \eqref{eq:cyc_sum_q}.
Set
\begin{align*}
&
S_{p}(k_{1}, \ldots , k_{d}; l)=\sum_{p>m_{1}>\cdots >m_{d+1}\ge 1}
\frac{q^{(k_{1}-l-1)m_{1}}}{[m_{1}]^{k_{1}-l}}
\prod_{j=2}^{d}\frac{q^{(k_{j}-1)m_j}}{[m_{j}]^{k_{j}}}\,
\frac{q^{(l-1)m_{d+1}}}{[m_{d+1}]^{l}}
\frac{q^{m_{1}}}{[m_{1}-m_{d+1}]}, \\
&
T_{p}(k_{1}, \ldots , k_{d})=\sum_{p>m_{1}>\cdots >m_{d}>m_{d+1}\ge 0}
\prod_{j=1}^{d}\frac{q^{(k_{j}-1)m_j}}{[m_{j}]^{k_{j}}}\,
\frac{q^{m_{1}-m_{d+1}}}{[m_{1}-m_{d+1}]}.
\end{align*}
For $0 \le s \le k_{1}-2$, using
\begin{align}
\frac{[m_{d+1}]}{[m_{1}]}\frac{q^{m_{1}-m_{d+1}}}{[m_{1}-m_{d+1}]}=
\frac{1}{[m_{1}-m_{d+1}]}-\frac{1}{[m_{1}]},
\label{eq:decomp1}
\end{align}
we obtain
\begin{align*}
S_{p}(k_{1}, \ldots , k_{d}; s)=S_{p}(k_{1}, \ldots , k_{d}; s+1)-
H_{p-1}(k_{1}-s, k_{2}, \ldots , k_{d}, s+1;q).
\end{align*}
Hence
\begin{align*}
S_{p}(k_{1}, \ldots , k_{d}; 0)=S_{p}(k_{1}, \ldots , k_{d}; k_{1}-1)-
\sum_{s=0}^{k_{1}-2}
H_{p-1}(k_{1}-s, k_{2}, \ldots , k_{d}, s+1;q).
\end{align*}
Note that this is true also in the case of $k_{1}=1$.
On the other hand, decomposing $m_{d+1}\ge0$ in the summation of $T_p(k_1,\ldots,k_d)$ into either $m_{d+1}>0$ and $m_{d+1}=0$, we get
\begin{align*}
S_{p}(k_{1}, \ldots , k_{d}; 0)=
T_{p}(k_{1}, \ldots , k_{d})-H_{p-1}(k_{1}+1, k_{2}, \ldots , k_{d}).
\end{align*}
Using
\begin{align}
\frac{[m_{d+1}]}{[m_{1}]}\frac{q^{m_{1}-m_{d+1}}}{[m_{1}-m_{d+1}]}=
\frac{q^{m_{1}-m_{d+1}}}{[m_{1}-m_{d+1}]}-\frac{q^{m_{1}}}{[m_{1}]},
\label{eq:decomp2}
\end{align}
we see that
\begin{align*}
\sum_{m_{1}=m_{2}+1}^{p-1}
\frac{[m_{d+1}]}{[m_{1}]}\frac{q^{m_{1}-m_{d+1}}}{[m_{1}-m_{d+1}]}=
\sum_{n=0}^{m_{d+1}-1}\frac{q^{m_{2}-n}}{[m_{2}-n]}-
\sum_{n=1}^{m_{d+1}}\frac{q^{p-n}}{[p-n]}.
\end{align*}
Hence
\begin{align*}
&S_{p}(k_{1}, \ldots , k_{d}; k_{1}-1)\\
&=\sum_{p>m_1>\cdots>m_{d+1}>0} \prod_{j=2}^{d} \frac{q^{(k_j-1)m_j}}{[m_j]^{k_j}} \frac{q^{(k_1-1)m_{d+1}}}{[m_{d+1}]^{k_{1}}} \frac{[m_{d+1}]}{[m_{1}]}\frac{q^{m_{1}-m_{d+1}}}{[m_{1}-m_{d+1}]}\\
&=\sum_{p>m_2>\cdots>m_{d+1}>0} \prod_{j=2}^{d} \frac{q^{(k_j-1)m_j}}{[m_j]^{k_j}} \frac{q^{(k_1-1)m_{d+1}}}{[m_{d+1}]^{k_{1}}} \left(\sum_{n=0}^{m_{d+1}-1}
\frac{q^{m_{2}-n}}{[m_{2}-n]}-\sum_{n=1}^{m_{d+1}}\frac{q^{p-n}}{[p-n]}\right)\\
&=T_p(k_2,\ldots,k_d,k_1)-V_{p}(k_{2}, \ldots , k_{d}, k_{1}),
\end{align*}
where we put
\begin{align*}
V_{p}(k_{2}, \ldots , k_{d}, k_{1})=\sum_{p>m_{2}>\cdots >m_{d+1} \ge m_{1} \ge 1}
\prod_{j=2}^{d}\frac{q^{(k_{j}-1)m_{j}}}{[m_{j}]^{k_{j}}}\,
\frac{q^{(k_{1}-1)m_{r+1}}}{[m_{r+1}]^{k_{1}}}
\frac{q^{p-m_{1}}}{[p-m_{1}]}.
\end{align*}
As a consequence we get
\begin{align*}
&
T_{p}(k_{1}, \ldots , k_{d})-T_{p}(k_{2}, \ldots , k_{d}, k_{1}) \\
&=H_{p-1}(k_{1}+1, k_{2}, \ldots , k_{d};q)-\sum_{s=0}^{k_{1}-2}
H_{p-1}(k_{1}-s, k_{2}, \ldots , k_{d},s+1;q) \\
&-V_{p}(k_{2}, \ldots , k_{d}, k_{1}).
\end{align*}
Since $\sum_{\kk\in\alpha} \big(T_{p}(\kk)-T_{p}(\kk^1, \kk_{1}) \big)=0$, it holds that
\begin{align*}
\sum_{\kk \in \alpha}\left\{
\sum_{s=0}^{\kk_{1}-2}
H_{p-1}(\kk_{1}-s, \kk^1, s+1;q)-
H_{p-1}(\kk_{1}+1, \kk^1;q)
\right\}=-\sum_{\kk \in \alpha}V_{p}(\kk^1, \kk_{1}).
\end{align*}
Using \eqref{eq:partial_franction_q} and \eqref{eq:[p-m]-inv},
we find that
\begin{align*}
&
V_{p}(\kk^1,\kk_{1}) =-\sum_{l=0}^{n-1}((q^{p})^{-1}[p])^{l}\\
&\times
\bigg\{
H_{p-1}(\kk^1, \kk_{1},l+1;q)+H_{p-1}(\kk^1, \kk_{1}+ l+1;q)+ (1-q)H_{p-1}(\kk^1, \kk_{1}+l;q)
\bigg\}
\end{align*}
in $Z_{p, n}$.
This completes the proof of \eqref{eq:cyc_sum_q}.

For the proof of \eqref{eq:cyc_sum_q_star}, set
\begin{align*}
&
B_{p}(k_{1} \ldots , k_{d}; l)=
\sum_{\substack{p>m_{1} \ge \cdots \ge m_{d+1}>0 \\ m_{1}\not=m_{d+1}}}
\frac{q^{(k_{1}-l-1)m_{1}}}{[m_{1}]^{k_{1}-l}}
\prod_{j=2}^{d}\frac{q^{(k_{j}-1)m_j}}{[m_{j}]^{k_{j}}}\,
\frac{q^{(l-1)m_{d+1}}}{[m_{d+1}]^{l}}
\frac{q^{m_{1}}}{[m_{1}-m_{d+1}]}, \\
&
C_{p}(k_{1} \ldots , k_{d})=
\sum_{\substack{p>m_{1} \ge \cdots \ge m_{d+1}>0 \\ m_{1}\not=m_{d+1}}}
\prod_{j=1}^{d}\frac{q^{(k_{j}-1)m_j}}{[m_{j}]^{k_{j}}}\,
\frac{q^{m_{1}-m_{d+1}}}{[m_{1}-m_{d+1}]}.
\end{align*}
It follows from \eqref{eq:decomp1} that
\begin{align*}
&
B_{p}(k_{1} \ldots , k_{d}; s) \\
&=B_{p}(k_{1} \ldots , k_{d}; s+1)-H_{p-1}^{\star}(k_{1}-s, k_{2}, \ldots , k_{d},s+1;q)+
\sum_{p>m>0}\frac{q^{(k-d)m}}{[m]^{k+1}}
\end{align*}
for $0 \le s \le k_{1}-2$, where $k=\sum_{j=1}^{d}k_{j}$.
Hence
\begin{align*}
B_{p}(k_{1} \ldots , k_{d}; 0)
&=B_{p}(k_{1} \ldots , k_{d}; k_{1}-1)-\sum_{l=0}^{k_{1}-2}
H_{p-1}^{\star}(k_{1}-l, k_{2}, \ldots , k_{d}, l+1;q) \\
&+(k_{1}-1)\sum_{p>m>0}\frac{q^{(k-d)m}}{[m]^{k+1}}.
\end{align*}

We have
\begin{align*}
B_{p}(k_{1} \ldots , k_{d}; 0)=C_{p}(k_{1}, \ldots , k_{d}).
\end{align*}
On the other hand, by \eqref{eq:decomp2} we see that
\begin{align*}
&
\sum_{m_{1}=m_{2}}^{p-1}
\frac{[m_{d+1}]}{[m_{1}]}\frac{q^{m_{1}-m_{d+1}}}{[m_{1}-m_{d+1}]}=
\sum_{n=1}^{m_{d+1}}\frac{q^{m_{2}-n}}{[m_{2}-n]}-
\sum_{n=1}^{m_{d+1}}\frac{q^{p-n}}{[p-n]}, \\
&
\sum_{m_{1}=m_{2}+1}^{p-1}
\frac{[m_{2}]}{[m_{1}]}\frac{q^{m_{1}-m_{2}}}{[m_{1}-m_{2}]}=
\sum_{n=0}^{m_{2}-1}\frac{q^{m_{2}-n}}{[m_{2}-n]}-
\sum_{n=1}^{m_{2}}\frac{q^{p-n}}{[p-n]}.
\end{align*}
They imply that
\begin{align*}
&B_{p}(k_{1}, \ldots , k_{d}; k_{1}-1)\\
&=\sum_{\substack{p>m_1\ge \cdots\ge m_{d+1}>0\\m_1\neq m_{d+1}}} \prod_{j=2}^{d} \frac{q^{(k_j-1)m_j}}{[m_j]^{k_j}} \frac{q^{(k_1-1)m_{d+1}}}{[m_{d+1}]^{k_{1}}} \frac{[m_{d+1}]}{[m_{1}]}\frac{q^{m_{1}-m_{d+1}}}{[m_{1}-m_{d+1}]}\\
&=\sum_{\substack{p>m_2\ge \cdots\ge m_{d+1}>0\\m_2\neq m_{d+1}}} \prod_{j=2}^{d} \frac{q^{(k_j-1)m_j}}{[m_j]^{k_j}} \frac{q^{(k_1-1)m_{d+1}}}{[m_{d+1}]^{k_{1}}} \left(
\sum_{n=1}^{m_{d+1}}\frac{q^{m_{2}-n}}{[m_{2}-n]}-
\sum_{n=1}^{m_{d+1}}\frac{q^{p-n}}{[p-n]}\right)\\
&+ \sum_{p-1>m_2=\cdots= m_{d+1}>0} \prod_{j=2}^{d} \frac{q^{(k_j-1)m_j}}{[m_j]^{k_j}} \frac{q^{(k_1-1)m_{d+1}}}{[m_{d+1}]^{k_{1}}} \left(
\sum_{n=0}^{m_{2}-1}\frac{q^{m_{2}-n}}{[m_{2}-n]}-
\sum_{n=1}^{m_{2}}\frac{q^{p-n}}{[p-n]}\right) \\
&=C_{p}(k_{2}, \ldots , k_{d}, k_{1})+\sum_{p>m>0}\frac{q^{(k+1-d)m}}{[m]^{k+1}}-
W_{p}(k_{2}, \ldots , k_{d}, k_{1}),
\end{align*}
where
\begin{align*}
W_{p}(k_{2}, \ldots , k_{d}, k_{1})=
\sum_{p>m_{2} \ge \cdots \ge m_{d+1} \ge m_{1}>0}
\prod_{j=2}^{d}\frac{q^{(k_{j}-1)m_{j}}}{[m_{j}]^{k_{j}}}\,
\frac{q^{(k_{1}-1)m_{d+1}}}{[m_{d+1}]^{k_{1}}}
\frac{q^{p-m_{1}}}{[p-m_{1}]}.
\end{align*}
Thus we get
\begin{align*}
&
C_{p}(k_{1}, \ldots , k_{d})-C_{p}(k_{2}, \ldots , k_{d}, k_{1}) \\
&=-\sum_{s=0}^{k_{1}-2}H_{p-1}^{\star}(k_{1}-s, k_{2}, \ldots , k_{d}, s+1;q)-
W_{p}(k_{2}, \ldots , k_{d}, k_{1})\\
&+(q-1)\sum_{p>m>0}\frac{q^{(k-d)m}}{[m]^{k}}+k_1 \sum_{p>m>0}\frac{q^{(k-d)m}}{[m]^{k+1}}
\end{align*}
Therefore
\begin{align*}
&\sum_{\kk \in \alpha}
\sum_{s=0}^{\kk_{1}-2}H_{p-1}^{\star}(\kk_{1}-s, \kk^1, s+1;q)\\
&=-\sum_{\kk \in \alpha}
W_{p}(\kk^1, \kk_{1})+|\alpha|(q-1)\sum_{p>m>0}\frac{q^{(k-d)m}}{[m]^{k}}+\frac{k}{d}|\alpha| \sum_{p>m>0}\frac{q^{(k-d)m}}{[m]^{k+1}}.
\end{align*}
One computes
\begin{align*}
\frac{q^{(k-d)m}}{[m]^k}=\frac{q^{(k-d)m}}{[m]^k}((1-q^{m})+q^{m})^{d-1}=
\sum_{l=0}^{d-1}\binom{d-1}{l}(1-q)^l \frac{q^{(k-l-1)m}}{[m]^{k-l}}.
\end{align*}
Likewise we see that
\begin{align*}
\frac{q^{(k-d)m}}{[m]^{k+1}}=\sum_{l=0}^{d-1}\binom{d}{l}(1-q)^l \frac{q^{(k-l)m}}{[m]^{k-l+1}}.
\end{align*}
Now the desired equality is obtained by applying the expansion \eqref{eq:[p-m]-inv} to
$W_{p}(\kk^1, \kk_{1})$.
\end{proof}

\begin{proof}[Proof of Corollary \ref{cor:cyclic_sum}]
Since $\phi_{\widehat{\mathcal{F}}}(q^{\pm \boldsymbol{p}}) =1$, applying Theorems \ref{thm:hatA_from_q} and \ref{thm:hatS_from_q}
%Remark \ref{rem:complex_conj_hat_A} and \eqref{eq:power_of_qmt_lim}
to Theorem \ref{thm:cyclic-sum}, we get the desired formulas.
\end{proof}

%%%%%%%%%%%%%%%%%%%%%%%%%%%%%%%%%%%%%%%%%%


\begin{thebibliography}{99}

\bibitem[AHY]{AHY} K.~Akagi, M.~Hirose, S.~Yasuda, {\itshape Integrality of $p$-adic multiple zeta values and application to finite multiple zeta values}, preprint.


\bibitem[And99]{Andrews99} G.E.~Andrews, {\itshape $q$-analogs of the binomial coefficient congruences of Babage, Wolstenholme and Glaisher}, Discrete Math.~{\bf 204} (1999), 15--25.

\bibitem[AIK]{ArakawaIbukiyamaKaneko} T.~Arakawa, T.~Ibukiyama, M.~Kaneko, {\itshape Bernoulli Numbers and Zeta Functions}, Springer, Tokyo, 2014.

\bibitem[BTT18]{BTT18} H.~Bachmann, Y.~Takeyama, K.~Tasaka, {\itshape Cyclotomic analogues of finite multiple zeta values}, Compositio Math. {\bf 154}(12) (2018), 2701--2721.

%\bibitem[BTT19]{BTT19} H.~Bachmann, Y.~Takeyama, K.~Tasaka, {\itshape Special values of finite multiple harmonic q-series at roots of unity}, IRMA Lectures in Mathematics and Theoretical Physics 31,
%Algebraic Combinatorics, Resurgence, Moulds and Applications (CARMA) (2020), Vol. 2,  1--18.
%IRMA Lectures in Mathematics and Theoretical Physics (EMS), CARMA Volume.

%\bibitem{BTT20} H.~Bachmann, Y.~Takeyama, K.~Tasaka, {\itshape Finite and symmetric Mordell-Tornheim multiple zeta values}, to appear in J. Math. Soc. Japan.

\bibitem[Bra05-1]{Bradley05} D.M.~Bradley, {\itshape Multiple $q$-zeta values}, J.~Algebra {\bf 283} (2005), no. 2, 752--798.

\bibitem[Bra05-2]{Bradley05_2} D.M.~Bradley, {\itshape Duality for finite multiple harmonic $q$-series}, Discrete Math. {\bf 300} (2005), no. 1-3, 44--56.

\bibitem[Bro12]{Brown12} F.~Brown, {\itshape Mixed Tate motives over $\Z$}, Ann.~of Math.~{\bf 175} (2012), no.~2, 949--976.

%\bibitem[Cha17]{Chatzistamatiou17} A.~Chatzistamatiou, {\itshape On integrality of $p$-adic iterated integrals}, Journal of Algebra {\bf 474} (2017), 240--270.

%\bibitem[CM21]{ChangMishiba21} C-Y.~Chang, Y.~Mishiba, {\itshape On a conjecture of Furusho over function fields}, Invent.~Math., {\bf 223}, no. 1, (2021), 49--102.

%\bibitem[DG05]{DeligneGoncharov05} P.~Deligne, A.~B.~Goncharov, {\itshape Groupes fondamentaux motiviques de Tate mixte}, Ann. Sci. \'{E}cole Norm. Sup. 38 (2005), 1--56.

%\bibitem{Dilcher08} K.~Dilcher, {\itshape Determinant expressions for q-harmonic congruences and degenerate Bernoulli numbers}, Electron.~J.~Combin.~{\bf 15}(1) (2008), R63.

%\bibitem[Dri90]{Drinfeld90} V.G.~Drinfeld, {\itshape On quasitriangular quasi-Hopf algebras and on a group that is closely connected with ${\rm Gal}(\overline{\Q}/\Q)$}, Algebra i Analiz 2 (1990), 149--181.

%\bibitem[EFMS16]{Ebrahimi-FardManchonSinger16} K.~Ebrahimi-Fard, D.~Manchon, J.~Singer, {\itshape Renormalisation of $q$-regularised multiple zeta values}, Lett.~Math.~Phys. {\bf 106} (2016), 365--380.

\bibitem[Eul76]{Euler76} L.~Euler, {\itshape Meditationes circa singulare serierum genus}, Novi Comm. Acad. Sci. Petropol. {\bf 20} (1776), 140--186, reprinted in Opera Omnia ser. I, vol. 15, B. G. Teubner, Berlin (1927) 217--267.

%\bibitem[Fur04]{Furusho04} H.~Furusho, {\itshape $p$-adic multiple zeta values. I. $p$-adic multiple polylogarithms and the $p$-adic KZ equation}, Invent.~Math., {\bf 155}, no. 2, (2004), 253--286.

\bibitem[Fur07]{Furusho07} H.~Furusho, {\itshape $p$-adic multiple zeta values. II. Tannakian interpretations}, Amer.~J.~Math., {\bf 129}, no. 4, (2007), 1105--1144.

%\bibitem[FJ07]{FurushoJafari07} H.~Furusho, A.~Jafari, {\itshape Regularization and generalized double shuffle relations for $p$-adic multiple zeta values}, Compositio Math. Vol {\bf 143}, (2007), 1089--1107.

\bibitem[Gon01]{Goncharov01} A.~Goncharov, \textit{Multiple polylogarithms and mixed Tate motives}, arXiv:math/0103059.

\bibitem[Gra97]{Granville97} A.~Granville, {\itshape A decomposition of Riemann's zeta-function}, in London Math. Soc. Lecture Note Ser. {\bf 247}, Cambridge, (1997), 95--101.

%\bibitem[GZ19]{GuoZudilin19} V.J.W.~Guo, W.~Zudilin, {\itshape A q-microscope for supercongruences}, Adv. in Math. {\bf 346} (2019) 329--358.

\bibitem[Hir20]{Hirose} M.~Hirose, {\itshape Double shuffle relations for refined symmetric multiple zeta values}, Doc. Math. {\bf 25} (2020), 365--380

\bibitem[HMO]{HiroseMuraharaOno} M.~Hirose, H.~Murahara, M.~Ono, {\itshape On variants of symmetric multiple zeta-star values and the cyclic sum formula}, to appear in Ramanujan J., arXiv:2001.03832.

\bibitem[HHT17]{HPT17} Kh.~Hessami Pilehrood, T.~Hessami Pilehrood, R.~Tauraso, {\itshape Some $q$-congruences for homogeneous and quasi-homogeneous multiple $q$-harmonic sums}, Ramanujan J.~{\bf 43} (2017), 113--139.

%\bibitem[Hof97]{Hoffman97} M.E.~Hoffman, {\itshape The algebra of multiple harmonic series}, J.~Algebra {\bf 194} (1997), no. 2, 477--495.

\bibitem[HO03]{HoffmanOhno03} M.~Hoffman, Y.~Ohno, {\itshape Relations of multiple zeta values and their algebraic expression}, J.~Algebra {\bf 262} (2003), no. 2, 332--347.

\bibitem[Hof15]{Hoffman15} M.E.~Hoffman, {\itshape Quasi-symmetric functions and mod $p$ multiple harmonic sums}, Kyushu J. Math. {\bf 69} (2015), 345--366.

%\bibitem{IharaKajikawaOhnoOkuda11} K.~Ihara, J.~Kajikawa, Y.~Ohno and J.~Okuda, {\itshape Multiple zeta values vs. multiple zeta-star values}, J.~Algebra {\bf 332} (2011), 187--208

\bibitem[IKZ06]{IharaKanekoZagier06} K.~Ihara, M.~Kaneko and  D.~Zagier,
{\itshape Derivation and double shuffle relations for multiple zeta values}, Compositio Math. {\bf 142} (2006), 307--338.

%\bibitem[Jar14]{Jarossay14} D.~Jarossay, {\itshape Double m\'{e}lange des multiz\^{e}tas finis et multiz\^{e}tas sym\'{e}tris\'{e}s}, Comptes rendus - Math\'{e}matique {\bf 352} (2014), 767--771.

%\bibitem[Jar1]{Jarossay16} D.~Jarossay, {\itshape Algebraic relations, Taylor coefficients of hyperlogarithms and images by Frobenius - II : Relations with other motives and the Taylor period map}, preprint (2016), arXiv:1601.01158v1.

%\bibitem[Jar2]{Jarossay16_2} D.~Jarossay, {\itshape Algebraic relations, Taylor coefficients of hyperlogarithms and images by Frobenius - I: The prime multiple harmonic sum motive}, preprint (2016), arXiv:1412.5099v2.

%\bibitem[Jar3]{Jarossay17} D.~Jarossay, {\itshape An explicit theory of $\pi_1^{\rm un,crys}(\mathbb{P}^1 - \{0, \mu, \infty\}$)-II-1: Standard algebraic equations of prime weighted multiple harmonic sums and adjoint multiple zeta values}, preprint (2017), arXiv:1412.5099v3.

\bibitem[Jar1]{Jarossay18} D.~Jarossay, {\itshape Pro-unipotent harmonic actions and a computation of $p$-adic cyclotomic multiple zeta values}, preprint (2018), arXiv:1501.04893v5


\bibitem[Jar2]{Jarossay19} D.~Jarossay, {\itshape Adjoint cyclotomic multiple zeta values and cyclotomic multiple harmonic values}, preprint (2019), arXiv:1412.5099v5.

%\bibitem{Kamano16} K.~Kamano, {\itshape Finite Mordell-Tornheim multiple zeta values}, Funct.~Approx.~Comment.~Math. {\bf 54} (2016), 65--72.

\bibitem[KZ]{KanekoZagier} M.~Kaneko and D.~Zagier, {\itshape Finite multiple zeta values}, in preparation.

\bibitem[Kan19]{Kaneko19} M.~Kaneko, {\itshape  An introduction to classical and finite multiple zeta values}, Publications Math\'{e}matiques de Besan\c{c}on, no.1 (2019), 103--129.

\bibitem[Kaw19]{Kawasaki} N.~Kawasaki, {\itshape Hyperlogarithms, Bernoulli polynomials, and related multiple zeta values}, Dissertation,
Tohoku University (2019).

\bibitem[KO20]{KawasakiOyama} N.~Kawasaki, K.~Oyama, {\itshape Cyclic sums of finite multiple zeta values}, Acta Arith. {\bf 195} (2020), 281--288.

\bibitem[Kaw10]{Kawashima10} G.~Kawashima, {\itshape A generalization of the duality for finite multiple harmonic q-series}, Ramanujan J. {\bf 21} (2010), 335--347.

\bibitem[Kon09]{Kontsevich09} M.~Kontsevich, {\itshape Holonomic D-modules and positive characteristic}, Jpn. J. Math. {\bf 4} (2009), no. 1, 1--25.

%\bibitem{Murahara16} H.~Murahara, \textit{A Note on Finite Real Multiple Zeta Values}, Kyushu J. Math. {\bf 70} (2016), no. 1, 197--204.

\bibitem[MOS20]{MuraharaOnozukaSeki20} H.~Murahara, T.~Onozuka, S.~Seki, {\itshape Bowman-Bradley type theorem for finite multiple zeta values in $\mathcal{A}_2$}, Osaka J., Math., {\bf 57} (2020), no.~3, 647--653.

%\bibitem[Ohn99]{Ohno99} Y. Ohno, {\itshape A generalization of the duality and sum formulas on the multiple zeta values}, J. Number Theory 74 (1999), no. 1, 39--43.

\bibitem[OO07]{OhnoOkuda07} Y.~Ohno and J.~Okuda, {\itshape On the sum formula for the $q$-analogue of non-strict multiple zeta values}, Proc. Amer. Math. Soc. {\bf 135} (2007), no. 10, 3029--3037.

%\bibitem{OkudaTakeyama07} J.~Okuda and Y.~Takeyama, \textit{On relations for the multiple $q$-zeta values}, Ramanujan J. {\bf14} (2007), no. 3, 379--387.

\bibitem[OSY21]{OnoSekiYamamoto} M.~Ono, S.~Seki, S.~Yamamoto, {\itshape Truncated $t$-adic symmetric multiple zeta values and double shuffle relations}, Research in Number Theory {\bf 7}, Article number 15 (2021).

\bibitem[OSS]{OnoSakuradaSeki} M.~Ono, K.~Sakurada, S.~Seki, {\itshape A note on $\mathcal{F}_n$-multiple zeta values}, to appear in Commentarii mathematici Universitatis Sancti Pauli.

%\bibitem{Pan08} H.~Pan, {\itshape A Generalization of Wolstenholme's Harmonic Series Congruence}, Rocky Mountain J.~Math., {\bf 38}(4) (2008), 1263--1269.

\bibitem[PARI]{PARI} The PARI Group, {\itshape PARI/GP version 2.11.0}, Univ. Bordeaux, 2018, http://pari. math.u- bordeaux.fr/.

\bibitem[Rac02]{Racinet02} G.~Racinet, {\itshape Doubles m\'{e}langes des polylogarithmes multiples aux racines de l\'{u}nit\'{e}}, Publ. Math. IHES 95 (2002), 185--231.

%\bibitem{Reu} C.~Reutenauer, {\itshape Free Lie algebras} (Oxford Science Publications, Oxford, 1993).

\bibitem[Ros13]{RosenPHD} J.~Rosen, {\itshape The arithmetic of multiple harmonic sums}, Thesis (2013).

\bibitem[Ros15]{Rosen15} J.~Rosen, {\itshape Asymptotic relations for truncated multiple zeta values}, J. Lond. Math. Soc. (2) {\bf 91} (2015), 554--572.

\bibitem[Ros19]{Rosen19} J.~Rosen, {\itshape The completed finite period map and Galois theory of supercongruences}, Int. Math. Res. Not. {\bf 2019}, no. 23, 7379--7405.

\bibitem[Ros]{Rosen} J.~Rosen, {\itshape Sequential periods of the crystalline Frobenius}, preprint (2018), arXiv:1805.01885.

%\bibitem{SakugawaSeki17} K.~Sakugawa and S.~Seki, {\itshape On functional equations of finite multiple polylogarithms}, J. Algebra {\bf 469} (2017), 323--357.

%\bibitem[Sek17]{SekiPHD} S.~Seki, {\itshape Finite multiple polylogarithms}, Doctoral dissertation in Osaka University, 2017.

%\bibitem[Sch]{Schlesinger} K.-G.~Schlesinger, {\itshape Some remarks on $q$-deformed multiple polylogarithms}, arXiv:math/0111022, 2001.

\bibitem[Sek19]{Seki19} S.~Seki, {\itshape The $\pp$-adic duality for the finite star-multiple polylogarithms}, Tohoku Math. J. {\bf 71} (2019), 111--122.

%\bibitem{SaitoWakabayashi15} S.~Saito and N.~Wakabayashi, {\itshape Sum formula for finite multiple zeta values}, J.~Math.~Soc.~Japan {\bf 67} (2015), 1069--1076.

\bibitem[SP07]{ShiPan07} L.-L.~Shi, H.~Pan,
{\itshape A $q$-analogue of Wolstenholme's harmonic series congruence}, Amer. Math. Monthly,
{\bf 114}(6) (2007), 529--531.

%\bibitem[Sin15]{Singer15} J.~Singer, {\itshape On $q$-analogues of multiple zeta values}, Funct.~Approx.~Comment.~Math. {\bf 53}(1) (2015), 135--165.

%\bibitem{Takeyama09} Y.~Takeyama, \textit{A $q$-analogue of non-strict multiple zeta values and basic hypergeometric series}, Proc. Amer. Math. Soc. {\bf 137} (2009), no. 9, 2997--3002.

%\bibitem{Takeyama12} Y.~Takeyama, \textit{Quadratic relations for a $q$-analogue of multiple zeta values}, Ramanujan J.  {\bf 27} (2012), no. 1, 15--28.

%\bibitem{Takeyama13} Y.~Takeyama, {\itshape The algebra of a $q$-analogue of multiple harmonic series}, SIGMA {\bf 9} (2013), 061, 1--15.

\bibitem[Tas21]{Tasaka21} K.~Tasaka, {\itshape Finite and symmetric colored multiple zeta values and multiple harmonic q-series at roots of unity}, Selecta Math., 27, Article number:21 (2021).


\bibitem[Ter02]{Terasoma02} T.~Terasoma, \textit{Mixed Tate motives and multiple zeta values}, Invent. Math. 149 (2002), no. 2, 339--369.

%\bibitem[Unv13]{Unver13} S.~\"{U}nver, {\itshape Drinfel’d--Ihararelations for $p$-adic multi-zeta values}, J.Number Theory, {\bf 133}, no. 5, (2013) 1435--1483.

%\bibitem[Van96]{Hamme96} L.~Van Hamme, {\itshape Some conjectures concerning partial sums of generalized hypergeometric series}, $p$-adic functional analysis (Nijmegen, 1996) Lecture Notes in Pure and Appl. Math., vol. 192, Dekker, New York, 1997, pp. 223--236.

%\bibitem{Washington97} L.~Washington, {\itshape Introduction to Cyclotomic Fields}, GTM 83, Springer-Verlag, New York, 1997.

\bibitem[Wol62]{Wolstenholme62} J.~Wolstenholme, {\itshape On certain properties of prime numbers}, Quart. J. Math. Oxford
Ser. {\bf 5} (1862), 35--39.


\bibitem[Yas16]{Yasuda16} S.~Yasuda, {\itshape Finite real multiple zeta values generate the whole space Z}, Int. J. Number Theory {\bf 12} (2016), no. 3, 787--812.

\bibitem[Yas19]{Yasuda19} S.~Yasuda, {\itshape From $p$-adic multiple zeta values to finite multiple zeta values}, In: K.~Sakugawa, K.~Tasaka, Y.~Mishiba (eds.) Proceedings of 26th summer school on Number Theory "Multiple Zeta Value" (2018), 189--202.

\bibitem[Zag94]{Zagier94} D.~Zagier, {\itshape Values of zeta functions and their applications}, First European Congress of Mathematics, Vol. II (Paris, 1992), Progr.\ Math., {\bf 120}, Birkh\"{a}user, Basel (1994), 497--512.

\bibitem[Zha08]{Zhao08} J.~Zhao, {\itshape Wolstenholme type theorem for multiple harmonic sums}, Int. J. Number Theory, {\bf 4}(1) (2008), 73--106.

\bibitem[Zha13]{Zhao13} J.~Zhao, {\itshape On $q$-analog of Wolstenholme type congruences for multiple harmonic sums}, Integers {\bf 13} (2013), A23.

\bibitem[Zha16]{Zhao16} J.~Zhao, {\itshape Multiple zeta functions, multiple polylogarithms and their special values}, Series on Number Theory and Its Applications, World Scientific, Vol. {\bf 12} (2016).


%\bibitem[Zud03]{Zudilin03} W.~Zudilin, {\itshape Algebraic relations for multiple zeta values}, Uspekhi Mat. Nauk {\bf 58}, 2003.

%\bibitem[Zud19]{Zudilin19} W.~Zudilin, {\itshape Congruences for $q$-binomial coefficients}, Annals of Combinatorics {\bf 23} (2019), no. 3-4, 1123--1135.


\end{thebibliography}
\end{document}